\documentclass[11pt]{amsart}
\usepackage{geometry}                
\usepackage{graphicx}
\usepackage{amssymb}
\usepackage{epstopdf}
\newtheorem{theorem}{Theorem}[section]
\newtheorem{lemma}{Lemma}[section]

\newtheorem{proposition}{Proposition}[section]
\newtheorem{conjecture}{Conjecture}[section]
\newtheorem{definition}{Definition}[section]

\numberwithin{equation}{section}

\def\Z{\mathbb Z}

\def\R{\mathbb R}

\def\om{\omega}
\def\d{\partial}

\def\e{\epsilon}
\def\a{\alpha}
\def\b{\beta}
\def\g{\gamma}

\title[Doodles and blobs on a ruled page]{Doodles and blobs on a ruled page: convex  quasi-envelops of traversing flows on surfaces}

 \author[G.~Katz]{Gabriel Katz}
\address{MIT, Department of Mathematics, 77 Massachusetts Ave., Cambridge, MA 02139, U.S.A.}
\email{gabkatz@gmail.com}

\begin{document}

\maketitle 

\begin{abstract}
Let $A$ denote the cylinder $\mathbb R \times S^1$ or the band $\mathbb R \times I$, where $I$ stands for the closed interval.
We consider $2$-{\sf moderate immersions} of closed curves (``{\sf doodles}") and compact surfaces (``{\sf blobs}") in $A$, up to cobordisms that also are $2$-moderate immersions in $A \times [0, 1]$ of surfaces and solids. By definition, the $2$-moderate immersions of curves and surfaces do not have tangencies of order $\geq 3$ to the fibers of the obvious projections $A \to S^1$,\; $A \times [0, 1] \to S^1 \times [0, 1]$ or $A \to I$,\; $A \times [0, 1] \to I \times [0, 1]$. 
These bordisms come in different flavors: in particular, we consider one flavor based on {\sf regular embeddings} of doodles and blobs in $A$. We compute the bordisms of regular embeddings and construct many  invariants that distinguish between the bordisms of immersions and embeddings. In the case of oriented doodles on $A= \mathbb R \times I$, our computations of $2$-moderate immersion bordisms $\mathbf{OC}^{\mathsf{imm}}_{\mathsf{moderate \leq 2}}(A)$ are near complete: we show that they can be described by an exact sequence of abelian groups 
$$0 \to \mathbf K  \to \mathbf{OC}^{\mathsf{imm}}_{\mathsf{moderate \leq 2}}(A)\big/\mathbf{OC}^{\mathsf{emb}}_{\mathsf{moderate \leq 2}}(A) \stackrel{\mathcal I \rho}{\longrightarrow} \mathbb Z \times \mathbb Z \to 0,$$
where $\mathbf{OC}^{\mathsf{emb}}_{\mathsf{moderate \leq 2}}(A) \approx \mathbb Z \times \mathbb Z$, the epimorphism $\mathcal I \rho$ counts different types of crossings of immersed doodles, and the kernel $\mathbf K$ contains the group $(\mathbb Z)^\infty$ whose generators are described explicitly.  
\end{abstract}

\section{Introduction}

This paper illustrates the richness of {\sf traversing} vector flows (see Definition \ref{def.2.3XXX}) on surfaces with boundary. It also provides tools for constructing such flows (see Fig. \ref{fig.1.4A}). Some multi-dimensional versions of these ideas and constructions can be found in \cite{K6}, \cite{K7}, and \cite{KSW1}. However, the case of, so called, $2$-{\sf moderate}  one- and two-dimensional embeddings and immersions against the background of a fixed $1$-dimensional foliation on a target surface $A$ has its unique and pleasing features. One of which is the drastic simplification of the combinatorial considerations that characterize our treatment in \cite{K6}, \cite{K7}, and \cite{KSW2} of similar multidimensional immersions. \smallskip

Propositions \ref{prop.5.1_XYZ} and \ref{prop.Triple_int_X} below illustrate the flavor of some of our results. 

Consider the vector space $\R^3_{xyz}$ with coordinates $x, y, z$ and the obvious projections $P_z:  \R^3_{xyz} \to \R^2_{xy}$ and $p_z:  \R^2_{xz} \to \R^1_{x}$. 
Let $\mathcal C \subset \R^2_{xz}$ be a smooth simple curve, the boundary of a compact domain $\mathcal D \subset \R^2_{xz}$, such that $p_z:  \mathcal C \to  \R^1_{x}$ has only singularities that are quadratic (folds). 
We say that $\mathcal C$ is {\sf positively (negatively) concave} if the function $x: \mathcal  D \to \R$ has more local maxima (minima) than minima (maxima) 
(see Fig. \ref{fig.1.9_X},  {\sf(b)}). 

\begin{proposition}\label{prop.5.1_XYZ} Let $\mathcal C \subset \R^2_{xz}$ be a simple smooth curve such that the projection $p_z:  \mathcal C \to  \R^1_{x}$ has only quadratic folding singularities and $\mathcal C$  is positively (negatively) concave. 

Then any smooth compact surface $\mathcal S \subset \R^3_{xyz} \cap \{y \geq 0\}$ that bounds $\mathcal C$ must have at least cubic singularities (cusps) under the map $P_z: \mathcal S \to \R^2_{xy}$. \hfill $\diamondsuit$
\end{proposition}

Proposition \ref{prop.5.1_XYZ} is implied directly by Theorem \ref{th1.8}.
The key feature here is the interplay between the curve $\mathcal C$,  surface $\mathcal S$, and the simple $1$-dimensional foliation in $\R^3_{xyz}$ whose leaves are the fibers of the projection $P_z$. \smallskip

A different, but related phenomenon is exemplified by the next proposition, which follows from the proof of Theorem \ref{th.1.DOODLES}. 

\begin{proposition}\label{prop.Triple_int_X} Consider the immersed curve $\mathcal C$ in $\R^2_{xz}$ shown in Figure \ref{fig.4.8888XX}, {\sf(a)}. 

Then any compact immersed surface $\mathcal S \subset \R^3_{xyz} \cap \{y \geq 0\}$, that bounds  $\mathcal C$ must have two triple intersection points at least. 
\hfill $\diamondsuit$
\end{proposition}

Although doodles on surfaces has been a well-traveled destination \cite{CKS}, \cite{BFKK}, doodles against the background of a given foliation on a compact surface are not. The same can be said about submersions $\a: X \to A$ of compact surfaces $X$ on the cylinders or strips $A$, equipped with a product foliation $\mathcal F(\hat v)$.  This simple product foliation is responsible for the term ``{\sf ruled}" in the title of the paper.\smallskip

Our general inspiration comes from the pioneering works of V. I. Arnold \cite{Ar}, \cite{Ar1}, \cite{Ar2}, and V. A. Vassiliev \cite{Va}, \cite{Va1}. \smallskip
\smallskip 

The main problem we are dealing here is to classify such submersions $\a: X \to A$ (up to a kind of cobordism), while \emph{controlling the tangency patterns} of the boundary $\d X$, or rather of the curves $\a(\d X)$ to the foliation $\mathcal F(\hat v)$.  As a byproduct, we getting some computable bordism-like relation among traversing vector fields on compact surfaces with boundary.
\smallskip

\begin{figure}[ht]
\centerline{\includegraphics[height=1.8in,width=5.2in]{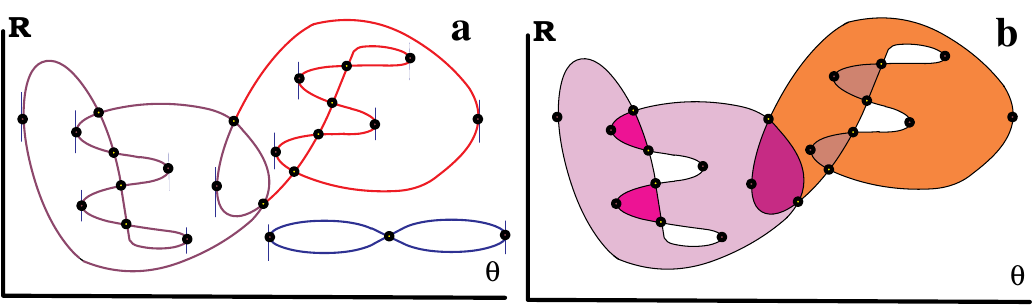}}
\bigskip
\caption{\small{Diagram {\sf(a)} shows {\sf  doodles}---an immersion $\b: \mathcal C \to A$ of $3$ loops $\mathcal C$ in the surface $A = \R \times [0,1]$. Diagram {\sf(b)} shows {\sf blobs}---an immersion $\a: X \to A$ of two disks $X$ in $A$. The self-intersections of the curves $\b(\mathcal C)$ and of $\a(\d X)$ and the points of tangency of $\b(\mathcal C)$ and of $\a(\d X)$ to the vertical foliation $\mathcal F(\hat v)$ on $A$ are marked. Thanks to the presence of figure ``$\infty$", $\b$ does not extend to an immersion $\a$ of any compact surface $X$ into $A$.}}
\label{fig.1.4B}
\end{figure}

Thus, we consider two target spaces, the cylinder $\R \times S^1$ and and the strip $\R \times [0, 1]$.  Some of the constructions will work for the cylinder, and some for the strip; for both, we use the same notation ``$A$".  The space $A$ is equipped with a traversing (see Definition \ref{def.2.3XXX} and \cite{K1}, \cite{K2}) vector field $\hat v$ and the $1$-dimensional oriented foliation $\mathcal F(\hat v)$ it generates. Its leaves are of the form $\{\R \times \theta\}_\theta$, where $\theta$ belongs either to the circle $S^1$ or to the interval $I = [0,1]$.

We draw some ``{\sf doodles}" (finite collections of loops) on $A$ (as in Fig. \ref{fig.1.4B}, diagram {\sf(a)}) and pay close attention to the way they intersect the leaves of the foliation $\mathcal F(\hat v)$, especially to the way they are tangent to the leaves. These interactions of doodles with the leaves are combinatorial in nature. We will impose some a priori restrictions (called ``$2$-{\sf moderate}") on these combinatorial patterns and will  classify the doodles that respect the restrictions. 

In the paper, we will also consider doodles that are the images of boundaries of compact surfaces $X$, \emph{immersed} in $A$ (as shown in Fig. \ref{fig.1.4B}, {\sf(b)}). The images of $X$ in $A$ form overlapping  ``{\sf blobs}".\smallskip

Our main results about doodles and blobs on a ruled page $(A,\, \mathcal F(\hat v))$ are contained in Theorems \ref{th.1.7}-\ref{th1.8} and Theorems \ref{th.1.DOODLES}-\ref{th.1_BLOBS_KERNEL_ORIENT}. Perhaps, some new invariants of doodles and blobs that lead to these theorems will have an independent life. \smallskip

Let us set the stage for these results in a more formal way. Let $X$ be a compact surface with boundary and $v$ a {\sf traversing} vector field (see Definition \ref{def.2.3XXX}) on $X$. As a function of a point $x \in X$, the $v$-trajectory $\g_x \subset X$ through $x$ exhibits a discontinuous behavior in the vicinity of any ``concave" point (see Definition \ref{def.1.2}) of the boundary 
\cite{K1}. In order to get around this fundamental difficulty, we ``envelop" the pair $(X, v)$ into a pair $(\hat X, \hat v)$, where an ambient compact surface $\hat X \supset X$ with a traversing vector field $\hat v$, is such that: 

{\sf (1)} $\d \hat X$ is convex (see Definition \ref{def.1.2}) with respect to the $\hat v$-flow on $\hat X$,  

{\sf (2)} $\hat v|_X = v$.  \smallskip


Without lost of generality, the reader may think of $\hat X$ as residing in the cylinder $\R \times S^1$ or in the strip $\R \times [0, 1]$ and the vector field $\hat v$ as being the unit vector field, tangent to the leaves of the product foliation $\mathcal F(\hat v)$.
\smallskip

Not any pair $(X, v)$ admits such {\sf convex envelop} (see Lemma \ref{lem.1.4}). However, when available, the convex envelop $(\hat X, \hat v)$ simplifies the analysis of $(X, v)$ greatly. \smallskip

In this context, our goal is to study convex envelops (and their generalizations, the so called, {\sf quasi-envelops}, shown in Fig. \ref{fig.1.4A}) together with the doodles or blobs they contain, up to some natural equivalence relations that we call in \cite{K6}, \cite{K7}  ``{\sf quasitopies}", a crossover between bordisms and pseudoisotopies of immersions. 
When $A = \R \times [0, 1]$, the quasitopy (bordism) equivalence classes can be organized into  \emph{groups}. In \cite{K6}, we compute these algebraic structures for an a priori prescribed set of combinatorial tangency patterns of ``$n$-dimensional doodles" to product-like $1$-dimensional foliations. For $(n+1)$-dimensional traversing convexly enveloped flows, this goal is achieved in full generality in \cite{K7}. However, the case of quasi-enveloped traversing flows is far from being settled. 
Although in two dimensions these structures are relatively primitive, as this paper illustrates, they are not completely trivial ether. \smallskip

Recall that in the study of manifolds and fibrations the universal classifying spaces like Grassmanians play a pivotal role. In the category of convex envelops, the role of universal objects (of ``the new Grassmanians") is played by various spaces of smooth functions $f: \R \to \R$ whose zeros (considered with their multiplicities) are modeled after the real divisors of real polynomials. The topology of these functional spaces with {\sf constrained zero divisors} is interesting on its own right (see \cite{KSW1}, \cite{KSW2}, where it is described in detail). One particular kind of these functional spaces, called spaces of smooth functions/polynomials with $k$-{\sf moderate singularities}, has been introduced and  studied in depth by V. I. Arnold \cite{Ar}, \cite{Ar1} and V. A. Vassiliev \cite{Va}, \cite{Va1}.  In  \cite{KSW1}, \cite{KSW2}, 
we compute the homology of similar functional spaces, based on real polynomials in one variable,  in terms of appropriate \emph{universal} combinatorics.  This is reminiscent to the role played by the Schubert calculus in depicting the characteristic classes of classical Grassmanians.\smallskip


Recall that, for {\sf boundary generic} $2$-dimensional traversing flows (see Definition \ref{def.1.1}), no tangencies of orders $\geq 3$ to the boundary may occur \cite{K2}. In light of what has been outlined above, we should anticipate a link between boundary  generic  flows on surfaces and the {\sf spaces of smooth functions} $f: \R \to \R$ (or even real polynomials) that have no zeros of multiplicities $\geq 3$. 
 This connection and its generalizations are validated in \cite{K6}, \cite{K7}. 
 \smallskip
 

Let  $\mathcal F$ denote the space of smooths functions $f: \R \to \R$ that are identically $1$ outside of a compact interval (the interval may depend on a particular function). The space $\mathcal F$ is considered in the $C^\infty$-topology. Let $\mathcal F^{< k}$ be its subspace, formed by functions  that have zeros only of the multiplicities less than $k$. Arnold calls such functions ``{\sf functions with $k$-moderate singularities}".  
The property of a function to have $k$-moderate zeros is an \emph{open property} in the $C^\infty$-topology; that is, the space $\mathcal F^{< k}$ is open in $\mathcal F$. 
 \smallskip
 
For $k = 2$, the combinatorial patterns $\om$ of zero divisors of such functions are finite sequences of natural numbers, build of $1$'s and $2$'s only, as well as the empty sequence.\smallskip

A fundamental theorem of Vassiliev (\cite{Va}, Corollary on page 81 and The First and Second Main Theorems on pages 78-79) identifies the weak homotopy types of the spaces $\mathcal F^{< k}$ (for all $k \geq 4$) and their integral homology types (for all $k \geq 3$) as $\Omega S^{k-1}$, the space of loops on a $(k-1)$-sphere!  In particular, the integral homology of the space  $\mathcal F^{\leq 2} =_{\mathsf{def}} \mathcal F^{< 3}$ is isomorphic to the homology of the loop space $\Omega S^2$. 
\smallskip

Arnold proved that the fundamental group $\pi_1(\mathcal F^{\leq 2}) \approx \Z$ (\cite{Ar}), the result that we will use on a number of occasions. Thus, $H^1(\mathcal F^{\leq 2}; \Z) \approx \Z$ as well. 

Therefore, for a $2$-dimensional convex quasi-envelop $\a: (X, v) \subset (A, \hat v)$ with no tangencies of $\a(\d X)$ to the oriented foliation $\mathcal F(\hat v)$ of order $\geq 3$, this fact allows to define a {\sf characteristic class} $J^\ast_\a \in H^1(A, \d_\star A; \Z) \approx \Z$ of $\a$. Here $\d_{\star} A = \emptyset$ for $A = \R \times S^1$, and $\d_\star A = \R \times \d([0, 1])$ for $A = \R \times [0, 1]$. 
\smallskip

Let $d$ be an even positive integer. We will employ the subspaces $\mathcal F_d^{\leq 2} \subset \mathcal F^{\leq 2}$, formed by functions whose ``{\sf degree}"---the sum of multiplicities of all its zeros---is even and does not exceed $d$.
\smallskip

\section{Convex quasi-envelops of traversing flows on surfaces and spaces of smooth functions  with $2$-moderate singularities}

Let us start with a number of definitions that have their origins in \cite{K1}-\cite{K5} and introduce different kinds of vector fields. Unfortunately, the list of definitions is long.
\smallskip

Let $v$ be a smooth vector field on a compact connected surface $X$ with boundary, such that $v \neq 0$ along the boundary $\d X$. Following \cite{Mo}, we consider the closed locus $\d_1^+X(v) \subset \d X$, where the field points inside of $X$ or is tangent to $\d X$,  and the closed locus $\d_1^-X(v)\subset \d X$, where $v$ points outside of $X$ or is tangent to $\d X$. The intersection $$\d_2X(v) =_{\mathsf{def}} \d_1^+X(v)\cap \d_1^+X(v)$$ is the locus where $v$ is \emph{tangent} to the boundary $\d X$. 

Points $z \in \d_2X(v)$ come in two flavors: by definition, $z \in  \d^+_2X(v)$ when $v(z)$ points inside of the locus $\d_1^+X(v)$, otherwise  $z \in  \d^-_2X(v)$. 


\begin{definition}\label{def.1.1}
A vector field $v$ on a compact surface $X$ is {\sf boundary generic} if:
\begin{itemize}
\item $v|_{\d X} \neq 0$, 
\item $v|_{\d X}$, viewed as a section of the normal $1$-dimensional (quotient) bundle \hfill\break $n_1 =_{\mathsf{def}} T(X)|_{\d X} \big/ T(\d X)$, is transversal to its zero section at the points of the locus $\d_2X(v)$. 
\hfill $\diamondsuit$
\end{itemize} 
\end{definition}

\noindent In particular, for a boundary generic $v$, the loci $\d_1^\pm X(v)$ are finite unions of closed intervals and circles, residing in $\d X$; the loci $\d_2^\pm X(v)$ are finite unions of points, residing in $\d_1 X$.

Each trajectory $\g$ of a traversing vector field $v$ must reach the boundary both in positive and negative times: otherwise $\g$ is not homeomorphic to a closed interval. \smallskip

\begin{definition}\label{def.1.2} A boundary generic vector field $v$ is {\sf boundary convex} if $\d_2^+X(v) = \emptyset$. A boundary generic $v$ is {\sf boundary concave} if $\d_2^-X(v) = \emptyset$. 
\hfill $\diamondsuit$
\end{definition}

\begin{definition}\label{def.2.3XXX} A non-vanishing vector field $v$ on a compact surface $X$ is {\sf traversing} if all its trajectories are closed segments or singletons.\footnote{It easy to see that the ends of these segments, as well as the singletons, must reside in the boundary $\d X$.}  

Equivalently, $v$ is traversing if there exists a Lyapunov function $f: X \to \R$ such that $df(v) > 0$ \cite{K1}.
\hfill $\diamondsuit$
\end{definition}

\begin{definition}\label{def.1.5} A traversing vector field $v$ on a compact surface $X$ is called {\sf traversally generic}, if the following properties hold:

 {\sf(1)} if a  trajectory $\g$ is tangent to the boundary $\d X$, then the tangency is simple, 
 
 {\sf (2)} no $v$-trajectory $\g$ contains more then one simple point of tangency to $\d X$. 
  \hfill $\diamondsuit$
\end{definition}

\begin{definition}\label{def.1.7} Let $\hat v$ be the standard traversing vector field on the surface $A$ (a strip or a cylinder), tangent to the fibers of the projections $\R \times [0, 1] \to [0, 1]$ or $\R \times S^1 \to S^1$. Let $\mathcal C$ be a closed $1$-dimensional manifold (a finite collection of several circles). Consider an \emph{immersion} $\b: \mathcal C \to A$.
\smallskip 

We say that such $\b$ is {\sf $2$-moderately generic} relatively to $\hat v$ 
if no $\hat v$-trajectory $\hat\g$ has {\sf order of tangency} $\geq 3$ to $\b(\mathcal C)$. Here the order of tangency is understood as the sum of tangency orders of local branches of $\b(\mathcal C)$ that pass through a given point of $\hat\g$. 
 \hfill $\diamondsuit$
\end{definition}

By standard techniques of the singularity theory,  we can perturb any given immersion $\b: \mathcal C \to A$ so that $\b(\mathcal C)$ will have only transversal self-intersections and will become boundary 
generic, and thus, $2$-moderately generic. \smallskip


In fact, any orientable connected surface $X$ with boundary admits an {\sf immersion} $\a: X \to A$ (or in the plane $\R^2$) (see Fig. \ref{fig.1.4A}). We will use this fact to pull-back to $X$ the standard traversing vector field $\hat v$ on $A$. 

\begin{definition}\label{def.1.8_A} Consider an \emph{immersion} $\a: X \to A$ of a  given compact (orientable) surface $X$ into the surface $A$, equipped with the standard vector field $\hat v$ and the foliation $\mathcal F(\hat v)$ it generates. 
\smallskip

$\bullet$ Given a transversing vector field $v$ on $X$, we call an immersion $\a$ a {\sf convex quasi-envelop} of $(X, v)$ if $v = \a^\dagger(\hat v)$, the pull-back of $\hat v$ under $\a$. 

$\bullet$ Such $\a$ is called {\sf $2$-moderately generic relative to} $\hat v$, if the restriction $\a|_{\d X}$ is $2$-moderately generic with respect to $\hat v$ in the sense of Definition \ref{def.1.7}. \smallskip
 \hfill $\diamondsuit$
\end{definition}

\begin{figure}[ht]
\centerline{\includegraphics[height=2.5in,width=3.2in]{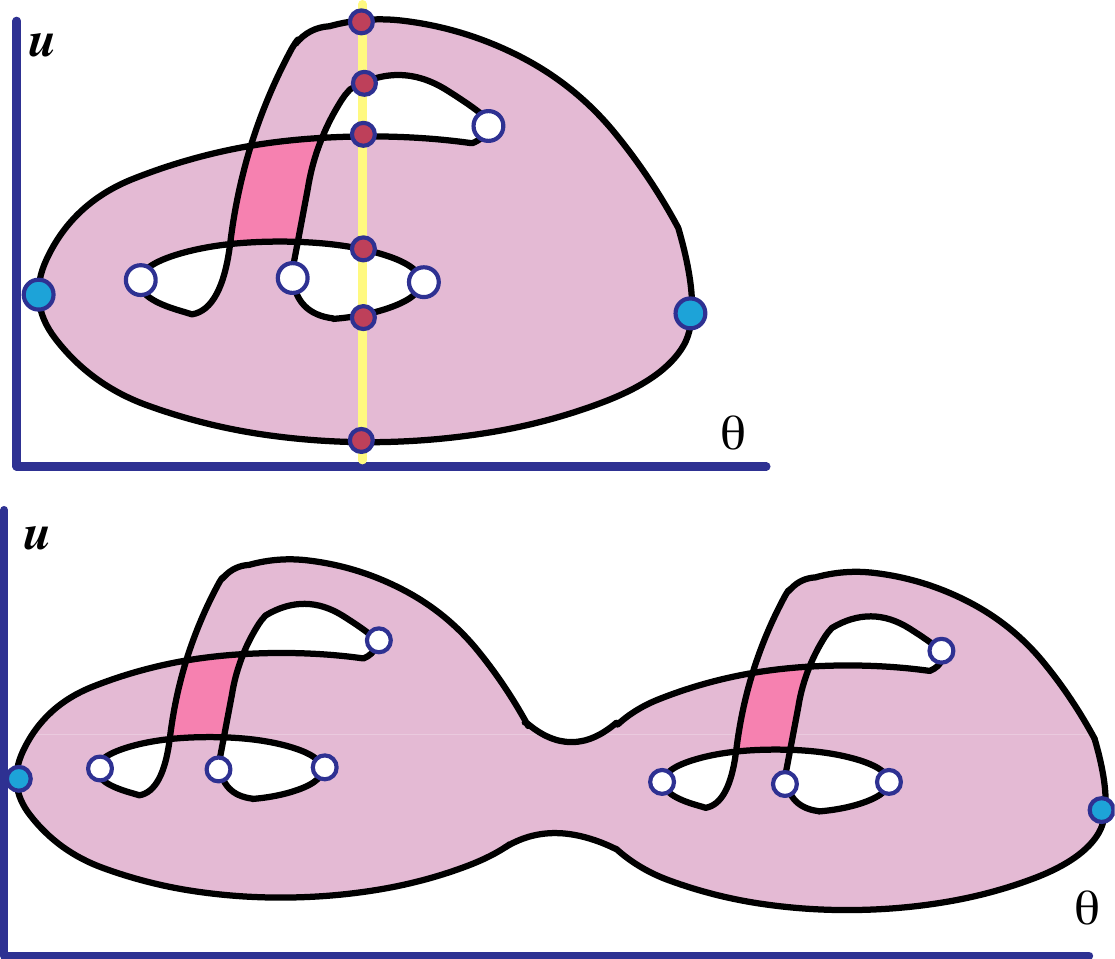}}
\bigskip
\caption{\small{A convex quasi-envelop $\a: X \to A$ of a $2$-moderately generic (actually, even traversally generic) vector field $\a^\dagger(\d_u)$ on a compact surface $X$, the torus from which an open smooth disk  is removed (the top diagram), and on a compact surface $X$, the closed surface of genus $2$ from which an open smooth disk is removed (the bottom diagram). In both examples, the cardinality of the fibers of the map $\theta \circ \a: \d X \to \mathcal T(\hat v)$ does not exceed $6$.}}
\label{fig.1.4A}
\end{figure}

The next definition is a special case of Definition \ref{def.1.8_A}. 

\begin{definition}\label{def.1.8} Let $\a: X \to A$ be a \emph{regular embedding}  of a  given compact surface $X$ into the surface $A$, carrying the standard vector field $\hat v$. We denote by $v$ the pull-back $\a^\dagger(\hat v)$ of $\hat v$ under $\a$. 

We say that  the pair $(A, \hat v)$ is a {\sf convex envelop} of $(X, v)$. If $\a|_{\d X}$ is $2$-moderately generic, we call $\a$ a {\sf $2$-moderately generic convex envelop}. 
\hfill $\diamondsuit$
\end{definition}

As the next lemma testifies, the existence of a convex envelop puts significant restrictions on the topology of orientable surfaces $X$: such $X$ are disks with holes or unions of such. 

\begin{lemma}\label{lem.1.4} If a compact connected surface $X$ with boundary has a pair of loops whose transversal intersection is a singleton, then no traversal flow on $X$ admits a convex envelop. In other words, if a connected surface $X$ with boundary has a handle, then no traversal flow on $X$ can be  convexly enveloped. 
\end{lemma}

\begin{proof} By \cite{K4}, Theorem 1.2, the space $\hat X$ of a convex envelop is either a disk or an annulus, both surfaces residing in the plane. No two loops in the plane intersect transversally at a singleton. Thus, for surfaces with a handle, no convex envelops exist.
\hfill
\end{proof}

To incorporate surfaces with handles into our constructions, in Definition \ref{def.1.8_A}, we  introduce the less restrictive notion of a convex \emph{quasi}-envelop. \smallskip

Now we would like to explore closely a nice connection between: 

{\sf (1)} immersions $\a: (X, v) \to (A, \hat v)$ of a compact surface $X$ in the surface $A$, such that $v = \a^\dagger(\hat v)$ and $\a(\d X)$  is $2$-moderately generic with respect to $\hat v$, and \smallskip

 {\sf (2)} loops in the functional spaces $\mathcal F^{\leq 2}$.\smallskip

Let $\a(\d X)^\times$ denote the finite set of self-intersections of  the curves forming the image $\a(\d X)$. Let $\a(\d X)^\circ$ denote the set $\a(\d X) \setminus \a(\d X)^\times$. \smallskip

With the pattern $\a(\d X) \subset A$ we associate an auxiliary smooth function  
$z_\a: A \to \R$, defined by the following properirs:
\begin{eqnarray}\label{eq.1.2}
\end{eqnarray}
\begin{itemize}
\item $z_\a^{-1}(0) = \a(\d X)$,
\item $0$ is the regular value of $z_\a$ at the points of $\a(\d X)^\circ$, 
\item in the vicinity of each transversal crossing point $a \in \a(\d X)^\times$, there exist  locally defined smooth functions $x_1, x_2: A \to \R$  such that: $0$ is their regular value,  $\{x_1^{-1}(0)\}$ and $\{x_2^{-1}(0)\}$ define the two local intersecting branches of $\a(\d X)$, and   $z_\a = x_1\cdot x_2$ locally.
\item $z_\a$  approaches $1$ at infinity in $A = \R\times S^1$ or in $A= \R\times [0, 1]$. 
\end{itemize}

The sign of $z_\a$ changes to the opposite, as a path in $A$ crosses an arc from $\a(\d X)^\circ$ transversally, thus providing a ``checker board" coloring of the domains in $A \setminus \a(\d X)$. The properties in (\ref{eq.1.2}) do not determine a unique function $z_\a$, but any such function will serve equally well in the forthcoming constructions. 
\smallskip

Given a smooth traversing vector field $v$ on a compact surface $X$, we denote by $\mathcal T(v)$ the space of $v$-trajectories. For a boundary generic and traversing $v$, the space $\mathcal T(v)$ is a finite graph; for a traversally generic $v$, the space $\mathcal T(v)$ is a finite graph whose vertexes are only of valency $3$ and $1$ \cite{K4}. For the standard vector field $\hat v$ on $A$, the trajectory space $\mathcal T(\hat v)$ is either a circle $S^1$, or an interval $[0, 1]$.\smallskip

Let $u: A \to \R$ be the obvious projection on the first (``vertical") factor. In particular,  $d u(\hat v) > 0$ and $u(\hat\g) = \R$ for all $\hat v$--trajectories $\hat\g$ in $A$. Then, with the help of $z_\a$ and $u$, we get a map $J_{z_\a}: \mathcal T(\hat v) \to \mathcal F^{\leq 2}$ whose source, the trajectory space  $\mathcal T(\hat v)$, is either a circle or a closed interval.  The image  of $J_{z_\a}$ belongs to the space of smooth functions $f: \R \to \R$ such that the set $\{f \leq 0\}$ is compact in $\R$,  $f$ has no zeros of multiplicity $\geq 3$, and $\lim_{u \to \pm\infty} f(u) =1$. We define the map $J_{z_\a}$ by the formula 
\begin{eqnarray}\label{eq1.12}
J_{z_\a}([\hat \g]) = (z_\a|_{\hat \g}) \circ (u |_{\hat\g})^{-1},
\end{eqnarray}
where $\hat \g$ stands for a $\hat v$-trajectory in $A$, and $[\hat\g]$ for the corresponding point in the trajectory space $\mathcal T(\hat v)$. If $A$ is the strip, $\mathcal T(\hat v)= [0, 1]$. The two ends of the interval are mapped by $J_{z_\a}$ to the convex subspace $\mathcal F^{+}$ of $\mathcal F$ that is formed by \emph{strictly positive} smooth functions. Therefore, 
we get a map of pairs $$J_{z_\a}: ([0, 1],\; \d[0, 1]) \to (\mathcal F^{\leq 2},\; \mathcal F^{+}).$$
The constant point-function $\mathbf 1$ is a deformation retract of  $\mathcal F^{+}$. Thus, homotopically, $J_{z_\a}$ can be regarded as a based loop $$J_{z_\a}: (S^1, pt)  \to (\mathcal F^{\leq 2}, \; \mathcal F^{+}) \sim (\mathcal F^{\leq 2},\; \mathbf 1).$$

For a fixed $\a$, the homotopy class $[J_{\a}]$ of the map $J_{z_\a}$ does not depend on the choice of the auxiliary function $z_\a$, subject to the four properties in (\ref{eq.1.2}). Indeed, consider the $\a_\ast$-image $\a_\ast(\nu)$ of the outer normal vector field $\nu$ to the boundary $\d X$ in $X$. At the points of $a \in \a(\d X)^\times$, we get two vectors $v_1(a), v_2(a) \in \a_\ast(\nu)$, one for each branch of $\a(X)$. 

Let $\mathcal L_w$ denote the directional derivative in the direction of a vector $w$. If $z_\a$ and $z'_\a$ are any two functions that satisfy all the properties from the list (\ref{eq.1.2}), we get $\mathcal L_{\hat v} z_\a > 0$ and $\mathcal L_{\hat v} z'_\a > 0$ at each point of $\a(\d X)^\circ$. Therefore, we get $\mathcal L_{\hat v} \{\tau z_\a + (1-\tau)z'_\a\}  > 0$, where $\tau \in [0, 1]$, at each point of $\a(\d X)^\circ$. At the same time, for $i=1, 2$, at each transversal crossing $a \in \a(\d X)^\times$, we have $\mathcal L_{\hat v}x_i > 0$,  and  $\mathcal L_{\hat v} x'_i > 0$. Hence, $\mathcal L_{\hat v} \{\tau x_i + (1-\tau)x'_i\}  > 0$. The rest of properties from (\ref{eq.1.2}) are obviously ``convex" .

Therefore,  the space of functions $z_\a$ that satisfy (\ref{eq.1.2}) is convex and thus contractible, which implies that the homotopy class $[J_{\a}]$ does not depend on the choice of $z_\a$.\smallskip

As a result, for $A = \R \times S^1$, any moderately generic convex quasi-envelop $\a: X \to A$,  produces a homotopy class $[J_{\a}] \in [S^1, \mathcal F^{\leq 2}]$. For $A = \R \times [0, 1]$, any moderately generic convex quasi-envelop $\a: X \to A$,  produces a homotopy class $[J_{\a}] \in [(S^1, pt),\, (\mathcal F^{\leq 2}, \mathbf 1)]$.  

Note that these homotopy classes do not depend on the orientation of the surface $X$.
\smallskip

We pick a generator $\kappa \in \pi_1(\mathcal F^{\leq 2}, \mathbf 1) \approx \Z$ (see \cite{Ar1} and Fig. \ref{fig.1.9_X}, (a) or (b), for realistic portraits of $\kappa$). 
For $A = \R \times [0,1]$, we define an integer $\mathsf J^\a$ by the formula $\mathsf J^\a \cdot \kappa = [J_{\a}]$. \smallskip



The isomorphism $\pi_1(\mathcal F^{\leq 2}, \mathbf 1) \approx \Z$ follows from \cite{Ar1}. A slight modification of Arnold's arguments leads to Theorem \ref{th.1.7} below. The main difference between our constructions and the ones from \cite{Ar1} is that Arnold is concerned with immersions of $1$-dimensional oriented ``doodles" in $A$, while we deal also with immersions of ``blobs" (compact orientable surfaces) in $A$  (compare diagrams (a) and (b) in Fig. \ref{fig.1.4B}).  
\smallskip

We denote by $\hat v^\bullet$ the vector field $(\hat v, 0)$ on the solid $A \times [0,1]$ and by $\pi: A \times [0,1] \to [0, 1]$ the obvious projection. \smallskip

We have seen how a $2$-moderately generic immersion $\a: X \to A$ produces a homotopy class in $[S^1,\, \mathcal F^{\leq 2}]$ or in $[(D^1, \d D^1), (\mathcal F^{\leq 2}, \mathbf 1)]$. \smallskip

On the other hand, generic (oriented) loops in $\b: S^1 \to \mathcal F^{\leq 2}$ have an interpretation as finite collections  $\mathcal C$ of smooth \emph{embedded} closed curves (embedded ``doodles") in the surface $A$. These curves have no tangency to the $\hat v$-trajectories $\hat\g$ of order that exceeds two. In particular, any \emph{inflections} of the curves with respect to the $\hat\g$-trajectories are forbidden. 
Furthermore,  a generic homotopy between the maps $\b: S^1 \to \mathcal F^{\leq 2}$ corresponds to some {\sf cobordism like relation} between the corresponding plane curves $\mathcal C$.  The cobordism also avoids the forbidden tangencies of orders $\geq 3$ to the standard foliation $\mathcal F(\hat v^\bullet)$ on $A \times [0, 1]$.\smallskip


In order to define this cobordism, let us spell out more accurately the $2$-moderate genericity requirements on the collections  of doodles in $A$:
\begin{eqnarray}\label{eq1.4_A}
\end{eqnarray}
\begin{itemize}
\item $\b: \mathcal C  \to A$ is a smooth {\sf immersion} of a finite collection of (oriented) circles,
\item 
no intersections of $\b(\mathcal C)$ of multiplicities $\geq 3$ are permitted, 
\item the order of tangency between each branch of $\b(\mathcal C)$ and each $\hat v$-trajectory $\hat\g$ does not exceed $2$.
\item if a branch of $\b(\mathcal C)$ is quadratically tangent to a $\hat v$-trajectory at a point $x$, then no other branch contains $x$. \hfill $\diamondsuit$
\end{itemize}


\begin{definition}\label{def1.10XY}
Given two immersions $\b_0: \mathcal C_0 \to A$ and $\b_1: \mathcal C_1 \to A$  as in (\ref{eq1.4_A}), we say that they are {\sf (orientably) $2$-moderately cobordant}, 
if there is a compact (oriented) surface  $S$  and a smooth immersion $B: S \to A \times [0, 1]$
 such that:
\begin{enumerate}
\item $\d S = \mathcal C_0 \coprod -\mathcal C_1$,
\item $B|_{\mathcal C_0 \coprod \mathcal C_1} = \b_0 \coprod \b_1$,
\item the immersion $B| : S \to A \times [0, 1]$ is $2$-to-$1$ at most,
\item if two local branches of $B(S)$ intersect at a point $z = a \times \{t\}$, then each of them is transversal to the trajectory $\hat \g \times \{t\}$ through $z$, 
\item 
every $\hat v^\bullet$-trajectory $\hat\g \times \{t\} \subset A \times [0, 1]$ is tangent to each local branch of the surface $B(S)$ with the order of tangency that does not exceed $2$,\footnote{In particular, the cubic tangencies are forbidden.}
\item the map 
$S \stackrel{B|}{\longrightarrow} A \times [0, 1] \to [0, 1]$ is a Morse function with the regular values $\{0\}$ and $\{1\}$,
\hfill $\diamondsuit$
\end{enumerate}
\end{definition}

We introduce {\sf the bordism set} $\mathbf C^{\mathsf{imm}}_{\mathsf{moderate \leq 2}}(A)$ of immersions $\b: \mathcal C \to A$ with $2$-moderate tangencies to the foliation $\mathcal F(\hat v)$. It is possible to verify that this cobordism between immersions is an \emph{equivalence relation} (see \cite{K6}, \cite{K7}). \smallskip

Replacing ``immersions" with ``regular embeddings" in Definition \ref{def1.10XY}, we get a modified notion of bordisms of (oriented) doodles. We denote the set of such bordisms by $\mathbf C^{\mathsf{emb}}_{\mathsf{moderate \leq 2}}(A)$. 
Depending on whether we consider the loops, forming $\mathcal C$, oriented or not, we get oriented versions $\mathbf{OC}^{\mathsf{imm}}_{\mathsf{moderate \leq 2}}(A)$ and $\mathbf{OC}^{\mathsf{emb}}_{\mathsf{moderate \leq 2}}(A)$ of the bordisms $\mathbf C^{\mathsf{imm}}_{\mathsf{moderate \leq 2}}(A)$ and $\mathbf C^{\mathsf{emb}}_{\mathsf{moderate \leq 2}}(A)$.
\smallskip

In the next definition, all surfaces and $3$-folds that admit submersions in $A$ or in $A \times [0, 1]$ automatically must be orientable, but not necessarily oriented. 

\begin{definition}\label{def1.10_A}
Given two (oriented) submersions $\a_0: X_0 \to A$ and $\a_1: X_1 \to A$ with $\a_0|_{\d X_0}$ and $\a_1|_{\d X_1}$ as in (\ref{eq1.4_A}), we say that they are {\sf $2$-moderately cobordant}, 
if there is a compact (oriented) $3$-fold $W$ with corners $\d X_0 \coprod \d X_1$ and a submersion $B: W \to A \times [0, 1]$
 such that:
\begin{enumerate}
\item $\d W = X_0 \cup -X_1\cup \delta W$, where $\delta W =_{\mathsf{def}} \d W \setminus \mathsf{int}(X_0 \coprod X_1)$,\smallskip

\item $B|_{X_0 \coprod X_1} = \a_0 \coprod \a_1$,\smallskip

\item $B| : \delta W \to A \times [0, 1]$ has at most double self-intersections,\smallskip

\item if two local branches of $B(\delta W)$ intersect at a point $z = a \times \{t\} \in A \times [0, 1]$, then each of them is transversal to the trajectory $\hat \g \times \{t\}$ through $z$, \smallskip

\item 
every $\hat v^\bullet$-trajectory $\hat\g \times \{t\} \subset A \times [0, 1]$ is tangent to each local branch of the surface $B(\delta W)$ with the order of tangency $\leq 2$, \smallskip 

\item the map 
$\delta W \stackrel{B|}{\longrightarrow} A \times [0, 1] \to [0, 1]$ is a Morse function with the regular values $\{0\}$ and $\{1\}$.
\hfill $\diamondsuit$
\end{enumerate}
\end{definition}
 Thus, we can talk about {\sf the bordism set} $\mathbf B^{\mathsf{imm}}_{\mathsf{moderate \leq 2}}(A)$ or $\mathbf{OB}^{\mathsf{imm}}_{\mathsf{moderate \leq 2}}(A)$ of immersions/embeddings $\a: X \to A$ with $2$-moderate tangencies to the foliation $\mathcal F(\hat v)$. It is possible to verify that this cobordism between immersions is an \emph{equivalence relation} (see \cite{K7}). 

If in Definition \ref{def1.10_A} we replace all the ``immersions" with the ``regular embeddings" and drop vacuous constraints (2) and (3) of the definition, we will get a similar version of (oriented) bordisms of regular embeddings. We denote them by $\mathbf B^{\mathsf{emb}}_{\mathsf{moderate \leq 2}}(A)$ and  $\mathbf{OB}^{\mathsf{emb}}_{\mathsf{moderate \leq 2}}(A)$. \smallskip

 Note that, if in Definition \ref{def1.10_A} we drop all the constraints related to how the boundary of embedded blobs and of the solid embedded cobordisms interact with the $\hat v$-trajectories in $A$ and  with the $\hat v^\bullet$-trajectories in $A \times [0, 1]$, then the corresponding bordism groups $\mathbf{B}^{\mathsf{emb}}(A)$ and $\mathbf{OB}^{\mathsf{emb}}(A)$ of $A$ are trivial. Indeed, by pushing (oriented) blobs  $\a: X \hookrightarrow A$ inside $A \times [0, 1]$, while keeping their boundaries $\a(\d X)$ fixed in $A \times \{0\}$, we create a (oriented) solid that delivers the cobordism between $\a(X)$ and the empty blob. 
 \smallskip
 

The following obvious maps are available: 
\begin{eqnarray}\label{eq.1.5_XYZ}
\mathcal A:\; \mathbf B^{\mathsf{emb}}_{\mathsf{moderate \leq 2}}(A) & \longrightarrow & \mathbf B^{\mathsf{imm}}_{\mathsf{moderate \leq 2}}(A), \nonumber \\
\mathcal A:\; \mathbf{OB}^{\mathsf{emb}}_{\mathsf{moderate \leq 2}}(A) & \longrightarrow & \mathbf{OB}^{\mathsf{imm}}_{\mathsf{moderate \leq 2}}(A), \\
\mathcal A:\; \mathbf C^{\mathsf{emb}}_{\mathsf{moderate \leq 2}}(A) & \longrightarrow & \mathbf C^{\mathsf{imm}}_{\mathsf{moderate \leq 2}}(A),  \nonumber\\
\mathcal A:\; \mathbf{OC}^{\mathsf{emb}}_{\mathsf{moderate \leq 2}}(A) & \longrightarrow & \mathbf {OC}^{\mathsf{imm}}_{\mathsf{moderate \leq 2}}(A).
\end{eqnarray}
Also, taking the boundaries of (oriented) 
blobs, we get the obvious maps:
\begin{eqnarray}\label{eq.1.5_XYZ}
\mathcal B^\d:\; \mathbf B^{\mathsf{imm}}_{\mathsf{moderate \leq 2}}(A) & \longrightarrow & \mathbf{C}^{\mathsf{imm}}_{\mathsf{moderate \leq 2}}(A),  \nonumber
\\
\mathcal B^\d:\; \mathbf{OB}^{\mathsf{imm}}_{\mathsf{moderate \leq 2}}(A) & \longrightarrow & \mathbf{OC}^{\mathsf{imm}}_{\mathsf{moderate \leq 2}}(A), 
\\
\label{eq.1.8XYZ} \mathcal B^\d:\; \mathbf B^{\mathsf{emb}}_{\mathsf{moderate \leq 2}}(A) & \longrightarrow & \mathbf{C}^{\mathsf{emb}}_{\mathsf{moderate \leq 2}}(A),  \nonumber
\\
\mathcal B^\d:\; \mathbf{OB}^{\mathsf{emb}}_{\mathsf{moderate \leq 2}}(A) & \longrightarrow & \mathbf{OC}^{\mathsf{emb}}_{\mathsf{moderate \leq 2}}(A). 
\label{eq.1.9XYZ}
\end{eqnarray}

\smallskip
 The maps $\mathcal B^\d$ in  (\ref{eq.1.9XYZ}) are \emph{bijections} (see Theorem \ref{th.1.7} and \cite{K6}, \cite{K7}). \smallskip

Let us add one mild requirement to the list of the six properties in Definition \ref{def1.10_A}: 
\begin{eqnarray}\label{eq.extra_transversality}
\bullet \text{ All  the double intersections of } B |: \delta W \to A \times [0,1] \text{ are transversal.} 
\end{eqnarray} 

Combining the six properties from Definition \ref{def1.10_A} with the property in (\ref{eq.extra_transversality}), we could, at the first glance, get a more rigid notion of cobordisms of immersions.  Let us denote them temporarily   by $\mathbf B^{\mathsf{\times imm}}_{\mathsf{moderate \leq 2}}(A)$.\smallskip

Thus, we have the obvious map
$$\mathcal A^\times: \mathbf B^{\mathsf{\times imm}}_{\mathsf{moderate \leq 2}}(A) \to \mathbf B^{\mathsf{imm}}_{\mathsf{moderate \leq 2}}(A).$$
$$\mathcal A^\times: \mathbf{OB}^{\mathsf{\times imm}}_{\mathsf{moderate \leq 2}}(A) \to \mathbf{OB}^{\mathsf{imm}}_{\mathsf{moderate \leq 2}}(A).$$
By the general position argument, we may isotop an immersion $B$ via immersions so that all its self-intersections become transversal. Therefore, the maps $\mathcal A^\times$ are actually bijections \cite{K6}. 
\smallskip

In fact, the sets $\mathbf B^{\mathsf{imm}}_{\mathsf{moderate \leq 2}}(A),\, \mathbf B^{\mathsf{emb}}_{\mathsf{moderate \leq 2}}(A)$ and $\mathbf{OB}^{\mathsf{imm}}_{\mathsf{moderate \leq 2}}(A),\, \mathbf{OB}^{\mathsf{emb}}_{\mathsf{moderate \leq 2}}(A)$ have a {\sf group structure}. The group operation $\star$ takes a pair  of $2$-moderate submersions, $\a_1: X_1 \to A$ and $\a_2: X_2 \to A$, to a new $2$-moderate submersion $\a_1 \star \a_2: X_1 \coprod X_2 \to A$ by placing the image of $\a_2$ in $\R \times S^1$ or in $\R \times [0,1]$ {\sf above} the image of $\a_1$ (see Fig. \ref{fig.1.9_X}, (c) and (d)). At the first glance, the operation $\star$ seems to be non-commutative. At least for $A = \R \times [0,1]$, this first impression is wrong: by an appropriate isotopy one can switch the vertical order of $\a_1(X_1)$ and  
$\a_2(X_2)$ without violating the requirements that all the immersions are $2$-moderate.  
\smallskip

Note that the cardinality of the fibers of the map $\theta \circ (\a_1 \star \a_2)$  is less than or equal to the \emph{sum} of cardinalities of the fibers of $\theta \circ \a_1$ and of $\theta \circ \a_2$, where $\theta: A \to \mathcal T(\hat v)$ is the obvious map. 
\smallskip

In the case of $A = \R \times [0,1]$, another \emph{commutative} group structure is available in the sets $\mathbf B^{\mathsf{imm/emb}}_{\mathsf{moderate \leq 2}}(A)$ and  $\mathbf{OB}^{\mathsf{imm/emb}}_{\mathsf{moderate \leq 2}}(A)$ (note that, for more general than the $2$-moderate combinatorial constraints on the tangency patterns, similar groups may not be commutative \cite{K7}). It is induced by the operation 
$$\uplus:\; \mathbf B^{\mathsf{imm/emb}}_{\mathsf{moderate \leq 2}}(A) \times \mathbf B^{\mathsf{imm/emb}}_{\mathsf{moderate \leq 2}}(A) \to \mathbf B^{\mathsf{imm/emb}}_{\mathsf{moderate \leq 2}}(A)$$
$$\uplus:\; \mathbf{OB}^{\mathsf{imm/emb}}_{\mathsf{moderate \leq 2}}(A) \times \mathbf {OB}^{\mathsf{imm/emb}}_{\mathsf{moderate \leq 2}}(A) \to \mathbf {OB}^{\mathsf{imm/emb}}_{\mathsf{moderate \leq 2}}(A)$$
that places the image of $\a_2$ {\sf to the right} of the image of $\a_1$ along the segment $[0, 1]$; that is, $\a_1(X_1)$ is placed in the strip $\R \times (0,\, 0.5)$, while $\a_2(X_2)$ is placed in the strip  $\R \times (0.5,\, 1)$ (see Fig. \ref{fig.1.4A}, lower diagram). 

The zero element in $\mathbf B^{\mathsf{imm/emb}}_{\mathsf{moderate \leq 2}}(A)$ or in $\mathbf{OB}^{\mathsf{imm/emb}}_{\mathsf{moderate \leq 2}}(A)$ is represented by the empty surface $X$, and the minus of an immersion $\a: X \to A$ by a composition of $\a$ with a flip of $A$ with respect to the trajectory $\R \times \{0.5\}$. 

In contrast with the operation $\star$, the cardinality of the fibers of the map $\theta \circ (\a_1 \uplus \a_2)$  is less or equal to the \emph{maximum} of cardinalities of the fibers of $\theta \circ \a_1$ and of $\theta \circ \a_2$. Due of this good feature of the operation $\uplus$, we choose to study its generalizations in \cite{K6}, \cite{K7}. 
\smallskip

\begin{lemma} \label{lem.1.5XY}
For $A = \R \times [0,1]$, the operation $\uplus$ in $\mathbf B^{\mathsf{imm}}_{\mathsf{moderate \leq 2}}(A)$ or in $\mathbf{OB}^{\mathsf{imm}}_{\mathsf{moderate \leq 2}}(A)$ and in $\mathbf B^{\mathsf{emb}}_{\mathsf{moderate \leq 2}}(A)$ or in $\mathbf{OB}^{\mathsf{emb}}_{\mathsf{moderate \leq 2}}(A)$ is commutative.
\end{lemma}

\begin{proof} By a parallel transport in $A$ of the images $\a_1(X_1) \subset \R \times (0,\, 0.5)$ and $\a_2(X_2) \subset \R \times (0.5,\, 1)$, we can switch their order along $[0,1]$. This can be done  by sliding up  $\a_2(X_2)$ so that it will reside in $(q, +\infty) \times (0.5,\, 1)$, then by  sliding down $\a_1(X_1)$ so that it will reside in $(-\infty, -q) \times (0.5,\, 1)$. Here $q$ is a positive number that exceeds the vertical size of both images. Then we slide horizontally the new $\a_2(X_2)$ and place it $(q, +\infty) \times (0,\, 0.5)$.
Similarly, we slide horizontally the new $\a_1(X_1)$ and place it $(q, +\infty) \times (0.5,\, 1)$. Finally, we slide down $\a_1(X_1)$ and slide up $\a_2(X_2)$, thus completing the exchange. 

Thanks to the nature of all these slides (parallel shifts), the maximal order of tangency of the $\hat v$-trajectories to the moving images $\a_1(\d X_1)$ and $\a_2(\d X_2)$ does not exceed $2$.

Note that the maximal cardinality of the fibers of the projection $\a_1(\d X_1) \coprod \a_2(\d X_2) \to [0, 1]$ did increase in the exchange process. 
\hfill
\end{proof}

Let $\mathcal F^{= 2} \subset \mathcal F^{\leq 2}$ denote the {\sf discriminant hypersurface} in the space $\mathcal F^{\leq 2}$, formed by the functions from $\mathcal F^{\leq 2}$ with at least one zero of multiplicity $2$. \smallskip
 
The following proposition is similar to Theorem from \cite{Ar1}; however, our notion of cobordism of embedded blobs is different from the Arnold's more combinatorial notion of the cobordism of embedded doodles with no vertical inflection points.
\begin{theorem}\label{th.1.7} 
$\bullet$ For $A = \R \times [0, 1]$, the construction $\{\a \leadsto J_{z_\a}\}$ in (\ref{eq1.12}), where the immersion $\a|: \d X \to A$ has only $2$-moderate tangencies to the foliation $\mathcal F(\hat v)$, 
delivers a group homomorphism $$J^{\mathsf{imm}}: \mathbf B^{\mathsf{imm}}_{\mathsf{moderate \leq 2}}(A) \to \pi_1(\mathcal F^{\leq 2}, \mathbf 1) \approx \Z,$$
where the group addition in  $ \mathbf B^{\mathsf{imm}}_{\mathsf{moderate \leq 2}}(A)$ is the operation $\uplus$.\smallskip

$\bullet$ The same construction produces a group \emph{isomorphism}  $$J^{\mathsf{emb}}: \mathbf B^{\mathsf{emb}}_{\mathsf{moderate \leq 2}}(A) \approx \pi_1(\mathcal F^{\leq 2}, \mathbf 1)  \approx \Z.$$ The inverse map $(J^{\mathsf{emb}})^{-1}$ is delivered by the correspondence 
$$K: \big\{\tau: ([0, 1], \d[0,1]) \to (\mathcal F^{\leq 2}, \mathbf 1)\big\}\; \Rightarrow \; \bigcup_{\theta \in [0, 1]} \big(\tau(\theta)^{-1}\big((-\infty, 0])\big),\; \theta \big) \subset \R^1 \times [0, 1],$$
where $\tau$ is a continuous path in $\mathcal F^{\leq 2}$,  transversal to the discriminant hypersurface $\mathcal F^{= 2}$. \smallskip

$\bullet$ For $A = \R \times S^1$, the construction $\{\a \leadsto J_{z_\a}\}$, where an embedding $\a: X \to A$ has only $2$-moderate tangencies, 
delivers a $1$-to-$1$ map $$J^{\mathsf{emb}}: \mathbf B^{\mathsf{emb}}_{\mathsf{moderate \leq 2}}(A) \approx [S^1,\, \mathcal F^{\leq 2}].$$ 

The inverse map $(J^{\mathsf{emb}})^{-1}$ is induced by the correspondence 
$$K: \big\{\tau: S^1 \to \mathcal F^{\leq 2}\big\}\; \Rightarrow \; \bigcup_{\theta \in S^1} \big(\tau(\theta)^{-1}\big((-\infty, 0])\big),\; \theta \big) \subset \R^1 \times S^1.$$
\end{theorem}

\begin{proof}
The validation of this theorem is inspired by the \emph{graphic calculus} in \cite{Ar1} that converts homotopies of loops in the functional space $\mathcal F_{\leq 2}$ into cobordisms (surgeries) of doodles in $A$ with $2$-moderate tangencies to the foliation $\mathcal F(\hat v)$. With the help of this calculus, the isomorphism $\pi_1(\mathcal F^{\leq 2}, \mathbf 1) \approx \Z$ is established \cite{Ar1}.  In the present case, the loops are images of boundaries of oriented compact surfaces under their immersions/embeddings in $A$. Fig. \ref{fig1.8_X} \, shows our modification of this graphic calculus in action. A generator of $\pi_1(\mathcal F_{\leq 2}, \mathbf 1) \approx \Z$ is depicted in Fig. \ref{fig.1.9_X},  (b).


 \smallskip
  
If two $2$-moderate immersions, $\a_0: X_0 \to A$ and $\a_1: X_1 \to A$, are cobordant with the help of a $2$-moderate  immersion $B: W \to A \times [0, 1]$ as in Definition \ref{def1.10_A}, extending the function $z_{\a_0} \coprod z_{\a_1}: A \times (\{0\} \coprod \{1\}) \to \R$ to a smooth function $Z: A \times [0, 1] \to \R$ with similar properties delivers a homotopy between the loops $J_{z_{\a_0}}$  and  $J_{z_{\a_1}}$ (see (\ref{eq1.12})) in $\mathcal F_{\leq 2}$. 

On the other hand, if a path $\tau: ([0, 1],\, \d[0,1]) \to (\mathcal F^{\leq 2}, \mathbf 1)$ is transversal to the discriminant hypersurface $\mathcal F^{= 2} \subset \mathcal F^{\leq 2}$, then 
the locus  $$K(\tau) =_{\mathsf{def}}\bigcup_{\theta \in [0, 1]} \big(\tau(\theta)^{-1}(0),\, \theta \big) \subset \R^1 \times [0, 1]$$ is a smooth embedded curve in $A$ (see \cite{KSW1}, Lemma 3.4, 
for validation of this claim and its generalizations). Since the set $f^{-1}((-\infty, 0]))$ is compact for any $f \in \mathcal F_{\leq 2}$, the curve $K(\tau)$ is the boundary of a compact surface $$X(\tau) =_{\mathsf{def}}\bigcup_{\theta \in [0, 1]} \big(\tau(\theta)^{-1}\big((-\infty, 0])\big),\, \theta \big) \subset A.$$
Similar arguments hold for a homotopy $$B: \big([0, 1],\;  \d[0, 1]\big)\times [0, 1] \to (\mathcal F^{\leq 2}, \mathbf 1)$$ 
between two such maps $\tau_0$ and $\tau_1$, where $B$ is transversal to $\mathcal F^{= 2}$. 
The $3$-fold  $$W(B) =_{\mathsf{def}}\bigcup_{\theta \in [0, 1],\; t \in [0, 1]} \big(B(\theta)^{-1}\big((-\infty, 0])\big),\, \theta,\, t \big) \subset A \times [0,1]$$
delivers the cobordism (as in Definition \ref{def1.10_A}) between $X(\tau_0)$ and $X(\tau_1)$. Therefore, 
$$J^{\mathsf{emb}}: \mathbf B^{\mathsf{emb}}_{\mathsf{moderate \leq 2}}(A) \approx \pi_1(\mathcal F^{\leq 2}, \mathbf 1)$$ is a bijection (actually, a group isomorphism with respect to the operation $\uplus$). 

This fact has a curious implication for the map $\mathcal A$ from (\ref{eq.1.5_XYZ}):  since the map $J^{\mathsf{imm}}: \mathbf B^{\mathsf{imm}}_{\mathsf{moderate \leq 2}}(A) \to \pi_1(\mathcal F^{\leq 2}, \mathbf 1)$ is available, we can compose it with the inverse of the bijection $J^{\mathsf{emb}}$ to get a surjective map 
\begin{eqnarray}\label{eq.1.7_XYZ}
\mathcal R_J:\; \mathbf B^{\mathsf{imm}}_{\mathsf{moderate \leq 2}}(A) \to \mathbf B^{\mathsf{emb}}_{\mathsf{moderate \leq 2}}(A) \approx \Z,
\end{eqnarray} 
which serves as {\sf the right inverse} of the map (homomorphism) $\mathcal A$; that is, $\mathcal R_J\circ \mathcal A = \mathsf{id}$. In fact, for $A = \R \times [0,1]$, $\mathcal R_J$ is a group epimorphism.

The case $A = \R \times S^1$ is similar, but the operation $\uplus$ is not available. Instead, the opreration $\star$ is available, but its commutativity is in question.
\hfill
\end{proof}


We orient the surface $A$ so that the the $\theta$-coordinate,  corresponding to the trajectory space $\mathcal T(\hat v) = [0, 1]\,\text{ or }  S^1$, is the first, and the $u$-coordinate, corresponding to the multiplier $\R$, is the second. 
With this {\sf counterclockwise orientation} of $A$ being fixed, any immersion $\a: X \to A$ induces an orientation of the surface $X$, thus choosing orientations of each component of $\d X$. However, this induced orientation, may differer from the original orientation of $X$, an ingredient in the definition of $\mathbf{OB}^{\mathsf{imm}}_{\mathsf{moderate \leq 2}}(A)$ !
\smallskip

Given an  immersion $\a: (X, v) \subset (A, \hat v)$ such that $\a(\d X)$ has the properties as in (\ref{eq1.4_A}), we attach a {\sf new $\a$-dependent polarity} to each ``concave" point $a \in \d_2^+X(v)$: 
by definition, the polarity of $a$  is ``$\oplus$" if  $\a_\ast(\nu_a)$, where $\nu_a$ is the inner normal to $\d X$ at $a$, points in the direction of the coordinate $\theta$. Otherwise, the new polarity of $a$ is defined to be ``$\ominus$" (see Fig. \ref{fig.1.9_X}). Equivalently,  $a$  is of polarity ``$\oplus$" if crossing the critical value $\theta(a)$ in the positive direction increases the cardinality of the fiber of the map $\theta: \a(\d X) \to \mathcal T(\hat v)$ in the vicinity of $a$.

As a result, the loci $\{\d_j^+X(v)\}_j$ and  $\{\d_j^-X(v)\}_j$ acquire the additional polarities $\oplus$ and $\ominus$; all together, four flavors for the tangencies of $\a(\d X)$ to $\mathcal F(\hat v)$ are available: $\{(+, \oplus), (+, \ominus),\hfill\break  (-, \oplus), (-, \ominus)\}$. These flavors are independent of the orientations of $X$.
\smallskip 


When dealing with \emph{oriented} $2$-moderate blobs and doodles, it will be useful to introduce still another polarity (orientation) of tangency points, which will be denoted ``$\uparrow$" and ``$\downarrow$". Each point $a \in \d_2^+X(v)$ has polarity ``$\uparrow$" if the vector field $\hat v(a)$ is consistent with the orientation of $\a(\d X)$ at $a$.  Otherwise, the polarity of $a$ is ``$\downarrow$". Similar definition is applied to the points of the locus  $\d_2^-X(v)$. As a result, each point of the locus $\d_2X(v)$, with the help of $\a$, acquires the following eight flavors:
$$\{(+, \oplus, \uparrow),\; (+, \ominus, \uparrow),\; (-, \oplus, \uparrow),\; (-, \ominus, \uparrow), \;  (+, \oplus, \downarrow),\; (+, \ominus, \downarrow),\; (-, \oplus, \downarrow),\; (-, \ominus, \downarrow)\}.$$
Four of these eight flavors will play an essential role:
\begin{eqnarray}\label{eq.FOUR_FLAVORS}
\big\{\underbrace{(+, \oplus, \uparrow),\; (+, \ominus, \downarrow)}_{clockwise},\;  \underbrace{(+, \oplus, \downarrow),\; (+, \ominus, \uparrow)}_{counterclockwise}\big\},
\end{eqnarray}
the first pair occurring in the blobs oriented {\sf clockwise}, the second pair in the blobs oriented {\sf counterclockwise}.

Thus, with any embedded \emph{oriented} surface $\a: X \hookrightarrow A$, we associate two integers 
\begin{eqnarray}\label{eq.NEW_ORIENTED_invariants}
\mathsf K^\a &=_{\mathsf{def}} & \#\{\d_2^{+,\, \oplus,\, \downarrow}X(\hat v)\} - \#\{\d_2^{+,\, \ominus,\, \uparrow}X(\hat v)\}, \nonumber \\
\mathsf L^\a & =_{\mathsf{def}} & \#\{\d_2^{+, \,\oplus,\, \uparrow}X(\hat v)\} - \#\{\d_2^{+,\, \ominus,\, \downarrow}X(\hat v)\}    
\end{eqnarray}
that correspond to the blobs that are oriented {\sf counterclockwise} and {\sf clockwise}, respectively.

We introduce also an integer  
\begin{eqnarray}\label{eq.J_INVARIANT}
\quad \quad \quad \mathsf J^\a =_{\mathsf{def}} \#\{\d_2^{+,\, \oplus}X(\hat v)\} - \#\{\d_2^{+,\, \ominus}X(\hat v)\}\\
 \;\; (\text{in the oriented case, }  \mathsf J^\a = \mathsf K^\a + \mathsf L^\a). \nonumber
\end{eqnarray}
As the next theorem testifies, this number plays a role as a bordism invariant of  immersions of non-oriented blobs and as a tool for characterizing elements of the fundamental group $\pi_1( \mathcal F^{\leq 2},\, \mathbf 1)$. 



\begin{theorem}\label{th1.8} Let $A = \R \times [0, 1]$. 

$\bullet$ Any regular embedding $\a: (X, v) \to (A, \hat v)$ with only $2$-moderate tangencies of $\a(\d X)$ to $\mathcal F(\hat v)$ produces a map $J_{z_\a}$ as in Theorem  \ref{th.1.7}. Its homotopy class  $[J_{z_\a}] = \mathsf J^\a \cdot \kappa$, where $\kappa$ denotes a generator of $\pi_1( \mathcal F^{\leq 2},\, \mathbf 1) \approx \Z$ (shown in Fig. \ref{fig.1.9_X}, {\sf(b)}), and $\mathsf J^\a \in \Z$. 
\smallskip

The integer $\mathsf J^\a$ can be computed by the formula (\ref{eq.J_INVARIANT}).
\smallskip

$\bullet$ Any regular embedding $\a: (X, v) \to (A, \hat v)$ of an \emph{oriented} surface $X$ with only $2$-moderate tangencies  produces two maps $$K_{z_\a}: ([0, 1],\, \d[0,1]) \to (\mathcal F^{\leq 2}, \mathbf 1) \text{ and } L_{z_\a}: ([0, 1],\, \d[0,1]) \to (\mathcal F^{\leq 2}, \mathbf 1)$$ as in Theorem  \ref{th.1.7}; the first map is generated by the counterclockwise oriented blobs, and the second one by the clockwise oriented blobs. 
The homotopy classes of these maps are:  $[K_{z_\a}] = \mathsf K^\a \cdot \kappa$ and  $[L_{z_\a}] = \mathsf L^\a \cdot \kappa$.  \smallskip

The integers $\mathsf K^\a$ and $\mathsf L^\a$ can be computed by the formula (\ref{eq.NEW_ORIENTED_invariants}).\smallskip

$\bullet$ Moreover, $\mathsf K^\a$ and $\mathsf L^\a$ deliver an isomorphism
\begin{eqnarray}\label{eq.ORIENT_BLOB_EMBED}
 \mathsf K \times \mathsf L:\, \mathbf{OB}^{\mathsf{emb}}_{\mathsf{moderate \leq 2}}(A) \stackrel{\approx}{\longrightarrow} \Z \times \Z.
 \end{eqnarray}
%
\end{theorem}

\begin{proof} Let $d =_{\mathsf{def}} \max_{\hat\g} \# \{\hat\g \cap \a(\d X)\}$ be the maximal cardinality of the intersections of the $\hat v$-trajectories $\hat\g$ with the loop pattern $\a(\d X)$. Since $X$ bounds $\d X$, $d$ must be even.\smallskip

We start with the case of non-oriented blobs.

For any $2$-moderate immersion $\a: X \to A$, we pick an auxiliary function $z_\a : A \to \R$, adjusted to $\a$ as in (\ref{eq.1.2}).  By the previous arguments, this choice  of $z_\a$ produces the relative loop  $J_{z_\a}: ([0,1], \d[0, 1]) \to  \mathcal (F_d^{\leq 2}, \mathbf 1)$. 



 Consider the restriction $B^\delta =_{\mathsf{def}} B|: \delta W \to A \times [0, 1]$, followed by the obvious projection $\phi: A \times [0, 1] \to [0, 1]$. By an arbitrary $C^\infty$-small perturbation of $B$, we may assume that the composition 
 $f =_{\mathsf{def}} \phi \circ B^\delta$ is a Morse function. Since the space $\mathcal F^{\leq 2}$ is open in $\mathcal F$, compactly supported perturbations of smooth $2$-moderate functions or of their families remain $2$-moderate.  Therefore, we may assume that the cobordism $B$ is such that 
 $\phi \circ B^\delta$ is a Morse function.
 
 \begin{figure}[ht]
\centerline{\includegraphics[height=3.5in,width=4in]{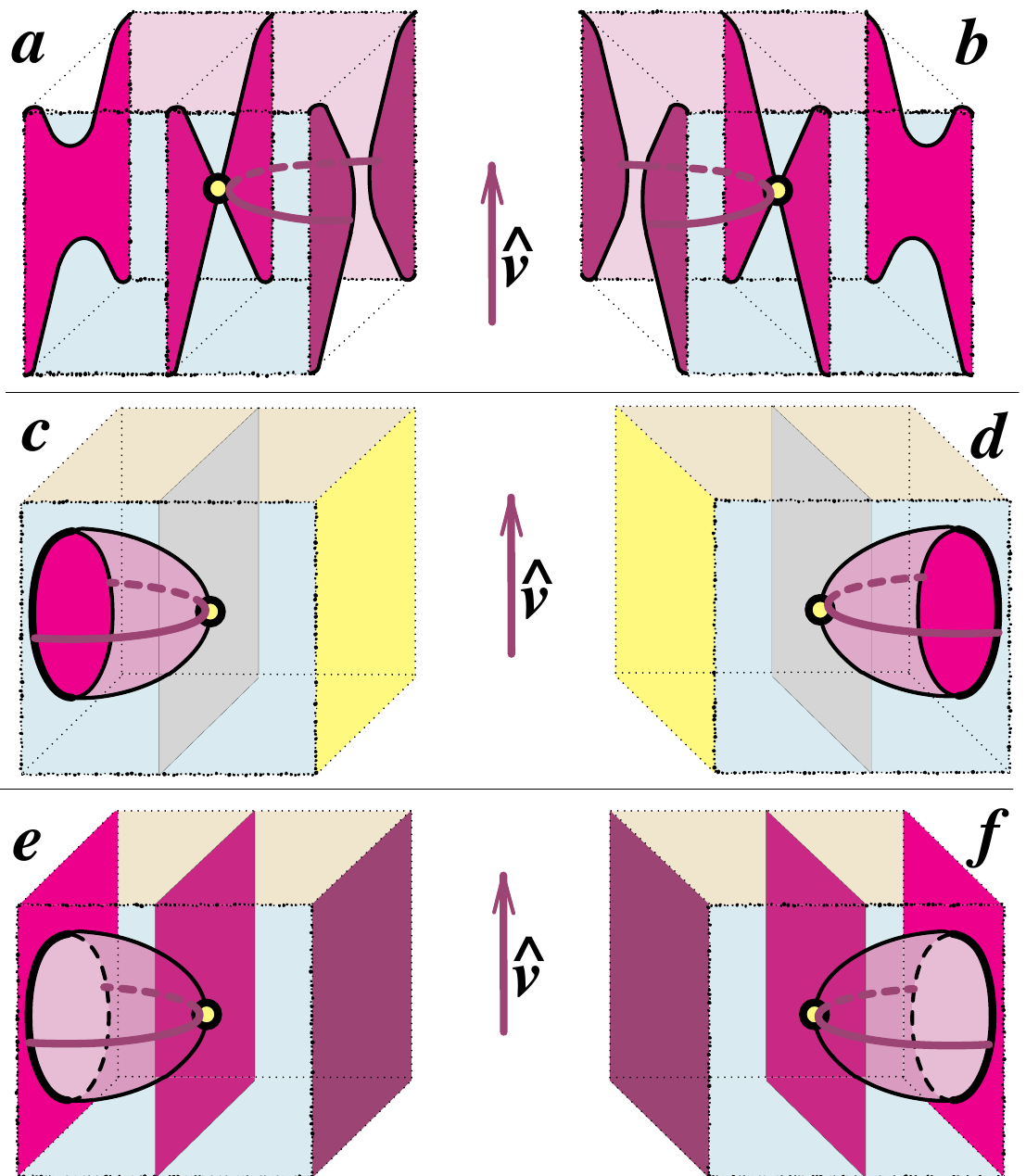}}
\bigskip
\caption{\small{Changing topology of slices $B^{-1}(A \times \{t\})$, as $t$ crosses a critical value $t_\star$ of the Morse function $f: \delta W \to [0, 1]$. Different shades correspond to different slices; each box is shown with $3$ slices. 
In $\mathsf{(a)}$ and $\mathsf{(b)}$, the portion of $B(W)$ over a small interval $[t_\star -\e, t_\star +\e]$ is a pair of solid pants. In $\mathsf{(c), (d)}$, this portion is a solid half-ball. In $\mathsf{(e)}$ and $\mathsf{(f)}$, it is the complements to such half-balls in the solid cube. The figure does not show the complements to solid pants, depicted in $\mathsf{(a)}$ and $\mathsf{(b)}$. Note the ``parabolic locus" (an arc of which is dashed) in $B(\delta W)$, where the vector field $\hat v$ is quadratically tangent to the surface $B(\delta W)$.}}
\label{fig.1.6XX}
\end{figure}

Since, by property (6) from Definition \ref{def1.10_A},  $B: W \to A \times [0, 1]$ is a submersion and  
$f: \delta W \stackrel{B^\delta}{\longrightarrow} A \times [0, 1] \to [0, 1]$ is a Morse function, as $t$ ranges in $[0,1]$, the topology of a regular slice $B^{-1}(A \times \{t\}) \subset W$ may change via a relative, 
elementary surgery only 
when $t$ crosses a critical value $t_\star$ of $f$.  The topology of the slice $B^{-1}(A \times \{t\}) \cap \delta W$ changes via an elementary surgery inside $\delta W$ as $t$ crosses $t_\star$. Thus, there are four types of local surgery, shown in Fig. \ref{fig.1.6XX}: two types of solid pants that correspond to $f$-critical points $(x_\star, t_\star) \in B(\delta W)$ of the Morse index $1$ and two types of ``indented" solids that correspond to $f$-critical points $(x_\star, t_\star) \in B(\delta W)$ of indexes $0$ and $2$.

It is crucial  for our arguments that the geometry of the elementary surgery blocks $W \subset A \times [0, 1]$  is such that the surface $\delta W \subset \d W$ has no tangency to the $\hat v^\bullet$-trajectories of order $3$ or higher (see the parabolic tangency curves with dashed arcs from Fig. \ref{fig.1.6XX}). \smallskip

Let us now describe an algorithm for the elementary moves (surgery)  (see Fig. \ref{fig1.8_X}) 
 that reduces a given pattern $X_0 = J_{z_\a}^{-1}((-\infty, 0]) \subset A$ to a pattern from \emph{the canonical set of patterns} $\{n\cdot K\}_{n \in \Z}$ (as in Fig. 1.9) 
 by a cobordism $B: W \to A \times [0, 1]$ as in Definition \ref{def1.10_A}, a cobordism which is a regular embedding. 
 
 By a small perturbation of $\a_0$, we may assume that no $\hat v$-trajectory has a combinatorial tangency pattern $\om$ with two or more $2$'s.
Then, we select the $\hat v$-trajectories $\hat\g_1, \ldots \hat\g_s$ that contain the tangency patterns $\om = (1, \dots, 1, 2, 1, \dots, 1)$, where the number of $1$'s that precede $2$  is \emph{odd} (which implies that the unique quadratic tangency point on $\hat\g_k$ resides in the \emph{interior} of the set $\hat\g_k \cap \a_0(X_0)$).  These trajectories correspond exactly to the ``concave" points of $\d_2^+(\a_0(X), \hat v)$. The trajectories are listed by the order of their images $[\hat\g_1], \ldots [\hat\g_s]$ in the oriented trajectory space $\mathcal T(\hat v)$. We notice that such trajectories $\hat\g_k$ come in two flavors: ``$\oplus$" and ``$\ominus$", depending on the polarity of the tangency of $\hat\g_k$ to $\a_0(\d X)$.

Then, for each adjacent pair $\hat\g_k, \hat\g_{k+1}$, we choose a $\hat v$-trajectory $\hat\g_k^\star$ in-between $\hat\g_k$ and $\hat\g_{k+1}$. 
Such $\hat\g_k^\star$ is traversal to $\a_0(X_0)$ and its combinatorial type $\om(\hat\g_k^\star)$ is a sequence of $1$'s of a length $2q$. The intersection  $\a_0(X_0) \cap \hat\g_k^\star$ is a disjoint union of closed intervals $I_{k, 1}, \ldots , I_{k, q}$. In the interior of each interval $I_{k, j}$ we pick a point $p_{k, j}^\star$ and will use it as the critical point for a surgery in $A \times [0,1]$ on the solid $W_0 = \a_0(X_0) \times [0, 0.5)$ (see Fig. \ref{fig.1.6XX}, {\sf(a)} and {\sf(b)}). 
If the polarity of $\hat\g_k$ is $\oplus$, we perform a surgery on $W_0$ as in Fig. \ref{fig.1.6XX}, {\sf(a)}, 
if the polarity of $\hat\g_k$ is $\ominus$, we perform a surgery on $W_0$ as in  Fig. \ref{fig.1.6XX}, {\sf(b)}, by attaching solid pants $Z_0$ to $W_0$. We denote by $W_1$ the resulting solid $W_0 \cup Z_0$ in $A \times [0, 1]$ and by $F_1: W_1 \to [0, 1]$ the obvious projection. 

It is essential that, since $F_1$ is a Morse function, the attached pants $Z_0$ do not violate property (5) from Definition \ref{def1.10_A}. Note that each elementary surgery of type {\sf(a)} introduces a pair of new $\hat v^\bullet$-tangent points to the slice $\d (F_1^{-1}(0.5))$; fortunately, none of new tangency pair is of the type $\om = (1, \dots, 1, 2, 1, \dots, 1)$, where the number of $1$'s that precede $2$ is odd (i.e., the pair belongs to $\d_2^-(F_1^{-1}(0.5))$) .

Now the preimage $F_1^{-1}(0.5)$ consists of several connected components $U_1, \ldots U_s$, each of which is a domain 
in $A \times \{0.5\}$ which contains at most a single trajectory of the combinatorial type $(\dots, 1, 2, 1, \dots)$, where the number of $1$'s that precede $2$  is odd and whose polarity is either $\oplus$ or $\ominus$. 

If a connected component $U$ is a ball with several holes, then applying the previous $1$-surgery to a segment of trajectory in $U$ that connects different connected components of $\d U$, we will convert $U$ into a topological $2$-ball without introducing  new trajectories of the combinatorial type $(\dots, 1, 2, 1, \dots)$ with the number of $1$'s that precede $2$ being odd.  Therefore, we may assume that all the connected components $U$ are $2$-balls.

Thus, we divide the balls $U$ in three types: {\sf(i)} $U$ intersects with a single trajectory of the polarity $\oplus$, {\sf(ii)} $U$ intersects with a single trajectory of the polarity $\ominus$, and {\sf(iii)} $U$ does not intersect with trajectories of the combinatorial type $(\dots, 1, 2, 1, \dots)$, where the number of $1$'s that precede $2$  is odd.  

\begin{figure}[ht]
\centerline{\includegraphics[height=2.2in,width=3in]{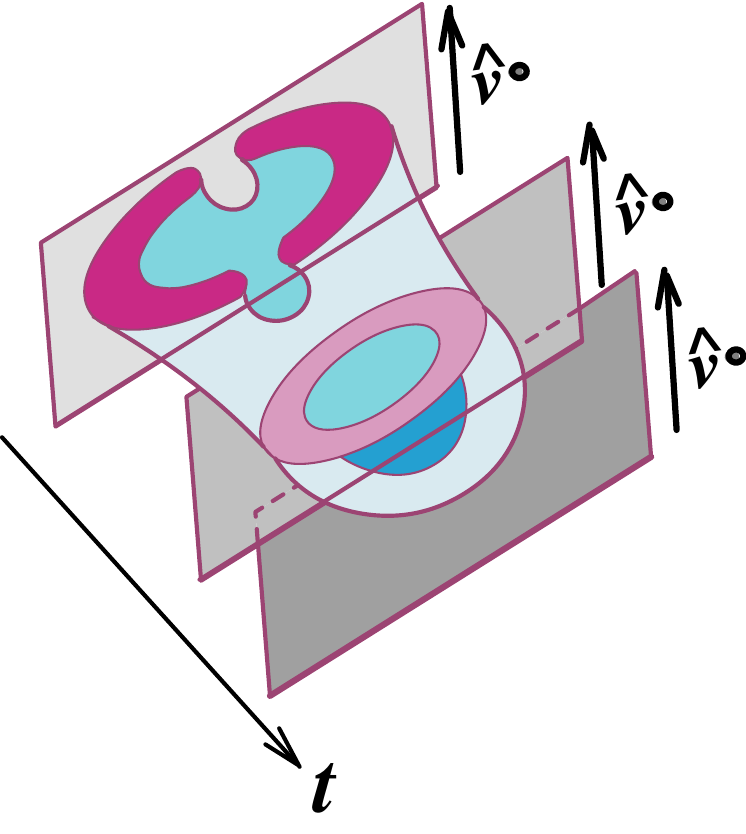}}
\bigskip
\caption{\small{Eliminating a pair of dark-shaded ``kidneys" with their horns facing each other by a surgery: the first surgery transforms the kidneys in the first slice into a ring in the second slice, and then the second surgery transforms the ring  into an empty slice.}}
\label{fig.1.7XX}
\end{figure}

Next, we eliminate the type {\sf(iii)} $\hat v^\bullet$-convex $2$-balls by a $3$-surgery on $W_1$ that amounts to attaching a few $3$-balls to the slice $F_1^{-1}(0.5)$ as in Fig. \ref{fig.1.6XX}, {\sf(c)}. We denote by $W_2$  the resulting solid in $A \times [0, 1]$, and by $F_2:  W_2 \to [0, 1]$ the obvious projection. We may assume that the $F_2$-image of $W_2$ in $[0, 1]$ resides in the interval $[0, 0.7)$. Because the balls of type {\sf(iii)} are $\hat v^\bullet$-convex, we may assume that $\delta W_2 \setminus \delta W_1$ avoids  $\hat v^\bullet$-tangencies of order $\geq 3$.  In particular, property (5) from Definition \ref{def1.10_A} is respected by these surgeries.

We are left now with balls $U$ in $F_2^{-1}(0.7)$ of types  {\sf(i)} and {\sf(ii)} (so called,  {\sf Arnold's `` kidneys"}). Each pair of $2$-balls of the distinct types  {\sf(i)} and {\sf(ii)} can be eliminated by an isotopy in $A \times \{0.7\}$ (which involves scaling and parallel shifts), followed by a surgery as shown in Fig. \ref{fig.1.7XX}. Indeed, if the horns of a pair of kidneys $U$ and $U'$ of different  polarities $\oplus$ and $\ominus$ are facing each other, then by attaching two solid pants we may surgery their union into a ring in a new slice. Then by attaching a relative handle $(D^2_+ \times S^1,\, D^1 \times S^1)$ to $A \times \{0\}$ in $A \times [0, 1]$, we eliminate the ring. If the horns of  $U$ and $U'$ of different polarities are facing in opposite directions, then by an isotopy in $A$ (as in Lemma \ref{lem.1.5XY}) we may switch their order along the $\mathcal T(\hat v)$-direction. The switching will reduce the situation to the case when the horns face each other. Again, this surgery will not introduce new points of the combinatorial types $(\dots 121 \dots)$, where the number of $1$'s that precede $2$  is odd. Thus, we will be left with a disjoint union of kidneys of the \emph{same} polarity $\oplus$ or $\ominus$. \smallskip

Let $\a: D^2 \to A$ be the regular imbedding realizing a $\oplus$-polarized kidney. By \cite{Ar}, under the map $J^\a$, each kidney is mapped to a non-contractible loop $\Phi_\a$ in $\mathcal F^{\leq 2}$, 
a generator of $\pi_1(\mathcal F^{\leq 2}, \mathbf 1) \approx \Z$. Therefore, no further simplifications among the remaining kidneys are possible.

Thus, the difference between the numbers of kidneys of types {\sf(i)} and {\sf(ii)} is an invariant of the bordism class, introduced in Definition \ref{def1.10_A}. This difference equals to the original difference $\#\{\d_2^{+, \oplus}X(v)\} - \#\{\d_2^{+, \ominus}X(v)\}$ between the numbers of $\hat v$-trajectories with polarities $\oplus$ and $\ominus$ and was preserved under all the surgeries that led to the final slice with the kidneys of the same polarity. Thus, the number
$
\mathsf J^\a =  \#\{\d_2^{+, \oplus}X(v)\} - \#\{\d_2^{+, \ominus}X(v)\}.
$ 
did not change under all the surgeries above.\smallskip
%
%
%
%


Now we are ready to investigate the $2$-moderate bordisms $\mathbf{O B}^{\mathsf{emb}}_{\mathsf{moderate \leq 2}}(A)$ of embedded \emph{oriented} blobs.
Each  connected blob can be oriented {\sf counterclockwise}, i.e., coherently with the fixed orientation of the ambient $A$, or {\sf clockwise},  i.e., opposite to the fixed orientation of $A$. We aim to show that the numbers $\mathsf K^\a, \mathsf L^\a$ from (\ref{eq.NEW_ORIENTED_invariants}) define an isomorphism (\ref{eq.ORIENT_BLOB_EMBED}).
\smallskip

It is possible to attach \emph{embedded/immersed} $1$-handles in $A$ only to  \emph{similarly oriented} blobs: the ``twisted $1$-handles" cannot be immersed in $A$. 

We notice  (see Fig. \ref{fig1.8_X}, the lower diagram) that attaching a narrow $1$-handle, whose core is transversal to the field $\hat v$, to similarly oriented blobs  contributes a new pair of points of the polarities $(+, \oplus, \downarrow)$ and  $(+, \ominus, \uparrow)$ (the counterclockwise oriented blobs), or a new pair of points of the polarities $(+, \oplus, \uparrow)$ and  $(+, \ominus, \downarrow)$ (the clockwise oriented blobs). Deleting a  $1$-handle contributes a new pair of points of the polarities $(-, \ominus, \uparrow)$ and  $(-, \oplus, \downarrow)$ (the counterclockwise oriented blobs), or $(-, \ominus, \downarrow)$ and  $(-, \oplus, \uparrow)$ (the clockwise oriented blobs); however, these changes affect only the convex loci $\d_2^-(\sim)$ which play secondary roles. Thus, all these surgeries do not change the values $\mathsf K^\a, \mathsf L^\a$. Similarly, adding or deleting a $2$-ball $D$ such that $\d_2^+D(\hat v) = \emptyset$, does not affect the values of $\mathsf K^\a, \mathsf L^\a$. \smallskip 

Thus, we will treat the counterclockwise and clockwise oriented blobs separately, recycling the our arguments in the non-oriented case.  Each embedded connected blob $\a(X_0)$ is a disk or a disk with a number of holes. By deleting narrow $1$-handles from $\a(X_0)$ as in Fig. \ref{fig.1.6XX}, $\mathsf{(a)}$, or as in Fig. \ref{fig1.8_X}, we replace, via $2$-moderate bordisms, each counterclockwise/clockwise oriented blob with holes by a similarly oriented topological $2$-ball $D_0$.  In the process, we keep the original values $\mathsf K^\a, \mathsf L^\a$. The tangency of $\d D_0$ to $\hat v$ is still $2$-moderate.

Now all the connected components of the new embedded surface $Y \subset A$ are $2$-balls. As in the non-oriented case, by deleting $1$-handles, each ball can be split into a number of balls, each of which has a single point from $\d_2^+(\sim)$ at most.  By a $2$-surgery, we delete the balls which are convex with respect to $\hat v$ (i.e., they do not have singletons from $\d_2^+(\sim)$). Now we a left with the balls that have a single point of the polarity from the list (\ref{eq.FOUR_FLAVORS}). Some of them are oriented counterclockwise, others clockwise.

Thus, 
\begin{eqnarray} \label{eq.EMB_ORIENT}
 \mathsf K \times \mathsf L:\, \mathbf{OB}^{\mathsf{emb}}_{\mathsf{moderate \leq 2}}(A) \longrightarrow \Z \times \Z. 
\end{eqnarray}
is an epimorphism. 

Assuming that $\mathsf K^\a = 0 = \mathsf L^\a$ for some $\a$, we see that this property allows to pair all the counterclockwise/clockwise oriented kidneys so that each pair by $1$-surgery can be transformed into a ring and then eliminated as in Fig. \ref{fig.1.7XX}. Therefore, $\mathsf K \times \mathsf L$ is a monomorphism. 
\hfill
\end{proof}

 \begin{figure}[ht]
\centerline{\includegraphics[height=2.2in,width=4.2in]{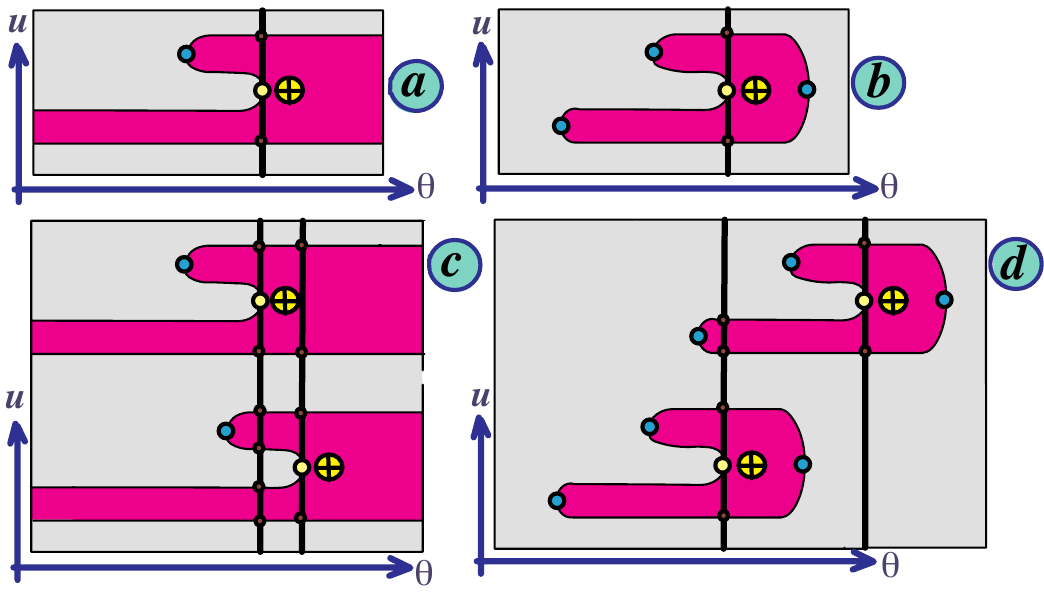}}
\bigskip
\caption{\small{Two portraits of a generator $\kappa \in \pi_1(\mathcal F^{\leq 2}, p)$, where the base point $p \in \mathcal F^{\leq 2}$ is modeled after the polynomial $p(u) = u^4- 1$ in diagram {\sf (a)} and by $p(u) = u^4 +1$ in diagram {\sf(b)}. Diagrams {\sf(c)} and {\sf(d)} portray $2\kappa$. In diagrams {\sf(a)} and {\sf(c)} that represent the case $\mathcal T(\hat v) = S^1$, the left and the right edges of the rectangle should be identified so that the shaded regions match.  
Note the polarity $\oplus$ of the tangent $\hat v$-trajectories with the combinatorial pattern $\om = (\dots 121 \dots)$, where the number of $1$'s that precede $2$  is odd.}}
\label{fig.1.9_X}
\end{figure}

As in the case of non-oriented blobs, using formulas (\ref{eq.NEW_ORIENTED_invariants}), we get a surjective map 
\begin{eqnarray}\label{eq.INVERSE_to_IMM_to_EMB}
\mathcal R_{\mathsf K, \mathsf L}:\; \mathbf{OB}^{\mathsf{imm}}_{\mathsf{moderate \leq 2}}(A) \to   \Z \times \Z \approx \mathbf {OB}^{\mathsf{emb}}_{\mathsf{moderate \leq 2}}(A),
\end{eqnarray} 
which serves as {\sf the right inverse} to the obvious homomorphism
\[
 \mathbf {OB}^{\mathsf{emb}}_{\mathsf{moderate \leq 2}}(A) \to 
\mathbf{OB}^{\mathsf{imm}}_{\mathsf{moderate \leq 2}}(A).
\]
Therefore, $\mathbf{OB}^{\mathsf{imm}}_{\mathsf{moderate \leq 2}}(A)$ contains canonically  $\Z \times \Z$.\smallskip

\noindent {\bf Remark 2.1.}
The number $c^+(v) =_{\mathsf{def}} \#\{\d_2^{+, \oplus}(v)\} +  \#\{\d_2^{+, \ominus}(v)\}$ may be interpreted as the \emph{complexity} of the vector field $v$ on $X$ \cite{K4}, \cite{K7}. 
We notice that 
$$c^+(v)  \geq \quad |\mathsf J^\a | =_{\mathsf{def}}  | \#\{\d_2^{+, \oplus}(v)\} - \#\{\d_2^{+, \ominus}(v)\}|. $$  

It is somewhat surprising  that the invariant $\mathsf J^\a =  \#\{\d_2^{+, \oplus}(v)\} - \#\{\d_2^{+, \ominus}(v)\}$ reflects more the local topology of the embedding $\a$ (or of the field $v= \a^\ast(\hat v)$)  than the global topology of the surface $X$: in fact,  any integral value of $\mathsf J^\a$ can be realized by a traversally generic field $v$ on a $2$-ball $D$ which even admits a convex envelop! A portion of  the boundary $\d D$ looks like a snake with respect to the field $\hat v$ of the envelop. \smallskip

For any $X$, the effect of deforming a portion of the boundary $\d X \subset A$ into a snake is equivalent to adding several spikes (an edge and a pair of a univalent and a trivalent verticies) to the graph $\mathcal T(v)$, the space of $v$-trajectories. Evidently, these operations do not affect $H_1(\mathcal T(v); \Z) \approx H_1(X; \Z)$.

For example, for $\a$ as in Figure \ref{fig.1.4A}, 
$\mathsf J^\a =0$. If we subject $\a$ to an isotopy in $A$ that introduces a snake-like pattern of Fig. \ref{fig.1.9_X}, 
{\sf(a)}, then for the new immersion $\a'$, the invariant $\mathsf J^{\a'} = 1$.\smallskip

In contrast,  the number $\#\{\d_2^{+, \oplus}(v)\} + \#\{\d_2^{+, \ominus}(v)\}$ has a topological significance for $X$. If $X$ is the compliment to $k$ disjoint balls in a closed orientable surface with $g$ handles, then by \cite{K7}, Lemma 1.1, $\#\{\d_2^{+, \oplus}(v)\} + \#\{\d_2^{+, \ominus}(v)\} \geq 4g-4 +2k.$ 
 \hfill $\diamondsuit$

\begin{figure}[ht]
\centerline{\includegraphics[height=3.5in,width=3.5in]{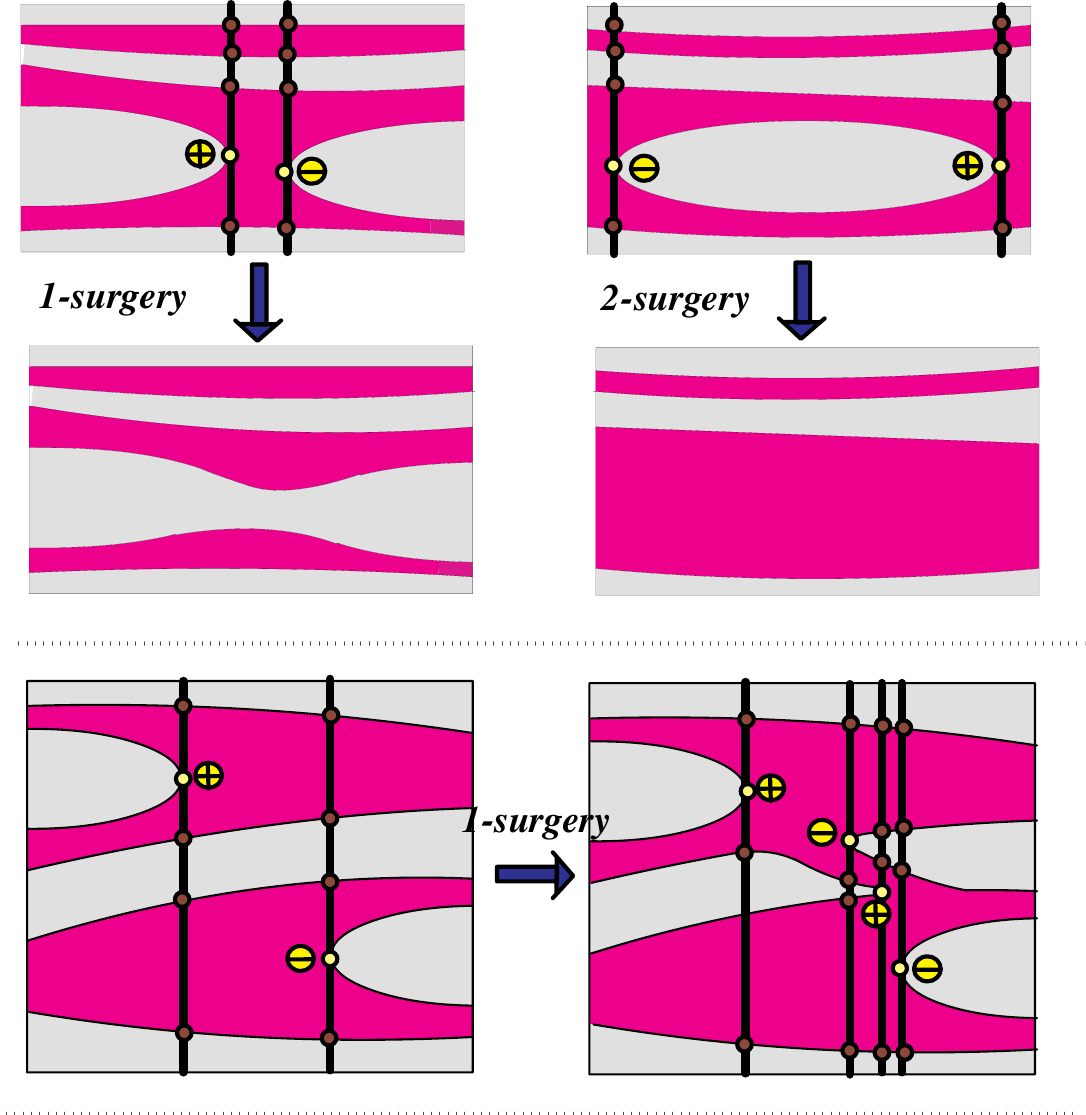}}
\bigskip
\caption{\small{The region $\a(X)$ in the cylinder $A = \R \times S^1$ is shaded, the vector field $\hat v$ is vertical.  Elementary cancellations,  via surgery in $A$, of a pair  tangencies of $\a(\d X)$ to $\hat v$ of opposite polarities $\oplus, \ominus$ (the upper diagrams). Increasing the local connectivity of the region $\a(X)$ between two tangent trajectories of opposite polarities by a $1$-surgery in $A$ (the lower diagram). This operation introduces a new pair of points of opposite polarities $\oplus, \ominus$.}}
\label{fig1.8_X}
\end{figure}
\smallskip

\section{Invariants that distinguish between bordisms of immersions and embeddings}
  
Recall that $\mathcal A^\times: \mathbf B^{\mathsf{\times imm}}_{\mathsf{moderate \leq 2}}(A) \to \mathbf B^{\mathsf{imm}}_{\mathsf{moderate \leq 2}}(A)$ is a bijection. Therefore, without lost of generality, we deal now with immersions under additional assumption (\ref{eq.extra_transversality}): we require that their self-intersections are {\sf transversal}. We adopt the notations from Definition \ref{def1.10_A}.\smallskip

\begin{figure}[ht]
\centerline{\includegraphics[height=1.1in,width=3.5in]{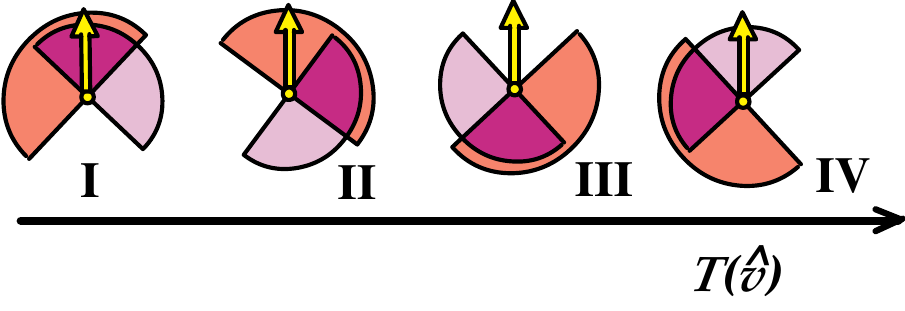}}
\bigskip
\caption{\small{Four configurations {\sf I, II, III, IV} in the vicinity of a point $a \in A$ where two branches of $\a(\d X)$ intersect transversally. The figure shows the $\a$-images of  neighborhoods in $X$ of the two points  $b_1, b_2 \in \a^{-1}(a) \cap \d X$. Note the position of the vector $\hat v(a)$ relative to the darkly shaded sector. The trajectory space $\mathcal T(\hat v)$ is oriented. 
}}
\label{fig.1.9XX}
\end{figure}

Consider an immersion $\a: X \to A$ with transversal  self-intersections of $\a(\d X)$ and no multiple self-intersections of more than two local branches. At each self-intersection $a \in  \a(\d X)$, there are exactly two points $a', a'' \in \d X$ such that $\a(a') = \a(a'') = a$. Then we take two small neighborhoods $D_+(a')$ and $D_+(a'')$ of $a'$ and $a''$, respectively, both diffeomorphic to a half-disk. The linearizations of their $\a$-images produces two preferred half-spaces, $H_{a'}$ and $H_{a''}$, in the tangent space $T_aA$. Put $K_{a', a''} = H_{a'} \cap H_{a''}$. We denote by $K_{a', a''}^\vee$ the image of $K_{a', a''}$ under the central symmetry in $T_aA$  (see Fig. \ref{fig.1.9XX}).

Using the sectors $K_{a', a''}$ and $K_{a', a''}^\vee$, we will divide the self-intersections $\{a\}$ into four types: {\sf type} {\bf  I} occurs when the vector $\hat v(a)$ points in the interior of the sector $K_{a', a''}$, {\sf type} {\bf III} occurs when the vector $\hat v(a)$ points in the interior of the sector $K_{a', a''}^\vee$; {\sf type} {\bf II} and {\sf type} {\bf IV} arise when $\hat v(a)$ points in the interior of the complimentary to $K_{a', a''} \cup K_{a', a''}^\vee$ sectors. {\sf Type} {\bf  II} corresponds to the case when the sector $K_{a', a''}$ is on the right of the line $\g_a$, while {\sf type} {\bf IV} corresponds to the case when the sector $K_{a', a''}$ is on the left of the line $\g_a$. 
Note that the types $\mathbf{II}$ and $\mathbf{IV}$ are ``crudely symmetric" with respect to a reflection in $A$ that has the $\hat v$-trajectory through $a$ as a line of symmetry $\g_a$.\smallskip

Consider a small disk $D$, centered on a crossing point $a$. We consider a diffeotopy $\{\Psi_\phi: A \to A\}_\phi$ whose final stage is the turn on a given angle $\phi_\star$ inside the concentric disk $D' \subset D$ and the identity outside $D$. Note that using $\Psi_\phi$ as a dial at $a$, we can change the type \{{\bf I, II, III, IV}\} of $\Psi_{\phi_\star}(\a(\d X))$ at $a$ at will. However, $\a_{\phi_\star}$, the composition  of $\a$ with $\Psi_{\phi_\star}$, is not $2$-moderately cobordant to $\a$: in the process, the curves $\a_\phi(\d X)$ may develop inflection points (cubic tangencies) to the foliation $\mathcal F(\hat v)$. Moreover, the $2$-moderate $\a_{\phi_\star}(\d X)$ may develop inside $D$ new quadratic tangencies to $\mathcal F(\hat v)$. Thus, there is a subtle interplay between the count of crossings of a particular type from \{{\bf I, II, III, IV}\} and the invariants of the type $J^\a$.

Let us find out which of these four types \{{\bf I, II, III, IV}\} or their combinations are invariant under non-oriented cobordisms of immersions with $2$-moderate tangencies.

We start with a key observation that guides us in this section.

Given a cobordism-immersion $B: W \to A \times [0, 1]$ between two immersions, $\a_0: X_0 \to A \times \{0\}$ and $\a_1: X_1 \to A \times \{1\}$,  with transversal self-intersection of $B(\delta W)$ of multiplicity $2$ at most, consider the curve $\ell$ where the surface $B(\delta W) \subset A \times [0, 1]$ self-intersects transversally.  Either $\ell$ is:  

{\sf(a)} a simple segment in $A \times [0, 1]$ that connects a pair of crossings of $\a_0(\d X_0)$, or 

{\sf(b)} a pair of crossings of $\a_1(\d X_1)$, or 

{\sf(c)} a pair of crossings, one of which belongs to $\a_0(\d X_0)$ and the other one to $\a_1(\d X_1)$, or 

{\sf(d)} a simple loop in $A \times (0, 1)$. 

The multiplicity of each intersection of $\ell$ with a $\hat v^\bullet$-trajectory $\hat\g \subset A \times [0, 1]$ is at least $2$, and, according to Definition \ref{def1.10_A}, cannot exceed $2$. Therefore, each $\hat v^\bullet$-trajectory must be \emph{transversal} to $\ell$: otherwise, the multiplicity of the intersection $\ell \cap \hat\g$, where $\hat\g$ is tangent to $\ell$, exceeds $2$. 

For each point $x \in \ell$, the two $\ell$-localized branches, $B(\delta W_1)$ and $B(\delta W_2)$, of $B(\delta W)$ that intersect along $\ell$ divide the tubular neighborhood $U_\ell$ of $\ell$ into four chambers $S_x^{\mathbf{I}}, S_x^{\mathbf{II}}, S_x^{\mathbf{III}}, \hfill\break S_x^{\mathbf{IV}}$, where each point from the interior of $S_x^{\mathbf{I}}$ has two $\a$-preimages  in the vicinity of $\delta W_1 \cup \delta W_2$ in $W$,  each point from the interior of $S_x^{\mathbf{III}}$ has no $\a$-preimages in the vicinity of $\delta W_1 \cup \delta W_2$ in $W$, and each point from the interior of $S_x^{\mathbf{II}} \cup S_x^{\mathbf{IV}}$ has one $\a$-preimage in the vicinity of $\delta W_1 \cup \delta W_2$ in $W$. 
Thanks to the transversality of the vector field $\hat v^\bullet$ to the curve $\ell$, the vector field  must point into the interior of one of the chambers; i.e., $\hat v^\bullet$ is not tangent to the branches $B(\delta W_1)$ and $B(\delta W_2)$. In $3D$, we cannot distinguish between the chambers $S_x^{\mathbf{II}}$ and $S_x^{\mathbf{IV}}$ without picking an orientation of $\ell$. 
Thus, $S_x^{\mathbf{I}}$,  $S_x^{\mathbf{III}}$, and $S_x^{\mathbf{II}} \cup S_x^{\mathbf{IV}}$ are geometrically distinct. 

Now $\hat v^\bullet$ helps to define the three types of $\ell$ (not to be confused with the four types {\sf(a)}-{\sf(d)}): if $\hat v^\bullet(x)$ points into $S_x^{\mathbf{I}}$ or $S_x^{\mathbf{III}}$, we say that the curve $\ell$ is of types ${\mathbf{I}}$ or ${\mathbf{III}}$, respectively. Otherwise, $\ell$ is of the ``mixed type" ${\mathbf{II}}\, \&\,  {\mathbf{IV}}$.

In all cases {\sf(a)}-{\sf(d)}, $\hat v^\bullet$ preserves  the three champer types ${\mathbf{I}}, {\mathbf{III}}, {\mathbf{II}}\, \&\, {\mathbf{IV}}$ along $\ell$. In fact, the case {\sf(d)}, in which $\ell$ is a loop, is irrelevant for what follows. 

Let $x \cup y = \d\ell$. Crucially, in cases {\sf(a)}-{\sf(b)} the type $S_x^{\mathbf{II}}$ turns into the type $S_y^{\mathbf{IV}}$, and $S_x^{\mathbf{IV}}$ turns into the type $S_y^{\mathbf{II}}$, while the types $S_x^{\mathbf{I}}$  and $S_y^{\mathbf{I}}$, $S_x^{\mathbf{III}}$  and $S_y^{\mathbf{III}}$ are the same.  

Therefore, in cases {\sf(a)}-{\sf(b)} we get the following possible combinations: $S_x^{\mathbf{I}}$ and $S_y^{\mathbf{I}}$, or $S_x^{\mathbf{II}}$ and $S_y^{\mathbf{IV}}$, or $S_x^{\mathbf{III}}$ and $S_y^{\mathbf{III}}$, or $S_x^{\mathbf{IV}}$ and $S_y^{\mathbf{II}}$. In the case {\sf(c)}, all the types at $x$ and $y$ are the same.

Given an immersion $\a: X \to A$ as in Definition \ref{def1.10_A}, let us denote by $\rho_{\mathbf{I}}(\a), \rho_{\mathbf{II}}(\a), \rho_{\mathbf{III}}(\a)$, $\rho_{\mathbf{IV}}(\a)$ the number of crossings of $\a(\d X)$ of the types $\bf{I, II, III, IV}$, where the type of a crossing is determined by the configuration of the sector, formed by the two intersecting domains, and the vector $\hat v$ at the crossing.

 For example, in Fig. \ref{fig.1.4A}, all the four crossings of $\a(\d X)$ are of the different types $\mathbf{I, II, III, IV}$; that is, $\rho_{\mathbf{I}}(\a) =1,\, \rho_{\mathbf{II}}(\a) =1,\, \rho_{\mathbf{III}}(\a) =1,\, \rho_{\mathbf{IV}}(\a) =1$.




These pairings between different crossings of $\a_0(\d X_0)$ and $\a_1(\d X_1)$, which are delivered by the curves $\{\ell\}$ of types {\sf (a), (b), (c)}, lead directly to the following conclusion.
\begin{lemma}\label{lem.1.5_XY}
If two $2$-moderate immersions $\a_0: X_0 \to A$ and $\a_1: X_1 \to A$ are cobordant in $\mathbf B^{\mathsf{\times imm}}_{\mathsf{moderate \leq 2}}(A) \approx \mathbf B^{\mathsf{imm}}_{\mathsf{moderate \leq 2}}(A)$, then
\begin{eqnarray}\label{eq.1.8 XYZ} 
\rho_{\mathbf{I}}(\a_0) & \equiv & \rho_{\mathbf{I}}(\a_1) \mod 2, \\
\rho_{\mathbf{III}}(\a_0) & \equiv & \rho_{\mathbf{III}}(\a_1) \mod 2, \nonumber \\
\rho_{\mathbf{II}}(\a_0) - \rho_{\mathbf{IV}}(\a_0) & = & \rho_{\mathbf{II}}(\a_1) - \rho_{\mathbf{IV}}(\a_1). \quad \quad \diamondsuit \nonumber
 \end{eqnarray}
\end{lemma}

Thus, we produced three numerical invariants that have a potential to distinguish between the elements of the bordism $\mathbf B^{\mathsf{imm}}_{\mathsf{moderate \leq 2}}(A)$, modulo the ones that are in the image of the map $\mathbf B^{\mathsf{emb}}_{\mathsf{moderate \leq 2}}(A) \to \mathbf B^{\mathsf{imm}}_{\mathsf{moderate \leq 2}}(A).$

\begin{figure}[ht]
\centerline{\includegraphics[height=1.7in,width=3in]{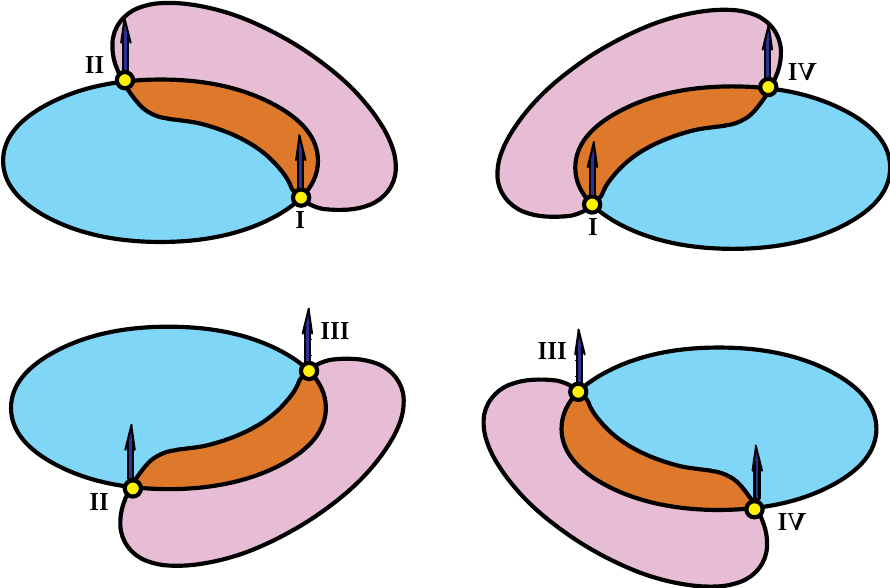}}
\bigskip
\caption{\small{Basic immersions of two disks in the surface $A$ and the intersection patterns $\mathbf{I, II, III, IV}$ from Fig. \ref{fig.1.9XX} they generate. As elements of $\mathbf B^{\mathsf{imm}}_{\mathsf{moderate \leq 2}}(A)$, the immersion on the left is minus the corresponding immersion on the right.}}
\label{fig.1.13XX}
\end{figure}

\begin{proposition}\label{prop.1.EXACT_SEQ_BLOBS}
For $A = \R \times [0, 1]$, the invariants $\rho_{\mathbf{I}}(\a),\; \rho_{\mathbf{III}}(\a) \in \Z_2,\; \rho_{\mathbf{II}}(\a) - \rho_{\mathbf{IV}}(\a) \in \Z$ define a homomorphism 
$$\mathcal I \rho :\; \mathbf B^{\mathsf{imm}}_{\mathsf{moderate \leq 2}}(A)/ \mathbf B^{\mathsf{emb}}_{\mathsf{moderate \leq 2}}(A) \to \Z_2 \times \Z_2 \times \Z$$
whose image is an abelian subgroup $\mathcal M \approx \Z_2 \times \Z$ of index $2$. 
 The monomorphism $\phi: \mathcal M \hookrightarrow \Z_2 \times \Z_2 \times \Z$ is delivered by the diagonal map $\Z_2 \to \Z_2 \times \Z_2$ on $\Z_2$ and $\phi(\Z) = 2\Z$. 
\end{proposition}

\begin{proof} The argument is based on the proof of Lemma \ref{lem.1.5_XY} and on formulas (\ref{eq.1.8 XYZ}).

Examining the types of crossings in the four diagrams from Fig. \ref{fig.1.13XX} that show the four immersions $\a_1, \a_2, \a_3, \a_4$, we get: 
 %
 \begin{eqnarray}\label{eq.1.9_XY}
 \rho_{\mathbf{I}}(\a_1) = 1, \; & \rho_{\mathbf{III}}(\a_1) = 0,\; & \rho_{\mathbf{II}}(\a_1) - \rho_{IV}(\a_1) = 1, \\
 \rho_{\mathbf{I}}(\a_2) = 1, \; & \rho_{\mathbf{III}}(\a_2) = 0,\; & \rho_{\mathbf{II}}(\a_2) - \rho_{\mathbf{IV}}(\a_2) = -1, \nonumber \\
 \rho_{\mathbf{I}}(\a_3) = 0, \; & \rho_{\mathbf{III}}(\a_3) = 1,\; & \rho_{\mathbf{II}}(\a_3) - \rho_{\mathbf{IV}}(\a_3) =  1,\nonumber \\
 \rho_{\mathbf{I}}(\a_4) = 0, \; & \rho_{\mathbf{III}}(\a_4) = 1,\; & \rho_{\mathbf{II}}(\a_4) - \rho_{\mathbf{IV}}(\a_4) = - 1. \nonumber 
 \end{eqnarray}

Note that in $\mathbf B^{\mathsf{imm}}_{\mathsf{moderate \leq 2}}(A)$,  $\a_2 = - \a_1$ and $\a_4 = - \a_3$. \smallskip

 Also, for $\a: X \to A$ from Fig. \ref{fig.1.4A}, where $X$ is a torus with a hole,  we have $\rho_{\mathbf{I}}(\a) = 1$, $\rho_{\mathbf{III}}(\a) = 1$, and $\rho_{\mathbf{II}}(\a_4) - \rho_{\mathbf{IV}}(\a_4) = 0$. 
 \smallskip
 
Consider now just these three immersions: $\a_1, \a_3$, and $\a$.
By linear algebra, applied to $\rho_{\mathbf{I}}(\sim), \rho_{\mathbf{III}}(\sim)$, and $\rho_{\mathbf{II}}(\sim) - \rho_{\mathbf{IV}}(\sim)$  
invariants of $\a_1, \a_3$, and $\a$, the image of the map $$\mathcal I \rho:\; \mathbf B^{\mathsf{imm}}_{\mathsf{moderate \leq 2}}(A) \to \Z_2 \times \Z_2 \times \Z$$ is a subgroup of index $2$ at most,  which contains the $\Z$-module $\mathcal M$, spanned by the triples $(e, 0; 1), (0, e; 1), (e, e; 0)$, where $2e=0$. 

\begin{lemma}\label{lem.ODD_self-intersection}
No loop $\mathcal C$ in $A = \R \times [0, 1]$ that bounds an immersed orientable surface $\Sigma$ has an \emph{odd} self-intersection number. 
\end{lemma}

\begin{proof} Indeed, if an immersion $\a: \Sigma \to A$ with $\a(\d \Sigma) = \mathcal C$ exists, then the pull-back vector field $v = \a^\dagger(\hat v) \neq 0$. Thus, its index $\mathsf{ind}(v) = 0$.  Recall that $deg(G)$, the degree of the Gauss normal map $G: \a(\d \Sigma)  \to S^1$ can be calculated by the Hopf-Gottlieb formula (\cite{H}, \cite{Got}): $deg(G) = \chi(\Sigma) - \mathsf{ind}(v) =  \chi(\Sigma) = 1 - \b_1(\Sigma)$, where the first Betti number $\b_1(\Sigma) = 2g$ is even. On the other hand, by the Whiney formula \cite{Wh} $deg(G) = \mu + N^+ - N^-$, where $\mu = \pm 1$ and $N^+$, $N^-$ count the positive and negative crossings of $\mathcal C$ (this polarity of crossings is based on ``global" considerations). So, each ``kink" of $\mathcal C$ adds $\pm 1$ to $deg(G)$. Odd number of self-crossings $N^+ + N^-$, by the Whiney formula, produces an even degree $deg(G)$, which contradicts to the fact that $1 - \b_1(\Sigma)$ is an odd number. In fact, the minimal number of self-intersections an orientable immersed surface of genus $g$ with a circular boundary may have is $2g+2$ \cite{Gu}. \hfill 
\end{proof}

Thus, by Lemma \ref{lem.ODD_self-intersection}, none of the elements $(e, 0, 0), (0, e, 0), (0, 0, 1) \in \Z_2 \times \Z_2 \times \Z$ resides in the image of $\mathcal I \rho$. However, the elements $(e, 0, 1), (0, e, 1), (0, 0, 2) \in \Z_2 \times \Z_2 \times \Z$ do belong to the $\mathcal I\rho$-image. Therefore, the $\mathcal I\rho$-image is $\mathcal M =_{\mathsf{def}} \mathsf{span}_\Z\{(e, 0, 1), (0, e, 1), (0, 0, 2)\} \approx \Z_2 \times \Z$. 

Evidently, $\mathbf B^{\mathsf{emb}}_{\mathsf{moderate \leq 2}}(A) \approx \Z$ is in the kernel of $\mathcal I \rho$.  This completes the proof of Proposition \ref{prop.1.EXACT_SEQ_BLOBS}.
\hfill
\end{proof}


 
\smallskip

Note that \emph{oriented} immersed doodles also generate patterns similar to the ones in Fig. \ref{fig.1.9XX} at each intersection point $a \in A$. Indeed, let $\b: \mathcal C \to A$ be an immersion as in (\ref{eq1.4_A}).  Consider the two local branches $\mathcal C'_a$ and $\mathcal C''_a$ of $\b(\mathcal C)$ at $a$ that are transversal and oriented. The pair of tangent vectors $u'_a$  and $u''_a$ to $\mathcal C'_a$ and to $\mathcal C''_a$  that are consistent with the orientations, generate a particular \emph{sector} in $T_a A$, the convex hull of $u'_a$  and $u''_a$ (see Fig. \ref{fig.4.8888XX}, {\sf a}).  

By counting the four types of sectors in relation to $\hat v(x)$, which an oriented immersion $\b$ generates at its crossings $\{a\}$, we get the new quantities $\rho_{\mathsf{I}}(\b),\;\rho_{\mathsf{II}}(\b), \; \rho_{\mathsf{III}}(\b),\; \rho_{\mathsf{IV}}(\b)$. The sector's number  may be different by a permutation of four elements from the darkly shaded sectors in Fig. \ref{fig.1.9XX}. 

Therefore, the similar kind of invariants  $\rho_{\mathsf{I}}(\b) - \rho_{\mathsf{III}}(\b),\; \rho_{\mathsf{II}}(\b) - \rho_{\mathsf{IV}}(\b) \in \Z$ are available for the oriented doodles.  \smallskip
 
 As in the case of blobs, these three quantities are invariants of the quotient \hfill\break
 $\mathbf{OC}^{\mathsf{imm}}_{\mathsf{moderate \leq 2}}(A)/ \mathbf{OC}^{\mathsf{emb}}_{\mathsf{moderate \leq 2}}(A)$. The arguments that validate this claim are exactly the same as the ones that led to Lemma \ref{lem.1.5_XY} for blobs. Moreover, for $A = \R \times [0,1]$, these invariants are additive with respect to the connected sums $\uplus$ of oriented doodles.     
 
 \begin{proposition}\label{prop.1.3.1_XY}
For any oriented doodle $\b: \mathcal C \to A$, the invariants $\rho_{\mathsf{I}}(\b) -   \rho_{\mathsf{III}}(\b),\,   \rho_{\mathsf{II}}(\b) - \rho_{\mathsf{IV}}(\b) \in \Z$ define a surjective map  
$$\mathcal I \rho : \mathbf{OC}^{\mathsf{imm}}_{\mathsf{moderate \leq 2}}(A)/ \mathbf{OC}^{\mathsf{emb}}_{\mathsf{moderate \leq 2}}(A) \to \Z \times  \Z.$$

For $A = \R \times [0, 1]$, the map $\mathcal I \rho$ is an epimorphism with respect to the group operation $\uplus$. 
\end{proposition}

\begin{proof} We recycle the list of $\rho$-invariants for the oriented boundaries of submersions $\a_1, \a_3$, and $\a$ from the proof of Proposition \ref{prop.1.EXACT_SEQ_BLOBS} and add to them the $\rho$-invariants of new doodles that do not bound immersions of surfaces.  Differently oriented figures ``$\infty$", placed in $A$ horizontally with respect to the vector field $\hat v$ by  immersions $\b_+, \b_-: S^1 \to A$, realize the values $\{\rho_{\mathsf{I}}(\b_+) = 1,\; \rho_{\mathsf{III}}(\b_+)= 0,\; \rho_{\mathsf{II}}(\b_+) - \rho_{\mathsf{IV}}(\b_+) =0\}$ and $\{\rho_{\mathsf{I}}(\b_-) = 0,\; \rho_{\mathsf{III}}(\b_-)= 1,\; \rho_{\mathsf{II}}(\b_-) - \rho_{\mathsf{IV}}(\b_-) =0\}$. Consider another oriented doodle $\tilde\b: S^1 \to A$ (not a boundary of an immersion) with a single self-intersection whose image is a ``loop within a loop", symmetric with respect to the $\hat v$-trajectory through the self-intersection point. It realizes the values  $\rho_{\mathsf{I}}(\tilde\b) =0,\; \rho_{\mathsf{III}}(\tilde\b) = 0 ,\; \rho_{\mathsf{II}}(\tilde\b) =1, \rho_{\mathsf{IV}}(\tilde\b) =0$. Flipping the orientation of $S^1$, we realize the values  $\rho_{\mathsf{I}}(\bar\b) =0,\; \rho_{\mathsf{III}}(\bar\b) = 0 ,\; \rho_{\mathsf{II}}(\bar\b) =0 , \rho_{\mathsf{IV}}(\bar\b) =1$. 
Therefore, by linear algebra, the map $$\mathcal I \rho :\; \mathbf{OC}^{\mathsf{imm}}_{\mathsf{moderate \leq 2}}(A)/ \mathbf{OC}^{\mathsf{emb}}_{\mathsf{moderate \leq 2}}(A) \to \Z \times \Z.$$
is surjective. For $A = \R \times [0, 1]$, the map $\mathcal I \rho$ is an epimorphism with respect to the group operation $\uplus$. 
\hfill
\end{proof}

Let us introduce one useful operation/notation. Given an (oriented) immersion $\b: \mathcal C \to A$ (doodles) and a nonnegative number $n$, we denote by $\b(n)$ the new immersion, obtained by surrounding $\b(\mathcal C)$ by $n$ counterclockwise oriented nested simple loops.  If $n$ is negative, $\b(n)$ denotes an immersion in which $\b(\mathcal C)$ is surrounded by $n$ clockwise oriented nested simple loops. When talking about $2$-moderate immersions, we assume that the surrounding simple loops are $2$-moderate as well. Similarly,  given an (oriented) immersion  $\a: X \to A$ (blobs) and a nonnegative $n$,  $\a(n)$ denotes a submersion in which $\a(X)$ is surrounded by $n$ nested (counterclockwise oriented) embedded disks whose boundaries are disjoint and  $2$-moderate. Again, for a negative $n$, we flip the orientations of the disks. \smallskip 

Let us list a few simple properties of these operations. 
We skip their straightforward validation. We consider all the immersions as elements of the appropriate ($2$-moderate) bordism groups, so that the following identities are understood as taking place in these groups: 
\begin{eqnarray} 
\b_1(n) \uplus \b_2(n) = (\b_1 \uplus \b_2)(n), \nonumber\\
(\b(n))(m) = \b(n+m);\\
\a_1(n) \uplus \a_2(n) = (\a_1 \uplus \a_2)(n), \nonumber\\
(\a(n))(m) = \a(n+m).
\end{eqnarray}

We will use the operations $\odot_n: \b \leadsto \b(n)$ or $\odot_n:  \a \leadsto \a(n)$, in combination with the addition $\uplus$, as main tools for generating new examples of immersions. In particular, we will apply $\odot_n$ to the immersed $2$-moderate figures ``$\infty$" and ``$8$", to get families of immersions ``$\infty(n)$" and ``$8(n)$"  
(see Fig. \ref{fig.4.8888XX}, {\sf (a)},  which depicts the immersion $\infty(2) \uplus \overline\infty$). We distinguish between ``$\infty$" and ``$8$" by their position in relation to the vector field $\hat v$, so that ``$\infty$" stands for the crossing of type {\sf I, III}, while ``$8$" for the crossing of type {\sf II, IV}.
 
\begin{figure}[ht]
\centerline{\includegraphics[height=3in,width=5in]{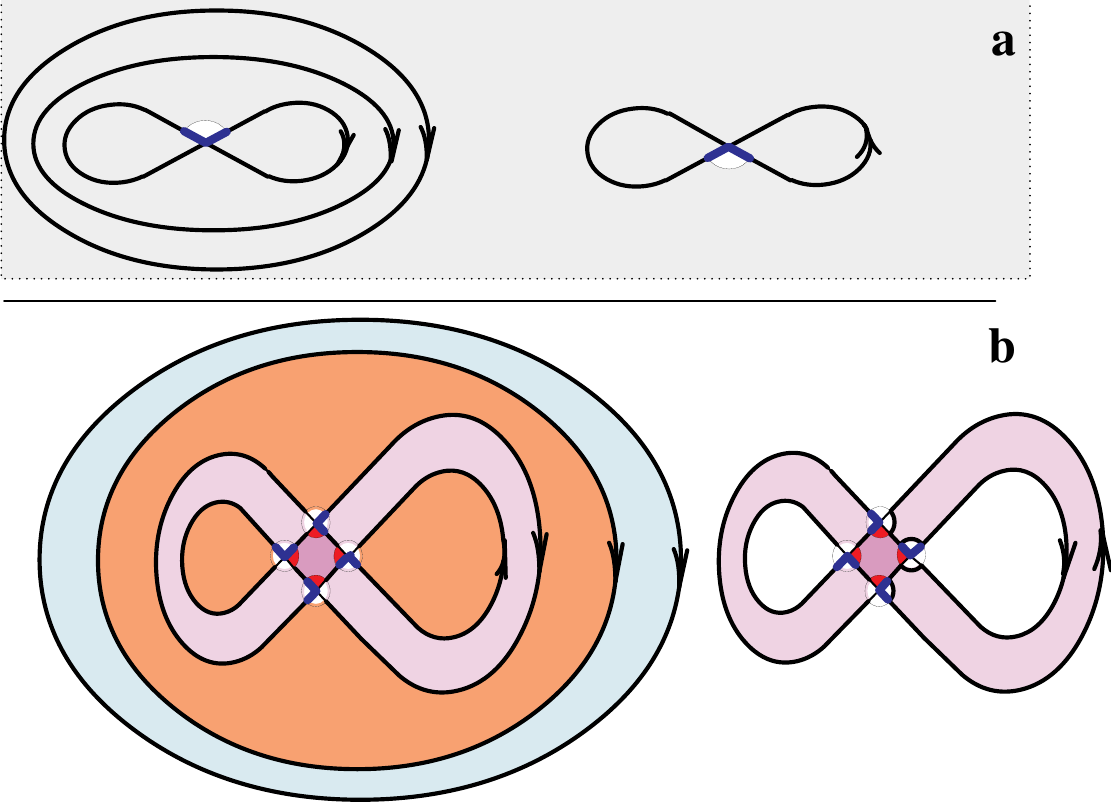}}
\bigskip
\caption{\small{{\sf(a)} Oriented doodles $\infty(2) \uplus \overline\infty$; the two crossings are of the types {\sf I} and {\sf III}.\;  {\sf (b)} Oriented blobs $\infty\text{\sf -ribbon}(2) \uplus \overline\infty\text{\sf -ribbon}$; the four crossings on the left are of the types (\bf{III},\,\sf{II}), (\bf{IV},\,\sf{III}), (\bf{I},\,\sf{IV}),  (\bf{II},\,\sf{I}),  and on the right of the types (\bf{III},\,\sf{IV}), (\bf{IV},\,\sf{I}), (\bf{I},\,\sf{II}),  (\bf{II},\,\sf{III}). The vector field $\hat v$ is vertical.}}
\label{fig.4.8888XX}
\end{figure}
\smallskip

We can improve the claim of Proposition \ref{prop.1.3.1_XY} by analyzing  algebraically and geometrically the kernel of the epimorphism $\mathcal I \rho$. 
With this goal in mind, we need to take a short detour into combinatorics that mimics the geometric situation we are facing. \smallskip

$\bullet$ Let $\Theta$ be a finite set of ``colors". In the near future, $\Theta = \{\mathsf{I, II, III, IV}\}$ or $\Theta = \{\mathsf{I, III}\}$. 

Consider a finite collection of smooth $2$-moderately embedded oriented loops $\mathcal E$ in $A$.  
We take a finite set of points $Q$ in the complement to $\mathcal E$ (soon, $Q$ will be the singular sets $\Sigma_\b$ of immersed doodles $\b(\mathcal C)$). In what follows, we consider the pairs $(Q, \mathcal E)$ up to an ambient isotopy of $A$, or up to an ambient isotopy of $A$ that preserves the foliation $\mathcal F(\hat v)$.

We divide $Q$ in complementary groups $Q_1, \ldots , Q_s$, each group being painted with a different color from the pallet $\Theta$. This coloring is a part of the $Q$-structure. 

The elements of $Q$ are considered in pairs. We impose some restrictions $\mathcal R$  on the {\sf pairings}  of $P: Q \to Q$, where $P$ is a {\sf free involution}.  These restrictions are expressed in terms of colors from the pallet $\Theta$. Some of the restrictions will require that the elements of a particular color may be paired only with other elements the same color, in which case we assume that the cardinality $|Q_i|$ of the corresponding $Q_i$ is even; other restrictions from $\mathcal R$ will pair elements of a particular color only with elements of another color, in which case $|Q_j| =|Q_k|$ for the corresponding sets $Q_j, Q_k$. 

We denote the set of such pairings (i.e., free involutions on $Q$), subject to the restrictions from $\mathcal R$, by $\wp(Q, \mathcal R)$.\smallskip

$\bullet$ In one special case, we will pair elements of $Q_{\mathsf{I}}$ with elements of $Q_{\mathsf{III}}$, elements of $Q_{\mathsf{II}}$ with elements of $Q_{\mathsf{IV}}$. 
Let us denote this special set of rules by $\mathcal R^\odot$. \smallskip


Employing $\mathcal E$, with each pairing $P \in \wp(Q, \mathcal R)$ we associate a function $c(P)$ on  the set of ordered pairs  $\{(q, P(q))\}_{q \in Q}$  with values in $\Z$. The integer $c(q, P) =_{\mathsf{def}}\; c(P)(q, P(q))$ counts the signed transversal intersections of an oriented path $\g(q, P(q)) \subset A$, connecting $q$ to $P(q)$, with the union of oriented loops that form the $1$-cycle $\mathcal E$. For homological reasons, $c(q, P)$ does not depend on the choice of the path $\g(q, P(q))$, provided $A = \R \times [0, 1]$. 

We may view $c(q, P)$ as the linking number of the pair $(q, P(q))$ with the cycle $\mathcal E$.  
Note that $c(q, P)=  -c(P(q), P)$.
\smallskip
This construction gives rise to a map 
\begin{eqnarray} \label{eq.FUNCTION_c}
\mathbf c_{\mathcal E,\, P}: \mathsf{Graph}(P) \to \Z,
\end{eqnarray}  
where  $\mathsf{Graph}(P) \subset Q \times Q$ denotes the graph of the free involution $P: Q \to Q$, subject to the $\mathcal R$-constraints. 

For any $\mathcal R$-admissible pairing $P: Q \to Q$, we always can choose the connecting paths $\{\g(q, P(q))\}_{q \in Q}$ so that they do not intersect each other and are transversal to $\mathcal E$. Without loss of generality, we may also assume that each path $\g(q, P(q))$ is $2$-moderate.


\begin{definition}\label{def.5_PRIMITIVE} $\bullet$ For a fix set of rules $\mathcal R$, we call the triple $\{Q, P, \mathcal E)$  {\sf irreducible} if there exist a set of paths $\{\g(q, P(q))\}_{q \in Q}$ such that all the paths are disjoint and each path $\g(q, P(q))$ has all its intersections with $\mathcal E$ of the same sign. By definition, if  for some $q$, the intersection $\g(q, P(q))\cap \mathcal E$ is empty, the collection $\big\{Q, P, \mathcal E \big\}$ is reducible. \smallskip

$\bullet$ For a fix  set of rules $\mathcal R$, we call the collection $\big\{Q, P, \mathcal E\big\}$ {\sf primitive} if it is irreducible and each colored set $Q_i \subset Q$ consists of exactly two elements, when $P(Q_i) = Q_i$, and of a single element when $P(Q_i) = Q_j$ for a $Q_j \neq Q_i$. \hfill $\diamondsuit$
\end{definition}

It is easy to see that, given two sets $(Q_1, P_1, \mathcal E_1)$ and $(Q_2, P_2, \mathcal E_2)$, their {\sf connected sum} $\uplus$ in $A$ produces a new triple $(Q_1 \uplus Q_2,\, P_1 \uplus P_2,\, \mathcal  E_1 \uplus \mathcal E_2)$ by squeezing  $(Q_1,  \mathcal E_1)$ in $\R \times (0, 0.5)$ and $(Q_2, \mathcal E_2)$ in $\R \times (0.5, 1)$ and applying the pairing  $P_1 \coprod P_2 \in \mathcal R$ to $Q_1 \uplus Q_2$. 
\smallskip

For a fixed set of rules $\mathcal R$, the operation $\uplus$ turns the set $\mathbf D(\mathcal R)$ of all triples  $\big\{Q, P, \mathcal E\big\},$ being considered up to an ambient isotopy of $A$ that preserves $\mathcal F(\hat v)$, into an abelian semi-group.
The associativity and commutativity of the operation $\uplus$ in 
$\mathbf D(\mathcal R)$ can be validated as in Lemma \ref{lem.1.5XY}.
\smallskip

For a given triple $(Q,  \mathcal R, \mathcal E)$, consider the $\ell_1$-norm $\|\mathbf c_{\mathcal E,\, P}\|_{\ell_1}$ of the discrete function $\mathbf c_{\mathcal E,\, P}: \mathsf{Graph}(P)  \to \Z$ in (\ref{eq.FUNCTION_c}). It is independent of the choice of paths $\{\g(q, P(q))\}_{q \in Q}$.
We introduce a non-negative integer  by the formula 
\begin{eqnarray}\label{eq.1_m-NORM}
m(Q,  \mathcal R, \mathcal E) =_{\mathsf{def}}\; \min_{P\, \in\; \wp(Q, \mathcal R)} \Big\{\|\mathbf c_{\mathcal E,\, P}\|_{\ell_1}\Big\}.
\end{eqnarray} 
Evidently,
$$m(Q_1 \uplus Q_2,\, \mathcal R,\, \mathcal E_1 \uplus \mathcal E_2)\, \leq \, m(Q_1,\, \mathcal R,\, \mathcal E_1) + m(Q_2,\,  \mathcal R,\, \mathcal E_2).
$$
Thus, $m(Q, \mathcal R, \mathcal E)$ resembles a semi-norm under the  
connected sum operation $\uplus$. However, in general, $m\big(\uplus_1^k\; \{Q,\, \mathcal R,\, \mathcal E\}\big) \neq k \cdot m(Q, \mathcal R, \mathcal E)$. At the same time, for a primitive $\{Q, \mathcal R, \mathcal E\}$, we have $m\big(\uplus_1^k\; \{Q,\, \mathcal R,\, \mathcal E\}\big) = k \cdot m(Q, \mathcal R, \mathcal E)$. \smallskip

We denote by $\mathbf D_0(\mathcal R)$  the sub-semigroup of $\mathbf D(\mathcal R)$ consisting of the triples $\{Q, \mathcal R, \mathcal E\}$ with the  property $\{m(Q, \mathcal R, \mathcal E) = 0\}$ and consider the quotients 
$\mathbf D(\mathcal R)\big / \mathbf D_0(\mathcal R)$.
\smallskip

 \begin{lemma} \label{lem.D_is_a_group}
$\mathbf D(\mathcal R)\big / \mathbf D_0(\mathcal R)$ is an abelian group with respect to the operation $\uplus$. 
\smallskip

If $\mathcal R$ allows for pairings within each given color,  then any non-zero element of the group $\mathbf D(\mathcal R)\big / \mathbf D_0(\mathcal R)$ is of order $2$.
\end{lemma}

\begin{proof} 
The only property of groups that needs a clarification is the existence of the negation (minus element). 
Let $\tau: A \to A$ be a mirror involution. For any triple $\{Q, P, \mathcal E\}$, we form the triple $\tau(\{Q, \bar P, \mathcal E\}$,
where each pair $(q, P(q))$ is replaced in $\tau(\{Q, \bar P, \mathcal E\})$ by the pair $(\tau(P(q)), \tau(q))$ so that  $\tau(q)$ is given the color of $P(q)$ and $\tau(P(q))$ is given the color of $q$. Consider the new triple $\{Q, P, \mathcal E\} \uplus \tau(\{Q, \bar P, \mathcal E\}$. Now the pairing $(q, \tau(q))$ is permissible. Since the orientations of $\tau(\mathcal E)$ is opposite to the orientations of $\mathcal E$, the path $\g_q$ that connects $q$ and $\tau(q)$ has the property $\g_q \circ (\mathcal E \uplus \tau(\mathcal E)) =0$. Therefore, $\{Q, P, \mathcal E\} \uplus \tau(\{Q, \bar P, \mathcal E\} \in \mathbf D_0(\mathcal R)$.
Hence, $\tau(\{Q, P, \mathcal E\})$ is the negative of $\{Q, P, \mathcal E\}$ in  $\mathbf D(\mathcal R)/\mathbf D_0(\mathcal R)$. This validates the first claim.\smallskip

Consider the triple $\{Q, P, \mathcal E\} \uplus \{Q, P, \mathcal E\}$ and 
a point $q \in Q$. Let $\g(q, s) \subset \R \times [0, 0.4]$ be any smooth path that is transversal to $\mathcal E$ and connects $q$ to a point $s$ whose $\R$-coordinate is bigger than $h$, where the box $[-h, h] \times [0, 0.4]$ contains $Q$ and $\mathcal E$. We think of the second copy of $\{Q, P, \mathcal E\}$ as being obtained by the $0.5$-shift $Sh: \R \times [0, 0.5] \to \R \times [0.5, 1]$ to the right. By the hypotheses about $\mathcal R$, we may pair $q$ and $Sh(q)$.

We can connect $s$ to $Sh(s)$ by a path $\omega$ that resides outside the box $[-h, h] \times [0, 1]$. Then the path $\tau = \g(q, s) \cup \omega \cup \overline{Sh(\g(q, s))}$ (where $\overline{Sh(\g(q, s))}$ denotes the path $Sh(\g(q, s))$ with the reverse orientation) connects $q$ to $Sh(q)$ and has a pair of adjacent intersections with $\mathcal E \cup Sh(\mathcal E)$ of opposite signs. 
Therefore, $\{Q, P, \mathcal E\} \uplus \{Q, P, \mathcal E\} \in \mathbf D_0(\mathcal R)$, and thus $\{Q, P, \mathcal E\}$ is an element of order $2$ in $\mathbf D(\mathcal R)\big / \mathbf D_0(\mathcal R)$.
\hfill
\end{proof}

\begin{lemma}\label{lem.1_invariant_NORM} The correspondence  $\{Q,  \mathcal R, \mathcal E\} \leadsto m(\{Q,  \mathcal R, \mathcal E\})$ defines a norm-like map \hfill \break $m: \mathbf D(\mathcal R)\big / \mathbf D_0(\mathcal R) \to \Z_+$ such that $m(a \uplus b) \leq m(a) + m(b)$ and $m(a) = 0$ implies $a = 0$, where $a, b \in \mathbf D(\mathcal R)\big / \mathbf D_0(\mathcal R)$.  
\end{lemma}

\begin{proof} 
The verification of the claims is on the level of definitions.
\end{proof}

$\bullet$ This ends the combinatorial detour.  \hfill $\diamondsuit$ \smallskip


We are in position to return to bordisms of $2$-moderate immersions of oriented doodles and to state one of our main results.\smallskip


\begin{figure}[ht]
\centerline{\includegraphics[height=3.5in,width=3.5in]{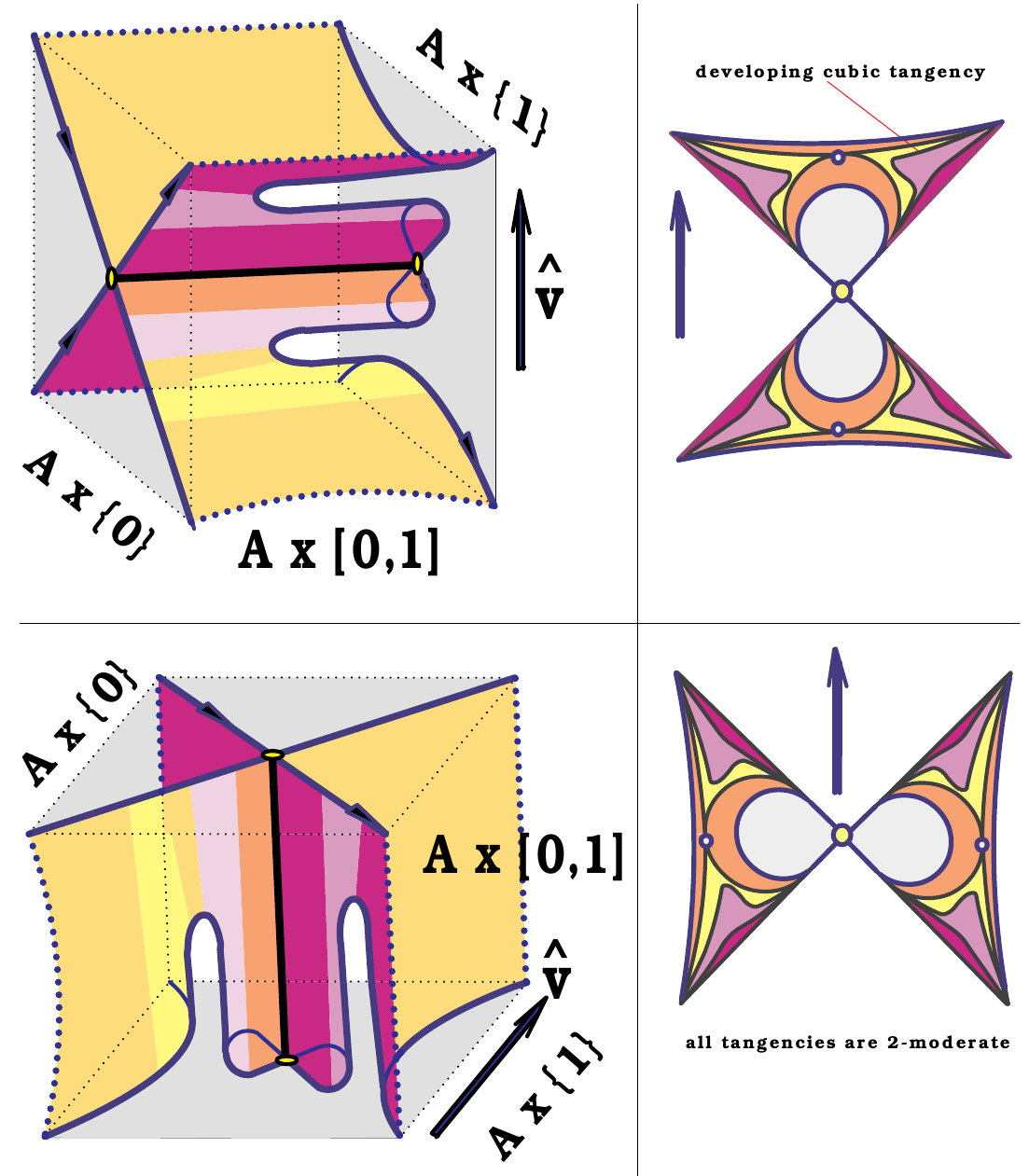}}
\bigskip
\caption{\small{Resolving  crossings of types {\sf I} or {\sf III} into the figures ``$\infty$",  ``$\overline\infty$" and a pair of arcs via a $2$-moderate cobordism (the lower figure). In contrast, resolving  crossings of types {\sf II} or {\sf IV} into the figures ``$8$", ``$\overline 8$" and a pair of arcs is impossible via a $2$-moderate cobordism (the upper figure).}}
\label{fig.4.RESOLUTIONXX}
\end{figure}

\begin{theorem}\label{th.1.DOODLES}  Let $A = \R \times [0, 1]$  
and $\mathcal R =\{\mathsf{I} \Leftrightarrow \mathsf{III}, \mathsf{II} \Leftrightarrow \mathsf{IV}\}$. 

$\bullet$ There is an exact sequence of abelian groups  
$$0 \to  \mathbf K \to \mathbf{OC}^{\mathsf{imm}}_{\mathsf{moderate \leq 2}}(A)\big/\mathbf{OC}^{\mathsf{emb}}_{\mathsf{moderate \leq 2}}(A) \stackrel{\mathcal I \rho}{\longrightarrow} \Z \times \Z \to 0,$$
where the homomorphism $\mathcal I \rho$ is given by the two integral invariants, $\rho^{\mathsf{I}}(\sim) - \rho^{\mathsf{III}}(\sim)$ and \hfill\break $\rho^{\mathsf{II}}(\sim) - \rho^{\mathsf{IV}}(\sim)$. 
\smallskip

$\bullet$ There exists a monomorphism $\Theta: \mathbf D(\mathcal R)/ \mathbf D_0(\mathcal R) \hookrightarrow  \mathbf K$ whose image $\mathbf G$ is spanned over $\Z$ by the doodles   $\{\infty(n) \uplus \overline\infty\}$ and $\{8(n) \uplus \overline 8\}$, where $n \in \Z,\, n \neq 0$. In fact,  $\mathbf G \approx (\Z)^\infty$. \smallskip

$\bullet$ $\mathbf{OC}^{\mathsf{emb}}_{\mathsf{moderate \leq 2}}(A) \approx \Z \times \Z$ via an isomorphism as in (\ref{eq.EMB_ORIENT}).
\end{theorem}


\begin{proof} 
We start with a finite set of oriented loops $\mathcal C$ and their $2$-moderately generic (relatively to $\hat v$) immersion $\b: \mathcal C \to A$. At each point of self-intersection $a \in \b(\mathcal C)$ of types $\mathsf I$ and $\mathsf{III}$, in the vicinity of the crossing, we perform a $2$-moderate surgery on $\b(\mathcal C)$ whose result $\b^\dagger$ is shown in Fig. \ref{fig.4.RESOLUTIONXX}, lower diagram. The resulting oriented curves contain  small figures ``$\infty_a$" and ``$\overline{\infty}_a$", 
whose self-intersection resides at $a$. The bar denotes the flip of the orientation of $\infty_a$, determined by the orientations of the two branches of $\b(\mathcal C)$ at $a$. 
The type of the self-intersection of $\infty_a$ 
in relation to $\hat v(a)$ is the same as the type of the original crossing of $\b(\mathcal C)$ at $a$ with values in $\mathsf{\{I,  III \}}$; so  $\infty_a$ occurs for type $\mathsf I$, and $\overline{\infty}_a$ for type $\mathsf {III}$ crossings. 
The rest of the oriented curves $\mathcal C'$ are immersed in $A$ by a map $\b'$ which has the crossings of the types $\mathsf{\{II,  IV \}}$ only. The immersions $\b$ and $\b^\dagger$ are $2$-moderately cobordant.

By a $2$-surgery on $\b(\mathcal C')$, we could replace the doodles $\mathcal C'$ by  embedded doodles $\mathcal C''$  union several figures ``$8_a$" for the crossings of types $\mathsf{II}$ and several figures ``$\overline{8}_a$" for the crossings of types $\mathsf{IV}$. Let us denote by $\b^\ddagger$ the resulting immersion. Crucially, such a surgery, although being canonical, is not $2$-moderate (see Fig. \ref{fig.4.RESOLUTIONXX})! 
As a result, the immersions $\b$ and $\b^\ddagger$ fail to be $2$-moderately cobordant. This fact complicates our efforts. \smallskip

In any case, via these canonical surgeries, we may replace any given immersion $\b: \mathcal C \to A$ with the  immersion $\b^\ddagger$ which comprises the embedded doodles $\mathcal C'' \subset A$ disjoint union with a number of $2$-moderately immersed figures ``$\infty_a$", ``$\overline{\infty}_a$", and ``$8_a$" , ``$\overline{8}_a$". 

\begin{figure}[ht]
\centerline{\includegraphics[height=2.5in,width=4in]{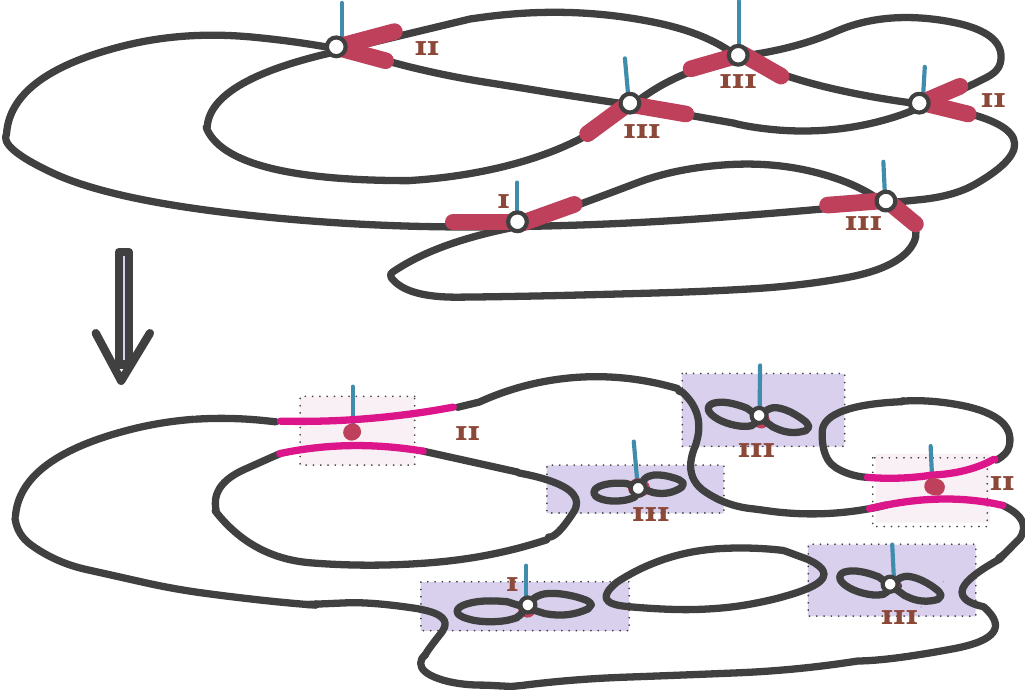}}
\bigskip
\caption{\small{Resolving an oriented immersed doodle into several figures ``$\infty$" (type {\sf I}),  ``$\overline\infty$" (type {\sf III}), and number of $2$-moderate embedded loops. The figure illustrates the homomorphism $\mathsf{res}^{\mathsf{II}, \mathsf{IV}}$ from (\ref{eq.RESOLUTION_II_IV}). The vector field $\hat v$ is vertical.}}
\label{fig.4.RESOLUTION_of_DOODLES}
\end{figure}

To summarize, for types $\mathsf I$ and $\mathsf {III}$, we get the configurations ``$\infty \uparrow_\vee$" and  ``$\overline\infty \uparrow_\wedge$"; for types $\mathsf {II}$ and $\mathsf{IV}$, we either keep the original configurations ``$\times\uparrow_< $" and  ``$\times\;_>\uparrow$" or, depending on the context,  replace them with configurations ``$8 \uparrow_< $" and  ``$8\;_>\uparrow$". Here ``$\uparrow$"  mimics the direction of the vector field $\hat v$ in relation to the oriented figures ``$\infty$ and ``$\times$" and the symbols ``$_\vee,\; _\wedge,\; _<,\; _>$" indicate the preferred sector (determined by the  orientations of the two branches at the crossing $a$) in relation to $\hat v(a)$, i.e., in relation to ``$\uparrow$". \smallskip

If the invariants $\rho_{\mathsf{I}}(\b)- \rho_{\mathsf{III}}(\b) \in \Z,\; \rho_{\mathsf{II}}(\b) - \rho_{\mathsf{IV}}(\b) \in \Z$
vanish, then we can pair figures ``$8\;\uparrow_< $" of type $\mathsf{II}$ with figures ``$8\;_>\uparrow $" of type $\mathsf{IV}$,  figures``$\infty \uparrow_\vee$" of type $\mathsf{I}$  with figures  ``$\infty \uparrow_\wedge$" of type $\mathsf{III}$. 
\smallskip

Consider immersions in the kernel of $\mathcal I\rho$ that consist of several $2$-moderate figures $\infty, \overline\infty, 8, \overline 8$ disjoint union with $2$-moderate embeddings. Thus, the number of figures $\infty$ is equal to the number of figures $\overline\infty$, and the number of figures $8$ is equal to the number of figures $\overline 8$. Let us denote by $\hat{\mathbf G}$ the set of such immersions, being considered up to the $2$-moderate cobordisms. In particular, we focus on the subgroup $\mathbf G \subset \mathbf K$, spanned over $\Z$ by the elements $\{\infty(n) \uplus \overline\infty\}$ and $\{8(n) \uplus \overline 8\}$, where $n \in \Z,\, n \neq 0$. 

First we need to verify that these elements are nontrivial and distinct in $\mathbf K$. Moreover, we will show that they are elements of infinite order in $\mathbf K$.  


Consider a $2$-moderate immersion $B: W \to A \times [0, 1]$ that delivers a cobordism between an immersion $\b: \mathcal C \to A \times \{0\}$ and some embedding $\b_1: \mathcal C_1 \to A \times \{1\}$. Then the  self-intersection curves $\{\ell\}$ of $B(W)$ define a pairing $\{P: a \leadsto P(a)\}_{a \in \Sigma_\b}$ consistent with the set of rules $\mathcal R = \{ \mathsf{I} \Leftrightarrow \mathsf{III}, \mathsf{II} \Leftrightarrow \mathsf{IV}\}$.  

On the other hand, for $\b: \mathcal C \to A$ which represents an element of $\hat{\mathbf G}$, any 
pairing $\{P: a \leadsto P(a)\}_{a \in \Sigma_\b}$ from $\mathcal R$ allows us to attach in $A \times [0, 1)$  $1$-dimensional $2$-moderate {\sf cross-handles} $\infty_a \times I$ and $8_a \times I$ (the immersed images  in $A \times [0, 1)$ of the handles $S^1_a \times I$) to each pair $\infty_a \coprod \overline{\infty}_{P(a)} \subset A \times \{0\}$ or $8_a \coprod \overline{8}_{P(a)} \subset A \times \{0\}$. By a general position argument, these $2$-moderate handles are disjoint.
The union of these cross-handles with the embedded surface  $\mathcal C_1 \times [0, 1]$ delivers an oriented cobordism  $$B:\,  \coprod_a \{S^1_a \times I\} \coprod \{\mathcal C_1 \times [0, 1]\}  \to A \times [0, 1]$$ between $\b$ and $\b_1$.  It is important to attach the cross-handles with some care in order to satisfy the requirement for $B$ to have only $2$-moderate tangencies to the $\hat v^\bullet$-trajectories in $A \times [0, 1]$.   

First, we need to avoid the intersections of the $1$-dimensional singular cores of the immersed surfaces $B(\coprod_a \{S^1_a \times I\})$ with the embedded surfaces $B({\mathcal C_1} \times [0, 1])$ since such intersections would contribute cubic tangencies  at least  to the $\hat v^\bullet$-trajectories. 

It turns out that there are new  combinatorial obstructions to $B$ being $2$-moderate. 
These new obstructions are described in the next lemma. 

\begin{lemma}\label{lem.1_KEY_for_DOODLES} Let $A = \R \times [0, 1]$ and let $\mathcal R =\{\mathsf{I} \Leftrightarrow \mathsf{III}, \mathsf{II} \Leftrightarrow \mathsf{IV}\}$ or $\mathcal R =\{\mathsf{I} \Leftrightarrow \mathsf{III}\}$. Let  $\mathcal C$ be  a finite collection  of oriented circles. Consider a $2$-moderate immersion $\b: \mathcal C \to A$
such that  $\rho_{\mathsf{I}}(\b) = \rho_{\mathsf{III}}(\b)$ and  $\rho_{\mathsf{II}}(\b) = \rho_{\mathsf{IV}}(\b)$. 
 We denote by $\Sigma_\b$ its self-intersection locus of $\b$.
 
 By replacing in $\b(\mathcal C)$  each crossing with the figures ``$\infty_a$", ``$\overline{\infty}_a$" or ``$8_a$" or ``$\overline{8}_a$"  we get a new $2$-moderate immersion $\b^\ddagger$ that is a disjoint union of these figures with a $2$-moderate embedding $\b'': \mathcal C'' \to A$. 
 \smallskip

Then $\b^\ddagger$ is cobordant, via $2$-moderate oriented immersions of surfaces, to an embedding $\b_1$ if and only if $m(\Sigma_\b,  \mathcal R,  {\mathcal C''}) = 0$ (see (\ref{eq.1_m-NORM})), i.e., iff there exists a pairing $P \in \wp (\Sigma_\b, \mathcal R)$ whose function 
$\mathbf c_{\mathcal E,\, P}:  \mathsf{Graph}(P) \to \Z$ is identically zero. 

In the case $\mathcal R =\{\mathsf{I} \Leftrightarrow \mathsf{III}\}$, $\b^\ddagger$ is $2$-moderately cobordant to the original $\b$.
\end{lemma}

\begin{proof} 
We will prove the claim for $\mathcal R =\{\mathsf{I} \Leftrightarrow \mathsf{III}, \mathsf{II} \Leftrightarrow \mathsf{IV}\}$; the arguments in the case $\mathcal R =\{\mathsf{I} \Leftrightarrow \mathsf{III}\}$ are very similar and simpler. As the Fig. \ref{fig.4.RESOLUTIONXX} testifies, for this choice of $\mathcal R$, $\b^\dagger$ is $2$-moderately cobordant to the original $\b$. \smallskip

Let $B: W \to A \times [0, 1]$ be a $2$-moderate immersion that delivers the oriented cobordism between the immersion $\b$ and some embedding $\b_1$.  We have seen that the self-intersection locus $\Xi_B = \coprod \ell$ of $B(W)$ delivers a particular pairing $P_B: \Sigma_\b \to \Sigma_\b$ that is consistent with $\mathcal R$. 

We denote by $I$ a closed interval. We replace the singular surface $B(W)$ with the disjoint union of narrow $2$-moderate cross-handles $$\big\{B^\ddagger_a: S^1_a \times I \to  A \times [0, 1]\big\}_{a \in \Sigma_\b/P_B}$$ whose images are $\infty_a \times I$,  $\overline\infty_a \times I$, and $8_a \times I$,  $\overline 8_a \times I$, together with a number of narrow $2$-moderate cross-tori $$\big\{B^\ddagger_\kappa: S^1_\kappa \times S^1 \to  A \times (0, 1)\big\}_\kappa$$ ($B^\ddagger_\kappa(S^1_\kappa \times S^1)$ are diffeomorphic to  $\infty_\kappa \times S^1$ or to $8_\kappa \times S^1$), and a $2$-moderate regular embedding $B'': W'' \hookrightarrow A \times [0, 1]$ of an oriented surface $W''$. The trace of this cobordism in $A \times \{0\}$ produces the curves $(\coprod_{a \in \Sigma^{\mathsf{I}}_\b} \infty_a) \coprod (\coprod_{a \in \Sigma^{\mathsf{III}}_\b} \overline\infty_a) \coprod(\coprod_{a \in \Sigma^{\mathsf{II}}_\b} 8_a) \coprod (\coprod_{a \in \Sigma^{\mathsf{IV}}_\b} \overline 8_a)$ disjoint union with a \emph{nonsingular} curve $\mathcal C'' = B''(W'') \cap (A \times \{0\})$ as has been described above. \smallskip

Consider now the cores $\{\delta =_{\mathsf{def}} \delta_{a, P(a)}\}_{a \in \Sigma_\b}$ of the cross-handles and of the cross-torii. We notice that $\delta \cap B''(W'')$ must be empty in order to insure that $B$ is $2$-moderate. Indeed, if $x \in \delta \cap B''(W'')$, then $x$ is a triple-intersection point for $B(W)$, which violates the $2$-moderate property of $B$.  \smallskip

Note that the algebraic (i.e., signed) intersection number $\delta_{a, P(a)} \circ B''(W'')$  is a homological invariant for the relative homology class $[\delta_{a, P(a)}] \in H_1(A \times [0, 1], a \cup P_B(a);\, \Z)$ of the path $\delta_{a, P(a)}$ and the relative homology class $$[B''(W'')] \in H_2\big(A \times [0, 1],\;\, [A  \setminus (a \cup P_B(a))] \times \{0\} \; \sqcup \; (A \times \{1\});\, \Z\big)$$
of the relative $2$-cycle $B''(W'')$.
Since the relative homology class $[\delta_{a, P(a)}]$ and the relative homology class 
$$[\g(a, P_B(a))] \in H_1(A \times [0, 1],\; a \cup P_B(a);\; \Z)$$ are equal (the two paths are even homotopic relative to $a \cup P_B(a)$ in $A \times [0, 1]$), we conclude that the algebraic intersection number   $$\delta_{a, P(a)} \circ B''(W'') = \g(a, P_B(a)) \circ  B''(W'') = \g(a, P_B(a)) \circ \mathcal C''.$$
Therefore, if $\g(a, P_B(a))\, \circ \, \mathcal C_1 \neq 0$, we get a contradiction with the requirement $\delta_{a, P(a)} \cap B'(W') = \emptyset$. As a result, if $\g(a, P_B(a)) \circ \mathcal C'' \neq 0$ for some $a \in \Sigma_\b$, then the cobordism $B$ is not $2$-moderate.\smallskip

On the other hand, assume that, for some $\mathcal R$-amenable involution $P: \Sigma_\b \to \Sigma_\b$, the intersection $\g(a, P(a)) \circ \mathcal C'' =  0$  for all $a \in \Sigma_\b$, where $\mathcal C''$ is a $2$-moderate nonsingular component of the replacement $\b^\ddagger$ of $\b$. Then, let us show that $\b$ is $2$-moderately cobordant to an embedding $\b_1$. 

Using the notations from the previous arguments,  
we start with the cobordism $W$ which is the union $H \subset A \times [0, 1]$ of $2$-moderate cross-handles, attached to $A \times \{0\}$ at points of $\Sigma_\b$  according to the coupling $P$, with the non-singular surface $\mathcal C'' \times [0,1]$.   
Since $\hat v^\bullet$ is transversal to the cores of the narrow cross-handles, we may ensure that $H$ is $2$-moderate.  

Although the surfaces $H$ and $S =\mathcal C'' \times [0,1]$ individually are  $2$-moderate, their union $W = H \cup S$ may be not. 

However, if we assume that all the cores $\delta$ of cross-handles from $H$ have the property $\delta \circ S = 0$, then we can correct the failure of $W = H \cup S$ to be $2$-moderate. Indeed, if two adjacent intersections $x, y$  of $\delta$  with the oriented surface $S$ are of opposite signs, then we can attach a $1$-handle $S^1 \times [x, y]$ to $S$ whose core is the segment $[x,y] \subset  \delta$ and whose interior $D^2 \times [x, y]$  engulfs the cross-handles $\infty_a \times [x,y]$, or $\overline\infty_a \times [x,y]$, or $8_a \times [x,y]$, or $\overline8_a \times [x,y]$.  The result is a new embedded oriented surface $S'$ which intersects  $\delta_a$ at fewer  points: the intersections $x$ and $y$ has been eliminated. The rest of the cores $\delta_b$ keep their old intersections with $S$. Continuing this elimination of the adjacent along the cores intersections of opposite signs with evolving nonsingular surfaces, eventually we will get a nonsingular oriented surface $\tilde S$ which is disjoint from all the cross-handles. The  union of $\tilde S$ with all the cross-handles delivers the desired $2$-moderate cobordism between the immersion $\b$ and some embedding $\b_1$. 

By the previous arguments, this recipe for constructing $\b_1$ works when, for some free involution $P: \Sigma_\b \to \Sigma_\b$, we have the property $\g(a, P(a))\, \circ \, \mathcal C'' =  0$ for all $a \in \Sigma_\b$. This means exactly that there exists a pairing $P \in \wp (\Sigma_\b, \mathcal R)$ whose function $\mathbf c_{\mathcal C'',\, P}:  \mathsf{Graph}(P) \to \Z$ is identically zero. In other words, the semi-norm $m(\Sigma_\b,  \mathcal R, \mathcal C'')$ must vanish. \hfill
\end{proof}

We continue with the proof of Theorem \ref{th.1.DOODLES}. As before, in our notations, we suppress  the choices of the paths $\{\g(q, P(q))\}_{q \in Q}$ (equivalently, we assume them to be oriented segments of lines). 

Given an element $\{Q, P, \mathcal E\}$ of the group $\mathbf D(\mathcal R)$, we associate with it a $2$-moderate immersion $\b: \mathcal C \to A$ by replacing each point $q \in Q$ with a small figure $\infty, \overline\infty$, or $8, \overline 8$ (depending on the type $\{\mathsf {I, II, III, IV}\}$ of $q$), and keeping the $2$-moderate embedding $\mathcal E$ for the role of $\mathcal C''$.  By Lemma \ref{lem.1_KEY_for_DOODLES}, $\b$ is  cobordant via a $2$-moderately cobordism $B: W \to A \times [0, 1]$ (whose self-intersection locus $\coprod \ell$ defines $P$)  to an embedding if and only if the element $\{Q, P, \mathcal E\}$ is such that the function $\mathbf c_{\mathcal E,\, P}:  \mathsf{Graph}(P) \to \Z$ is identically zero. Therefore, by the definition of $\mathbf D_0(\mathcal R)$, this construction $\{Q, P, \mathcal E\} \leadsto \b$ defines a momomorphism $\Theta: \mathbf D(\mathcal R)/\mathbf D_0(\mathcal R) \hookrightarrow \mathbf K$. Its image is exactly the subgroup $\hat{\mathbf G}$ (by this group's definition).
\smallskip

Consider the subgroup $\mathbf G \subset \hat{\mathbf G}$ spanned over $\Z$ by the elements $\{\infty(n) \uplus \overline\infty\}$ and $\{8(n) \uplus \overline 8\}$, where $n \in \Z,\, n \neq 0$. 
For any $\mathcal R$-admissible pairing $P$ between the two crossing $a$ and $P(a)$ of each figure,  the intersection number $\g(a, P(a)) \circ C''$ of any path $\g(a, P(a))$ with the concentric loops $C''$ surrounding $\infty, \overline\infty$, or $8, \overline 8$ is equal to $n \neq 0$. Thus, by Lemma \ref{lem.1_KEY_for_DOODLES}, the elements $\{\infty(n) \uplus \overline\infty\}$ and $\{8(n) \uplus \overline 8\}$ are nontrivial in $\mathbf K$ for all $n \neq 0$. Moreover, for the same reason, their $k$-multiple copies (with respect to $\uplus$) are not $2$-moderately cobordant to an embedding. Therefore, the elements $\{\infty(n) \uplus \overline\infty\}$ and $\{8(n) \uplus \overline 8\}$ are of the infinite order in $\mathbf K$. Along the same lines, for $n_1 \neq n_2 \neq 0$, no $k$-multiple of $\{\infty(n_1) \uplus \overline\infty\}$ is $2$-moderately cobordant to $\{\infty(n_2) \uplus \overline\infty\}$. Since $\mathbf K$ is abelian, it follows that $\mathbf G \approx (\Z)^\infty$. \smallskip

In fact, $\hat{\mathbf G} = \mathbf G$. Indeed, consider a $2$-moderate simple loop $\mathcal L$ that contains several figures $\infty, \overline\infty$, or $8, \overline 8$ and does not contain any other simple loops. By a $2$-moderate $1$-surgery (the $1$-handles must avoid the figures $\infty, \overline\infty, 8, \overline 8$) on the domain $\mathcal D$ that bounds $\mathcal L$, we split $\mathcal D$ in several topological disks so that each of them contains a single figure. Let us call temporarily a nested group of simple loops which contain a single figure  $\infty, \overline\infty$, or $8, \overline 8$ a ``{\sf block}". If a block contains oppositely oriented simple loops, they can be eliminated by a $2$-moderate surgery on the annulus that is bounded by them (see Fig. \ref{fig.1.7XX}, the last two slices). Thus, we may assume that all the loops in a block are oriented coherently. 

If a simple loop $\mathcal L'$ bounds a domain $\mathcal D'$ that contains several blocks, then again by a $2$-moderate $1$-surgery on $\mathcal D'$ we can split $\mathcal D'$ in a number of disks, each of which bounds a single new block. Proceeding this way, we transform any configuration of figures and simple loops into $2$-moderately cobordant disjoint union of coherently oriented blocks. This transformation does not affect the original figures $\{\infty, \overline\infty, 8, \overline 8\}$ and their pairings. Therefore, $\hat{\mathbf G} = \mathbf G \approx (\Z)^\infty$.

The last claim of the theorem follows from the isomorphism (\ref{eq.EMB_ORIENT}), since any oriented embedded doodle in $A$ is the boundary of an oriented embedded blob. 
This completes the proof of Theorem \ref{th.1.DOODLES}. \hfill 
\end{proof}

\smallskip

 




Let $\b: \mathcal C \to A$ be a $2$-moderate oriented immersion.
By  canonical resolutions of the types $\mathsf{II}$ and $\mathsf{IV}$ singularities ``$\times\uparrow_< $" and  ``$\times\;_>\uparrow$", we eliminate them in $\b'$ (see Fig. \ref{fig.4.RESOLUTION_of_DOODLES}). The result, $\b^\dagger$, is a new $2$-moderate embedding $\b'': \mathcal C'' \to A$ disjoint union with a number of immersions of the type ``$\infty_a$" and ``$\overline{\infty}_a$". The resolution $\mathsf{res_{II, IV}}: \b \leadsto \b^\dagger$ may change the $2$-moderate bordism invariant $\rho_{\mathsf{II}}(\b) - \rho_{\mathsf{IV}}(\b)$ and thus may change the bordism class of $\b$. 
These canonical resolutions produce a homomorphism
\begin{eqnarray} \label{eq.RESOLUTION_II_IV} 
\mathsf{res}^{\mathsf{II}, \mathsf{IV}}: \mathbf{OC}^{\mathsf{imm}}_{\mathsf{moderate \leq 2}}(A) \to \mathbf{OC}^{\mathsf{imm,\; \mathsf{I} \& \mathsf{III}}}_{\mathsf{moderate \leq 2}}(A),
\end{eqnarray}
where $\mathbf{OC}^{\mathsf{imm,\; \mathsf{I} \& \mathsf{III}}}_{\mathsf{moderate \leq 2}}(A)$ stands for the $2$-moderate bordisms of oriented immersions whose crossings are of the types $\mathsf{I}$ and  $\mathsf{III}$ only. The map $\mathsf{res}^{\mathsf{II}, \mathsf{IV}}$ is evidently the right inverse of the obvious homomorphism $ \mathbf{OC}^{\mathsf{imm,\; \mathsf{I} \& \mathsf{III}}}_{\mathsf{moderate \leq 2}}(A) \to \mathbf{OC}^{\mathsf{imm}}_{\mathsf{moderate \leq 2}}(A)$.\smallskip

The treatment of $\mathbf{OC}^{\mathsf{imm,\; \mathsf{I} \& \mathsf{III}}}_{\mathsf{moderate \leq 2}}(A)$ is very similar to our treatment of $\mathbf{OC}^{\mathsf{imm}}_{\mathsf{moderate \leq 2}}(A)$, although the case of $\mathbf{OC}^{\mathsf{imm,\; \mathsf{I} \& \mathsf{III}}}_{\mathsf{moderate \leq 2}}(A)$ is easier to analyze.

\begin{theorem}\label{th.1_kernel_for_I_III} Let $A = \R \times [0,1]$ and $\mathcal R^\circ =\{\mathsf{I} \Leftrightarrow \mathsf{III}\}$. 

$\bullet$ There exists an exact sequence of abelian groups:
$$0 \to \mathbf K^{\mathsf{I}\, \& \, \mathsf{III}} \to \mathbf{OC}^{\mathsf{imm,\; \mathsf{I}\, \& \, \mathsf{III}}}_{\mathsf{moderate \leq 2}}(A) / \mathbf{OC}^{\mathsf{emb}}_{\mathsf{moderate \leq 2}}(A) \stackrel{\mathcal I \rho^{\mathsf{I}\, \& \, \mathsf{III}}}{\longrightarrow} \Z  \to 0,$$
where the epimorphism $\mathcal I \rho^{\mathsf{I}\, \& \, \mathsf{III}}$ is given by the integral invariant $\rho^{\mathsf{I}}(\sim) - \rho^{\mathsf{III}}(\sim)$. \smallskip

$\bullet$ The kernel  $\mathbf K^{\mathsf{I}\, \& \, \mathsf{III}}$ is generated by the doodles $\{\infty(n) \coprod \overline\infty\}_{n \in \Z;\, n \neq 0}$ and is isomorphic to the group $\mathbf D(\mathcal R^\circ)/ \mathbf D_0(\mathcal R^\circ) \approx (\Z)^\infty$ .  \smallskip

$\bullet$ Thus, 
 any $2$-moderate oriented doodle $\b$ with the crossings of the types $\mathsf{I}$ and $\mathsf{III}$ only and such that $\rho_{\mathsf{I}}(\b) =  \rho_{\mathsf{III}}(\b)$,  is $2$-moderately cobordant to an embedding if and only if its image $\Theta(\b) \in \mathbf D_0(\mathcal R^\circ)$ (i.e., the ``semi-norm" $m(\Sigma_\b,  \mathcal R^\circ,  {\mathcal C'}) = 0$).
\end{theorem}

\begin{proof} 
Since the oriented figure ``$\infty$" has the crossing of type $\mathsf I$, and ``$\overline\infty$" has the crossing of type $\mathsf {III}$, the map $\mathcal I \rho^{\mathsf{I}\, \& \, \mathsf{III}} = \rho_{\mathsf{I}}(\sim) - \rho_{\mathsf{III}}(\sim)$ is an epimorphism. 
\smallskip 

 By Lemma \ref{lem.1_KEY_for_DOODLES}, any $2$-moderate immersion $\b: \mathcal C \to A$ with the crossings of the types $\mathsf{I}$ and $\mathsf{III}$ only and  such that $\rho_{\mathsf{I}}(\b) =  \rho_{\mathsf{III}}(\b)$ is cobordant in $\mathbf{OC}^{\mathsf{imm,\;\mathsf{I}\, \& \, \mathsf{III}}}_{\mathsf{moderate \leq 2}}(A)$ to an  embedding $\b_1$ if and only if the singular set $\Sigma_\b$ admits a free involution $P \in \wp(\Sigma_\b, \mathcal R^\circ)$ with the following property: for any $a \in \Sigma_\b$, the algebraic intersection number $\g(a, P(a)) \circ \mathcal C' = 0$. Here the embedded doodle $ \mathcal C'$ is the doodle $ \mathcal C''$ from the proof of Theorem \ref{th.1.DOODLES}. Equivalently, $m(\Sigma_\b,  \mathcal R^\circ,  {\mathcal C'}) = 0$. 

For any triple $\{Q, P, \mathcal E\}$, representing an element of $\mathbf D(\mathcal R^\circ)$, we construct a $2$-moderate immersion $\b(\{Q, P, \mathcal E\})$ by placing a small figure $\infty$ or $\overline\infty$ at each point $q \in Q$  (depending on whether $q$ is of color $\mathsf I$ or $\mathsf{III}$) and using the embedded blob $\mathcal E$ as a part of the immersion. The role of the involution $P: Q \to Q$ is tenuous: it makes sure that the colors of $q$ and $P(q)$ are different. By Lemma \ref{lem.1_KEY_for_DOODLES}, $\b(\{Q, P, \mathcal E\})$ is cobordant to an embedding if and only if there is a triple $\{Q, \tilde P, \mathcal E\}$ for which $m(\{Q, \tilde P, \mathcal E\}) = 0$; in other words, when $\{Q, \tilde P, \mathcal E\} \in \mathbf D_0(\mathcal R^\circ)$.

Therefore, the correspondence $\{Q, P, \mathcal E\} \leadsto \b(\{Q, P, \mathcal E\})$ gives rise to a monomorphism $$\Theta:  \mathbf D(\mathcal R^\circ)/\mathbf D_0(\mathcal R^\circ)\to \mathbf K^{\mathsf{I}\, \& \, \mathsf{III}}.$$

On the other hand, we have seen that, for any $2$-moderate immersion $\b \in \mathbf K^{\mathsf{I}\, \& \, \mathsf{III}}$, the immersion $\b^\dagger$ (see Fig. \ref{fig.4.RESOLUTION_of_DOODLES}) from Lemma \ref{lem.1_KEY_for_DOODLES} is $2$-moderately cobordant to $\b$. Evidently,  $\b^\dagger$ is in the image of $\Theta$. As a result, $\Theta$ is an epimorphism, and thus, an isomorphism.
\smallskip

Exactly as in the proof of Theorem \ref{th.1.DOODLES} and with the help of Lemma \ref{lem.1_KEY_for_DOODLES}, we verify 
 that all the oriented doodles $\{\infty(n) \coprod \overline\infty\}_{n \in \Z}$ are distinct elements of infinite order in $\mathbf K^{\mathsf{I}\, \& \, \mathsf{III}}$ and generate a subgroup $(\Z)^\infty$. 
\hfill
\end{proof}


Let ``$\infty$-{\sf ribbon}" denotes the immersion of a band $S^1 \times I$ in $A$, shaped as a fat figure $\infty$ in relation to $\hat v$ (see Fig. \ref{fig.4.8888XX}, {\sf(b)}).  The symbol 
``$\infty$-{\sf ribbon}$(n)$" stands for the $\infty$-{\sf ribbon} surrounded by $n$ concentric nested disks. 
Similarly, the ``$8$-{\sf ribbon}" is the immersion of a band $S^1 \times I$ in $A$, shaped as a fat figure $8$ in relation to $\hat v$ and $8$-{\sf ribbon}$(n)$ denotes the $8$-{\sf ribbon} surrounded by $n$ concentric nested disks. \smallskip

In the case of non-oriented blobs, we get the following result:

\begin{theorem}\label{th.1_BLOBS_KERNEL} Let $A = \R \times [0, 1]$.

$\bullet$ There is an exact sequence of abelian groups: 
$$0 \to \mathbf M \to \mathbf B^{\mathsf{imm}}_{\mathsf{moderate \leq 2}}(A) \big/ \mathbf B^{\mathsf{emb}}_{\mathsf{moderate \leq 2}}(A) \stackrel{\mathcal I \rho}{\longrightarrow} \Z_2 \times \Z \to 0,$$
where $\mathbf B^{\mathsf{emb}}_{\mathsf{moderate \leq 2}}(A) \stackrel{J}{\approx} \Z$ is a direct summand of $\mathbf B^{\mathsf{imm}}_{\mathsf{moderate \leq 2}}(A)$. \smallskip

$\bullet$ The homomorphism $\mathcal I \rho$ is defined by the invariants $$\rho_{\mathbf{I}}(\sim), \rho_{\mathbf{III}}(\sim) \in  \Z_2,\;  \rho_{\mathbf{II}}(\sim) - \rho_{\mathbf{IV}}(\sim) \in  \Z$$ (which realize the subgroup $\Z_2 \times \Z \hookrightarrow \Z_2 \times \Z_2 \times \Z$).\smallskip 

$\bullet$ The kernel $\mathbf M$ contains the group $(\Z_2)^\infty$ generated by blobs $\{\infty$-{\sf ribbon}$(n) \uplus \overline\infty$-{\sf ribbon}$\}_{n \in \Z}$ ($n\neq 0$) as in Fig. \ref{fig.4.8888XX}, {\sf(b)}.\smallskip
\end{theorem}

\begin{proof} Thanks to Proposition \ref{prop.1.EXACT_SEQ_BLOBS}, to prove the  theorem we need only to show that the kernel $\mathbf M$ contains the group $(\Z_2)^\infty$. It suffices to exhibit infinitely many distinct elements of $\mathbf M$ of order $2$. 

Consider the immersion $\b_\bullet: X_\bullet \to A$ of a band $X_\bullet = S^1 \times I$, whose image is shaped as $\infty$-{\sf ribbon} (see Fig. \ref{fig.4.8888XX}, {\sf (b)}). 
With the help of $\b_\bullet$, the orientation of $X_\bullet$ (and thus of $\d X_\bullet$) is induced by the preferred orientation of $A$. The image $\b_\bullet(\d X_\bullet)$ has four self-intersections of distinct colors $\{\mathbf{I, II, III, IV}\}$ as blobs (see Fig. \ref{fig.1.9XX}) and of distinct colors $\{\mathsf{I, II, III, IV}\}$ (as doodles). Alternatively, we could use the immersion $\b_\bullet: X_\bullet = D^2 \coprod D^2 \to A$ whose image has also four self-intersections of distinct colors $\mathbf{I, II, III, IV}$  and of distinct colors $\mathsf{I, II, III, IV}$ (see Fig. \ref{fig.1.TABLE_XX}). 

We add to $\b_\bullet$ the regular embeddings of $n$ coherently oriented concentric disks so that the image of the smallest disk in $A$ contains properly the  $\infty$-ribbon $\b_\bullet(X_\bullet)$. 

We claim that $\infty$-{\sf ribbon}$(n)$ is a nontrivial element of $\mathbf B^{\mathsf{imm}}_{\mathsf{moderate \leq 2}}(A)$.  Assume to the contrary that the $\infty$-{\sf ribbon}$(n)$ is the boundary of a solid $2$-moderate cobordism $B: W \to A \times [0, 1)$. Then the self-intersection curves $\ell \subset B(\delta W)$ cannot connect the unique self-intersection point of $\b_\bullet(\d X_\bullet)$ of type $\mathbf I$ to itself. Thus, no pairing consistent with $$\mathcal R^\bullet=_{\mathsf{def}} \; \{\mathbf{I} \Leftrightarrow \mathbf{I},\,  \mathbf{III} \Leftrightarrow \mathbf{III},\, \mathbf{II} \Leftrightarrow \mathbf{IV} \}$$ is available, and no such $2$-moderate cobordism $B$ exists. 

Moreover, for $n \neq m$, $\infty$-{\sf ribbon}$(n)\neq \infty$-{\sf ribbon}$(m)$  in $\mathbf B^{\mathsf{imm}}_{\mathsf{moderate \leq 2}}(A)/ \mathbf B^{\mathsf{emb}}_{\mathsf{moderate \leq 2}}(A)$. 
Indeed, if
\begin{eqnarray}\label{eq.RIBONS_are_distinct} 
\infty\text{\sf{-ribbon}}(n)=  \infty\text{\sf{-ribbon}}(m) \uplus \a' \text{\, in \,} \mathbf B^{\mathsf{imm}}_{\mathsf{moderate \leq 2}}(A),
\end{eqnarray} 
where $\a': X' \hookrightarrow A$ is a $2$-moderate embedding whose image is separated from the rest of the blobs by a $\hat v$-trajectory,
 then the $2$-moderate solid cobordism $B: W \to A \times [0, 1)$ between the two blobs in the RHS and LHS of equality (\ref{eq.RIBONS_are_distinct}) is available. 
 \smallskip
 
Consider the self-intersection curve $\ell \subset B(\delta W)$ that connects the unique self-intersection point $a$ of type $\mathbf I$ in $\infty\text{\sf{-ribbon}}(n)$  to the unique self-intersection point $b$ of type $\mathbf I$ in $\infty\text{\sf{-ribbon}}(m)$. Since $B$ is an immersion, the preimage $B^{-1}(\ell)$ must be union of several arcs, so that  $B: B^{-1}(\ell) \to \ell$ is an immersion.  
By the construction of $\ell$, two of these arcs, $\ell_1$ and $\ell_2$,  belong to the boundary $\delta W$, while the rest of the arcs start at $n$ points in the interior of $W$ and terminate at $m$ points in the interior of $W$. For $m \neq n$, at least one of these arcs must hit $\delta W$ at some point $c$. Therefore, the point $B(c) \in \ell$ belongs to triple intersection locus of $B(\delta W)$, a contradiction with the assumption that $B: \delta W \to A \times [0, 1]$ is $2$-moderate.
 \smallskip

Note that  $\infty$-{\sf ribbon}$(n)$ does not belong to the kernel of $\mathcal I\rho$ since its $\rho_{\mathsf{I}} =1$. At the same time, all the $\mathcal I\rho$-invariants of  $\infty$-{\sf ribbon}$(n) \uplus \overline\infty$-{\sf ribbon} do vanish; thus, it belongs to the kernel $\mathcal I\rho$. 
 
We recycle the previous arguments and examine the $\mathcal R^\bullet$-permissible pairings between the four self-intersections of type $\mathbf I$ of the ribbons $\infty\text{\sf{-ribbon}}(n) \uplus \infty\text{\sf{-ribbon}}$ and $\infty\text{\sf{-ribbon}}(m) \uplus \infty\text{\sf{-ribbon}}$. These pairings are delivered by a pair of curves $\ell \subset B(\delta W)$ and $\ell' \subset B(\delta W)$ that belong to the self-intersection locus of $B(\delta W)$.  Thus, for $m \neq n$,  $B(\delta W)$ must develop triple intersections in $A \times [0,1]$. We conclude that
$$\infty\text{\sf{-ribbon}}(n) \uplus \infty\text{\sf{-ribbon}} \neq \infty\text{\sf{-ribbon}}(m) \uplus \infty\text{\sf{-ribbon}} \mod \mathbf B^{\mathsf{emb}}_{\mathsf{moderate \leq 2}}(A).$$

On the other hand, if the mirror image (with respect to a $\hat v$-trajectory) of a non-oriented blob $\a: X \to A$ is isotopic to the original blob $\a$, then $\a \uplus \a =0$ in $\mathbf B^{\mathsf{imm}}_{\mathsf{moderate \leq 2}}(A)$. Here the isotopy is assumed to preserve the oriented foliation $\mathcal F(\hat v)$. 
Indeed, in such symmetric case, $\a \uplus \a$ is the boundary of the $2$-moderate cobordism $B: W \to A \times [0, 1)$, traced by a rotation in $A \times [0, 1)$ of $\a(X) \subset \R \times [0, 0.5] \subset A$ on the angles $0 \leq \phi \leq \pi$ around the axis $\g = \R \times \{0.5\} \subset A \times \{0\} \subset A \times [0, 1)$.
Thus, 
$$\big(\infty\text{\sf{-ribbon}}(n) \uplus \infty\text{\sf{-ribbon}}\big) \uplus \big(\infty\text{\sf{-ribbon}}(n) \uplus \infty\text{\sf{-ribbon}}\big) = 0$$ in $\mathbf B^{\mathsf{imm}}_{\mathsf{moderate \leq 2}}(A)$ and hence in $\mathbf M$.
Therefore, $\{\infty\text{\sf{-ribbon}}(n) \uplus  \infty\text{\sf{-ribbon}}\}_n$ deliver distinct elements of order $2$ in $\mathbf M$. 
\smallskip


Thanks to the homomorphism $\mathcal R_J: \mathbf B^{\mathsf{imm}}_{\mathsf{moderate \leq 2}}(A) \to \mathbf B^{\mathsf{emb}}_{\mathsf{moderate \leq 2}}(A)$ from (\ref{eq.1.7_XYZ}), \hfill\break $\mathbf B^{\mathsf{emb}}_{\mathsf{moderate \leq 2}}(A) \stackrel{J}{\approx} \Z$ is a direct summand of $\mathbf B^{\mathsf{imm}}_{\mathsf{moderate \leq 2}}(A)$.
\end{proof}

\noindent {\bf Remark 3.1}
We do not know whether $\mathbf M \approx (\Z_2)^\infty$ and we do not have geometric models for all the generators of $\mathbf M$, beyond the ones that have been constructed above and the similar ones from (\ref{eq.GENERATING_blobs}), based on Fig. \ref{fig.1.TABLE_XX}. The main difficulty is that we do not have \emph{localized} resolutions of singularities, present in the boundaries of immersed blobs, into canonical singular (like a small $\infty$-{\sf ribbon}) and nonsingular parts. This prevents us from establishing an analogue of Lemma \ref{lem.1_KEY_for_DOODLES} for blobs.
\hfill $\diamondsuit$
\smallskip

 In the case of \emph{oriented} blobs, many new invariants counting self-intersections of their boundaries arise. In principle, there are $16$ ways to combine the orientation independent types $\mathbf{I - IV}$ with the orientation dependent types $\mathsf{I - IV}$. As before, the curves $\ell$ that belong to the self-intersections of the $\delta$-boundaries of solid cobordisms $W$, define $8$ parings $\pi$ from the set $\mathcal R^{\bullet\bullet}$ between these $16$ types:
 \begin{eqnarray} \label{eq.16_8_PARING}
(\mathbf I, \mathsf I) \Leftrightarrow (\mathbf I, \mathsf{III}),\; (\mathbf I, \mathsf{II}) \Leftrightarrow (\mathbf I, \mathsf{IV}),\; \; (\mathbf{II}, \mathsf I) \Leftrightarrow (\mathbf{IV}, \mathsf{III}),\; (\mathbf{II}, \mathsf{II}) \Leftrightarrow (\mathbf{IV}, \mathsf{IV}), \\
\quad \quad (\mathbf{III}, \mathsf I) \Leftrightarrow (\mathbf{III}, \mathsf{III}),\;\; (\mathbf {III}, \mathsf{II}) \Leftrightarrow (\mathbf{III}, \mathsf{IV}),\;\; (\mathbf{IV}, \mathsf I) \Leftrightarrow (\mathbf{II}, \mathsf{III}),\;\; (\mathbf {IV}, \mathsf{II}) \Leftrightarrow (\mathbf{II}, \mathsf{IV}). \nonumber
\end{eqnarray}
These pairings from $\mathcal R^{\bullet\bullet}$ follow two rules: {\sf 1)} start with a bold Roman number, indexing a sector that belongs to a pair of intersecting blobs at a point $a$ where their boundaries intersect (see Fig. \ref{fig.1.9XX}); subject the sector to a reflection with respect to the vector $\hat v(a)$, then the new sector accquires new bold Roman number; {\sf 2)} the second Roman number that indexes the orientations of two intersecting boundaries, is changed to the Roman number that indexes the two flipped orientations of the boundary curves.

Consider two subsets, $\mathbf A$ and $\mathbf B$, of the $16$ patterns, each subset containing $8$ patterns: $\mathbf A$ consists of patterns whose second index is $\mathsf I$ or  $\mathsf{II}$, while $\mathbf B$ consists of patterns whose second index is $\mathsf{III}$ or $\mathsf{IV}$. We notice that the pairings from (\ref{eq.16_8_PARING}) couple  only elements from $\mathbf A$ with elements from $\mathbf B$.

 
For a given oriented blob $\a:  X \to A$, counting the crossings of each of the $16$ types $\kappa$ from the list (\ref{eq.16_8_PARING}), produces a number $\rho_\kappa$. Each of the $8$ pairings $\pi \in \mathcal R^{\bullet\bullet}$ from the list (\ref{eq.16_8_PARING}), with the help of the self-intersection curves $\ell \subset \delta W$, generates an oriented cobordism invariant $\rho_\kappa - \rho_{\pi(\kappa)} \in \Z$.
Together, $\{\rho_\kappa - \rho_{\pi(\kappa)}\}_\kappa$ produce a homomorphism $$\mathcal I\rho^{\bullet\bullet}: \mathbf{OB}^{\mathsf{imm}}_{\mathsf{moderate \leq 2}}(A) \longrightarrow (\Z)^8.$$ 
 
 Examining the pairings $\pi$ from the list (\ref{eq.16_8_PARING}) between different types of intersection patterns, we notice that $\pi$ couples crossings of two \emph{similarly oriented} (clockwise or counterclockwise) blobs with crossings of two \emph{similarly oriented}  blobs, and couples crossings of two \emph{oppositely oriented} blobs with crossings of two \emph{oppositely oriented} blobs. Therefore, $\mathcal I\rho^{\bullet\bullet}$ is the direct product of two homomorphisms $\mathcal I\rho^\bullet$ and $\overline{\mathcal I\rho^\bullet}$. 
 
 %
 
 \begin{theorem}\label{th.1_BLOBS_KERNEL_ORIENT} For $A = \R \times [0, 1]$, there is an exact sequence of abelian groups: 
$$0 \to \mathbf{OM} \to \mathbf{OB}^{\mathsf{imm}}_{\mathsf{moderate \leq 2}}(A) \big/ \mathbf{OB}^{\mathsf{emb}}_{\mathsf{moderate \leq 2}}(A) \stackrel{\mathcal I \rho^\bullet \, \times\, \overline{\mathcal I \rho^\bullet}}{\longrightarrow} (\Z)^4 \times  (\Z)^4 \stackrel{\Sigma \times \Sigma}{\longrightarrow} \Z_2 \times \Z_2 \to 0,$$
where $\mathbf{OB}^{\mathsf{emb}}_{\mathsf{moderate \leq 2}}(A) \stackrel{\mathsf K \times \mathsf L}{\approx} \Z\times \Z$ is a direct summand of $\mathbf{OB}^{\mathsf{imm}}_{\mathsf{moderate \leq 2}}(A)$.  
The homomorphism $\Sigma$ takes each $4$-vector to the sum of its components modulo $2$. 
\smallskip

The kernel $\mathbf {OM}$ contains the subgroup $(\Z)^\infty$, generated by the blobs $\{Y_{(\bar\mu, \nu)}(n, 0)\}$, \hfill\break  $\{Y_{(\bar\mu, \nu)}(0, m)\}$  from (\ref{eq.GENERATING_blobs}) (see Fig. \ref{fig.1.TABLE_XX}) and indexed by the elements of the set 
$$\big\{m, n \in \Z;\, m, n \neq 0;\,\; (\bar\mu, \nu) \in \tilde{\mathbf A} =_{\mathsf{def}} \{\mathbf{I, II}\} \times \{\mathsf{I, II}\}\big\}.$$
\end{theorem}

\begin{proof} Consider two smooth simple loops (say, ellipcii), $\mathcal C_1$ and $\mathcal C_2$, in $A = \R \times [0,1]$ such that they intersect transversally only at a pair of points $a$ and  $b$, where the vectors $\hat v(a)$ and $\hat v(b)$ are transversal to both curves (see Fig. \ref{fig.1.TABLE_XX}). We pick any pair of indexes $(\bar\mu(a), \nu(a))$, where $\bar\mu(a)$ from the list $\{\mathbf{I, II, III, IV}\}$ and $\nu(a)$ from the list $\{\mathsf{I, II, III, IV}\}$. This choice of 
$(\bar\mu(a), \nu(a))$ determines orientations (counterclockwise or clockwise) of the curves $\mathcal C_1$ and $\mathcal C_2$ as well as one of the four possible choices of the domains $\mathcal D_1$ and $\mathcal D_2$ in $A$ that bound $\mathcal C_1$ and $\mathcal C_2$, respectivelly. Only one of the four choices will produce a pair of compact domains.  To adress this complication, we trim  the non-compact domains by bounding them with very big concentric circles, disjoint from  $\mathcal C_1 \cup \mathcal C_2$.

With these choices in place, the intersection point $b$ acquired a unique  type $(\bar\mu(b), \nu(b)) \in \{\mathbf{I, II, III, IV}\} \times \{\mathsf{I, II, III, IV}\}$. Thus, with the help of oriented $\mathcal D_1$ and $\mathcal D_2$, we get a permutation map $\Pi$ of the set $\{\mathbf{I, II, III, IV}\} \times \{\mathsf{I, II, III, IV}\}$ to itself. 

\begin{figure}[ht]
\centerline{\includegraphics[height=4in,width=4in]{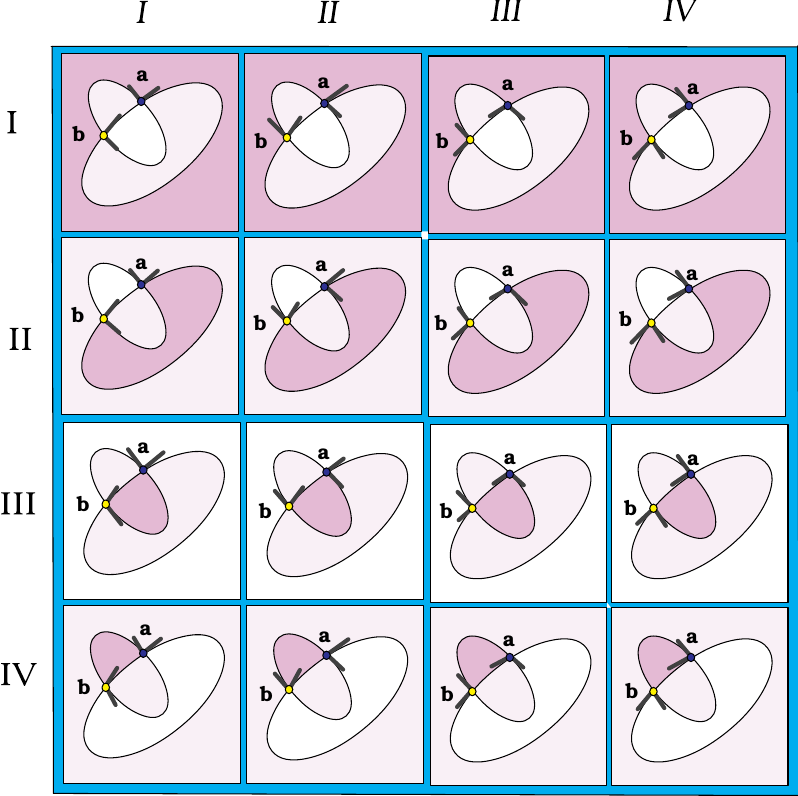}}
\bigskip
\caption{\small{The table lists $16$ possible types at the crossing $a$. Vertically are displayed blobs' types {\bf I, II, III, IV}, horizontally doodles' types {\sf I, II, III, IV}. The vector field $\hat v$ is vertical. The corresponding types at the second crossing $b$ define the pairing $\Pi$. The two bars at each crossing show the orientations of the intersecting curves. The shadings indicate the two choices of interior or exterior for each ellipse.   Any non-compact domain is assumed to be bounded by a big circle, not shown.}}
\label{fig.1.TABLE_XX}
\end{figure}

By a direct inspection, we see that $\Pi$ preserves the two subsets $\mathbf A = \{\mathbf{I, II, III, IV}\} \times \{\mathsf{I, II}\}$ and $\mathbf B = \{\mathbf{I, II, III, IV}\} \times \{\mathsf{III, IV}\}$. Moreover, $\Pi$ is a \emph{free} involution on $\mathbf A$ and on $\mathbf B$. In particular, the involution $\Pi$ generates $4$ pairs-orbits in $\mathbf A$ and $4$ pairs-orbits in $\mathbf B$.

In contrast, the involutions $$\pi:  \{\mathbf{I, II, III, IV}\} \times \{\mathsf{I, II, III, IV}\} \to  \{\mathbf{I, II, III, IV}\} \times \{\mathsf{I, II, III, IV}\}$$ 
from the list $\mathcal R^{\bullet\bullet}$ pair elements of  $\mathbf A$ to elements of $\mathbf B$.  As a result, the invariants \hfill\break $\{\rho_\kappa - \rho_{\pi(\kappa)}\}_{\kappa \in \mathbf A}$ do not vanish for the blobs $\mathcal D_1$ and $\mathcal D_2$ whose crossings are indexed by $(\bar\mu(a),\, \nu(a))$ and $(\bar\mu(b),\, \nu(b))$, both from $\mathbf A$ or $\mathbf B$. Moreover,
$\{\rho_{(\bar\mu(a),\, \nu(a))} = 1 = \rho_{(\bar\mu(b),\, \nu(b))} \}$ for any of the $8$ choices of $(\bar\mu(a),\, \nu(a)) \in \mathbf A$ and for any of the $8$ choices of $(\bar\mu(a),\, \nu(a)) \in \mathbf B$. Therefore, each choice of $(\bar\mu(a),\, \nu(a)) \in \mathbf A$ or $(\bar\mu(a),\, \nu(a)) \in \mathbf B$ produces a vector $\vec z_{(\bar\mu(a),\, \nu(a))} \in (\Z)^{16}$ with exactly two components equal to  $1$ and the rest of components equal to $0$. The pairs of $1$'s in each vector are indexed by either elements of $\mathbf A$ or by elements of $\mathbf B$.  Again, by a direct inspection, the $8\times 16$ matrix $T$, formed by these vectors is of the rank $8$. It consists of four $4\times 8$ blocks $T_{11}, T_{12}, T_{21}, T_{22}$, where $T_{11} = T_{22}$ and $ T_{12} = \mathbf 0 = T_{21}$. In turn, $T_{11}$ consists of two $4\times 4$ blocks, the first of which has $1$'s along the diagonal, and the second one has $1$'s along the anti-diagonal; the rest of entries are zeros. 

Therefore, the image of $\mathbf{OB}^{\mathsf{imm}}_{\mathsf{moderate \leq 2}}(A) \big/ \mathbf{OB}^{\mathsf{emb}}_{\mathsf{moderate \leq 2}}(A) \stackrel{\mathcal I \rho^\bullet \, \times\, \overline{\mathcal I \rho^\bullet}}{\longrightarrow} (\Z)^4 \times  (\Z)^4$ contains the lattice spanned by the $8$ vectors $\{\vec z_{(\bar\mu(a),\, \nu(a))}\}_{(\bar\mu(a),\, \nu(a)) \in \mathbf A \cup \mathbf B}$. 

On the other hand, by Lemma \ref{lem.ODD_self-intersection}, no vector in $(\Z)^4$ whose sum of components is odd (in particular, with a single component $1$) can be realized by an immersion of a compact surface $\a: X \to A$, since the total number of crossings $\sum_\kappa (\rho_\kappa + \rho_{\pi(\kappa)})$  for $\a(\d X)$ must be even, and $\sum_\kappa (\rho_\kappa - \rho_{\pi(\kappa)}) \equiv \sum_\kappa (\rho_\kappa + \rho_{\pi(\kappa)}) \mod 2$.

Thus, the image of $\mathcal I \rho^\bullet \, \times\, \overline{\mathcal I \rho^\bullet}$ is the kernel of the homomorphism $(\Z)^4 \times  (\Z)^4 \stackrel{\Sigma \times \Sigma}{\longrightarrow} \Z_2 \times \Z_2$. \smallskip

For any element $(\bar\mu, \nu) \in \mathbf A$, consider the two blobs $X_{(\bar\mu, \nu)}$ in Fig. \ref{fig.1.TABLE_XX} for which the crossing at $a$ has the type $(\bar\mu, \nu)$. Then the crossing at $b$ has the type $\Pi(\bar\mu, \nu) \in \mathbf A$. 

There are unique $\pi, \pi' \in \mathcal R^{\bullet \bullet}$ (see (\ref{eq.16_8_PARING})) such that  $\pi(\bar\mu, \nu) \in \mathbf B$ and $\pi'(\Pi(\bar\mu, \nu)) \in \mathbf B$. Moreover,  $\pi'(\Pi(\bar\mu, \nu)) = \Pi(\pi(\bar\mu, \nu))$.

We denote by $X_{(\bar\mu, \nu)}(n)$ the blob $X_{(\bar\mu, \nu)}$, surrounded by $n$ big concentric coherently oriented discs. Then the blob 
\begin{eqnarray}\label{eq.GENERATING_blobs}
Y_{(\bar\mu, \nu)}(n, m) =_{\mathsf{def}} X_{(\bar\mu, \nu)}(n) \uplus X_{\pi(\bar\mu, \nu)}(m)
\end{eqnarray}
 is in the kernel of $\mathcal I \rho^\bullet \, \times\, \overline{\mathcal I \rho^\bullet}$. 
Moreover, for $m \neq n$, $Y_{(\bar\mu, \nu)}(n, m)$ is an element of infinite order in $\mathbf{OM}$. The validation of this claim is similar to the one used in the proof of Theorem \ref{th.1_BLOBS_KERNEL}. First, we show that $Y_{(\bar\mu, \nu)}(n, m)$ is non-trivial element in $\mathbf{OM}$. 
Assume to the contrary that there is a solid cobordism $B: W \to A \times [0, 1)$ that bounds $Y_{(\bar\mu, \nu)}(n, m)$. Consider the self-intersection curve $\ell \subset B(\delta W)$ that connects the unique self-intersection point $a$ of type $(\bar\mu, \nu) \in X_{(\bar\mu, \nu)}(n)$ to the unique self-intersection point $a'$ of type $\pi(\bar\mu, \nu)$ in $X_{\pi(\bar\mu, \nu)}(m)$. Since $B$ is an immersion, the preimage $B^{-1}(\ell)$ must be union of several arcs. By the construction of $\ell$, two of these arcs, $\ell_1$ and $\ell_2$,  belong to the boundary $\delta W$, while the rest of the arcs start at $n$ points in the interior of $W$ and terminate at $m$ points in the interior of $W$. For $m \neq n$, at least one of these arcs must hit $\delta W$ at some point $c$. Therefore, the point $B(c) \in \ell$ belongs to triple intersection locus of $B(\delta W)$, a contradiction with the assumption that $B: \delta W \to A \times [0, 1]$ is $2$-moderate.  The demonstration that $Y_{(\bar\mu, \nu)}(n, m)$ is an element of infinite order is similar. 
\hfill
\end{proof}

\begin{conjecture}\label{conj.1.3} For $A = \R \times [0, 1]$, consider the exact sequence  
of abelian groups: 
$$0 \to \mathbf{OM}  \to  \mathbf{OB}^{\mathsf{imm}}_{\mathsf{moderate \leq 2}}(A)/(\Z \times \Z)  \stackrel{\mathcal I \rho^\bullet \, \times\, \overline{\mathcal I \rho^\bullet}}{\longrightarrow} (\Z)^4 \times  (\Z)^4  \stackrel{\Sigma \times \Sigma}{\longrightarrow}  \Z_2 \times \Z_2 \to 0 \quad$$
from Theorem \ref{th.1_BLOBS_KERNEL_ORIENT}. The kernel $\mathbf{OM}$ is $\Z$-generated by the blobs  $Y_{(\bar\mu, \nu)}(n, m)$ from (\ref{eq.GENERATING_blobs}).

For the set of rules $\mathcal R^{\bullet\bullet}$, 
there is an isomorphism 
$$\Theta^{\bullet\bullet}:\, \mathbf D(\mathcal R^{\bullet\bullet})/\mathbf D_0(\mathcal R^{\bullet\bullet}) \hookrightarrow \mathbf{OM}  . $$ 
With the help of  $\Theta^{\bullet\bullet}$, the elements of $\mathbf{OM}$ acquire ``norms" $m(Q,  \mathcal R^{\bullet\bullet}, \mathcal E) \in \Z_+$ as in (\ref{eq.1_m-NORM}).

\hfill $\diamondsuit$
\end{conjecture}

To formulate the next conjecture, we introduce very informally a new topological space $\mathcal G_{\mathsf{moderate \leq 2}}$. Its points are the $2$-moderate functions $f \in \mathcal F_{\mathsf{moderate \leq 2}}$ together with a ``{\sf coupling}" $\tau$ of their zeros.  This $\tau$ may ``link" either two simple zeros of $f$, or a double zero to itself, or a double zero to a pair of distinct simple zeros.  One needs to describe also the deformation rules for $\tau$, as a function $f$ deforms within the space $\mathcal F_{\mathsf{moderate \leq 2}}$. 

The space $\mathcal G_{\mathsf{moderate \leq 2}}$ admits a slightly different interpretation. 
Let $\mathcal Q_{\mathsf{moderate \leq 2}} \subset \mathcal F_{\mathsf{moderate \leq 2}}$ be the subspace of functions with either two simple zeros, or a single double zero, or with no zeros at all. Then any $f \in  \mathcal F_{\mathsf{moderate \leq 2}}$ is a finite product $\prod_i f_i$, where $f_i \in \mathcal Q_{\mathsf{moderate \leq 2}}$ and the order of the multipliers is unimportant. Of course, the presentation $f = \prod_i f_i$ is far from being unique. For example, we always can replace $\prod_i f_i$ by a product $\prod_i h_i f_i$, where $\{h_i\}_i$ have no zeros and $\prod_i h_i   = 1$.  

We would like to think of points of the hypothetical space $\mathcal G_{\mathsf{moderate \leq 2}}$ as the presentations of functions $f \in  \mathcal F_{\mathsf{moderate \leq 2}}$ as such unordered products, being considered up to the equivalence $\prod_i f_i \sim \prod_i h_i f_i$, where the positive functions $\{h_i\}$ are such that $\prod_i h_i   = 1$.

\begin{conjecture} Let $A = \R \times [0,1]$.
There exists a topological space $\mathcal G_{\mathsf{moderate \leq 2}}$ (whose construction is sketched above) and a surjective map $q: \mathcal G_{\mathsf{moderate \leq 2}} \to \mathcal F_{\mathsf{moderate \leq 2}}$, given by $q(f, \tau) = f$, with finite fibers and such that the fundamental group $$\pi_1(\mathcal G_{\mathsf{moderate \leq 2}}, pt) \approx  \mathbf B^{\mathsf{imm}}_{\mathsf{moderate \leq 2}}(A).$$ The isomorphism is delivered by an analogue $Q^{\mathsf{imm}}$ of the map $J^{\mathsf{imm}}$ from Theorem \ref{th.1.7} so that $q \circ Q^{\mathsf{imm}} = J^{\mathsf{imm}}$.
\hfill $\diamondsuit$
\end{conjecture}

 {\it Acknowledgments:} \quad The author is grateful to the referee and to the editor whose suggestions helped to improve the quality of this text and to sharpen the results.

\end{document}